\documentclass{amsart}[11pt]

\usepackage[lmargin=1in,rmargin=1in,tmargin=1in,bmargin=1in]{geometry}
\usepackage{amsmath,amsthm,amsfonts,amssymb,verbatim}
\usepackage{graphicx}
\usepackage{tikz}
\usepackage{hyperref}
\usepackage{xcolor}
\usepackage{booktabs}
\usepackage{colortbl}
\usepackage[font=footnotesize,labelfont=bf]{caption}
\usepackage{mathtools}

\DeclareMathAlphabet{\mathpzc}{OT1}{pzc}{m}{it}

\newcommand{\Z}{\mathbb{Z}}
\newcommand{\Q}{\mathbb{Q}}
\newcommand{\R}{\mathbb{R}}
\newcommand{\F}{\mathbb{F}}
\newcommand{\C}{\mathbb{C}}

\newcommand{\RP}{\R P^1}
\newcommand{\sI}{\mathcal{I}}
\newcommand{\sF}{\mathcal{F}}
\newcommand{\sM}{\mathcal{M}}
\newcommand{\sU}{\mathcal{U}}
\newcommand{\sS}{\mathcal{S}}
\newcommand{\sR}{\mathcal{R}}
\newcommand{\sRhat}{\widehat{\mathcal{R}}}

\DeclareMathOperator{\coker}{coker}
\DeclareMathOperator{\rot}{rot}
\DeclareMathOperator{\wind}{wind}
\DeclareMathOperator{\Mor}{Mor}
\DeclareMathOperator{\rank}{rank}
\DeclareMathOperator{\Tor}{Tor}
\DeclareMathOperator{\Ord}{Ord}

\newcommand{\into}{\hookrightarrow}

\newcommand{\Ainfty}{\mathcal{A}_\infty}

\newcommand{\bchain}{{\bf b}}
\newcommand{\bchainhat}{\widehat{\bf b}}

\newcommand{\tildeT}{\widetilde{T}}
\newcommand{\barT}{\overline{T}}

\newcommand{\gr}{\operatorname{gr}}

\newcommand{\balpha}{{\boldsymbol{\alpha}}}
\newcommand{\bbeta}{{\boldsymbol{\beta}}}
\newcommand{\x}{{\bf x}}
\newcommand{\y}{{\bf y}}

\newcommand{\HFhat}{\widehat{\mathit{HF}}}

\newcommand{\HFminus}{\mathit{HF}^{-}}
\newcommand{\CFminus}{\mathit{CF}^-}

\newcommand{\CFKminus}{\mathit{CFK}^-}

\newcommand{\CFKr}{\mathit{CFK}_{\sR}}
\newcommand{\CFKrhat}{\mathit{CFK}_{\sRhat}}

\newcommand{\HFKminus}{\mathit{HFK}^-}

\newcommand{\HFKhat}{\widehat{\mathit{HFK}}}

\newcommand{\CFKhat}{\widehat{\mathit{CFK}}}

\DeclareMathOperator{\Sym}{Sym}

\newcommand{\spin}{\mathfrak{s}}

\newcommand{\allcurves}{\mathfrak{C}}
\newcommand{\allcurveshat}{\widehat{\mathfrak{C}}}
\newcommand{\complexes}{\mathfrak{D}}
\newcommand{\hatcomplexes}{\widehat{\mathfrak{D}}}

\newtheorem{theorem}{Theorem}

\newtheorem{proposition}[theorem]{Proposition}

\newtheorem{conjecture}[theorem]{Conjecture}

\newtheorem*{namedtheorem}{\theoremname}
\newcommand{\theoremname}{testing}

\theoremstyle{definition}
\newtheorem{definition}[theorem]{Definition}

\newtheorem{remark}[theorem]{Remark}
\newtheorem{example}[theorem]{Example}

\title[A curve-based survey of knot Floer homology]{A curve-based survey of knot Floer homology and concordance invariants}
\author[Jonathan Hanselman]{Jonathan Hanselman}
\address {Department of Mathematics, Indiana University.\newline \it{E-mail address:} \tt{jonahans@iu.edu}}

\begin{document}
\maketitle

\begin{abstract} 
Knot Floer homology associates to a knot in $S^3$ a bigraded chain complex over $\F[W,Z]$, from which many classical and concordance invariants can be extracted. Recent work shows that this algebraic object can equivalently be represented by a decorated immersed multicurve in a marked surface. This survey explains the immersed curve interpretation of knot Floer homology, aided by many examples, and shows how several invariants arising from the knot Floer complex can be extracted from the corresponding immersed curves. The decorated multicurve associated to a knot has a distinguished curve component $\gamma_0$ and a distinguished connected component $\Gamma_0$, both of which are concordance invariants of the knot. We pay particular attention to these components and various numerical concordance invariants that can be extracted from them. We introduce new generalizations of the $V_s$ invariants, and we give a new curve-based description of the $\Upsilon$ invariant by showing that it is determined by generalized $V_s$ invariants. \end{abstract}

\section{Introduction}\label{sec:intro}

Knot Floer homology, introduced by Ozs{\'a}th and Szab{\'o} \cite{OzSz:knots} and independently by Rasmussen \cite{Ras:knot-floer}, is a powerful invariant of knots. In its most algebraic form, the invariant associates to a knot $K$ in $S^3$ a bigraded chain complex $\CFKr(K)$ over the polynomial ring $\sR=\F[W,Z]$, well-defined up to bigraded chain homotopy equivalence. This complex contains a great deal of information about the knot. It categorifies the Alexander polynomial, it detects the Seifert genus and fiberedness, and it gives rise to concordance invariants such as $\tau$, $\epsilon$, $\nu$, $\nu^-$, and $\Upsilon_K$. This machinery has produced numerous applications in knot theory and low-dimensional topology, which we will not attempt to enumerate.

The richness of $\CFKr(K)$ is also a source of difficulty. Many invariants extracted from knot Floer homology are defined by first choosing convenient bases, considering filtered subcomplexes, passing to quotients or localizations, or studying induced maps on homology. These algebraic constructions are often effective, but they can obscure the geometric meaning of the information being used. In recent years, an alternative viewpoint has emerged: the knot Floer complex can be encoded by a decorated immersed curve in a marked surface. Under this correspondence, generators of the complex become intersection points of the immersed curve with a fixed reference curve, terms in the differential become immersed bigons, and many features of the complex become visible as geometric features of the curve.

The goal of this paper is to give a survey of this curve-based approach to knot Floer homology. We aim to make the immersed curve interpretation accessible to researchers who are familiar with Heegaard Floer homology or knot Floer homology but have not worked with immersed curves. We pay particular attention to concordance invariants; indeed, the paper is partly inspired by Hom’s survey of concordance invariants from knot Floer homology \cite{Hom:survey} and most of the invariants discussed there are revisited here, translated into the language of immersed curves whenever possible.

The basic dictionary is as follows. We consider decorated immersed curves in an infinite marked strip $\overline \sS_u$, defined as $[-\frac 1 2, \frac 1 2] \times \R$ with marked points at $(0, n+\frac 1 2)$ for integers $n$. A knot-like complex $C$ over $\sR$ can be represented by a decorated immersed multicurve $(\Gamma,\bchain)$ in $\overline \sS_u$, where $\Gamma$ is a collection of immersed curves and one immersed arc and $\bchain$ is a bounding chain. Loosely speaking, $\Gamma$ records the part of the complex visible after setting $WZ=0$, while the decoration $\bchain$ records additional information needed to recover the full complex over $\sR$ (more accurately, encoding the $WZ=0$ complex may require a small portion of the decoration $\bchain$, which we denote $\bchainhat$, equivalent to equipping the curves with local systems, but we will ignore this subtlety for now). Taking the Lagrangian Floer chain complex (in an appropriate sense) of $(\Gamma,\bchain)$ with a vertical reference line $\mu$ recovers the chain complex $C$. Conversely, starting from a suitable representative of a knot-like complex, one can build the curve by placing one point on $\mu$ for each generator and drawing arcs on either side of $\mu$ corresponding to $W$- and $Z$-terms in the differential. In favorable examples this construction immediately gives the desired curves; in general, one must allow local systems and bounding chains and apply an algorithm to simplify the resulting decorated curve.

One reason the curve picture is useful is that several important invariants become easy to see. The Seifert genus is the maximum height at which the curve meets $\mu$, and fiberedness is detected by the number of intersections at this maximal height. The Alexander polynomial is recovered from algebraic intersection numbers with horizontal segments of $\mu$. The multicurve $\Gamma$ has exactly one arc component, which we denote $\gamma_0$, and the invariants $\tau$, $\epsilon$, $\nu$, and $\phi_i$ can be read from this distinguished component of the curve: $\tau$ is the height at which $\gamma_0$ first intersects $\mu$, $\epsilon$ records the direction in which the curve turns after that initial intersection, and the $\phi_i$ count certain right arcs of specified length in $\gamma_0$. The torsion order $\Ord(K)$ measures the maximum length of a right arc in the full multicurve $\Gamma$. None of these invariants use the bounding chain decoration, as they all are determined by the $WZ=0$ simplification of the complex. More sophisticated invariants, such as $V_i$, $H_i$, $d$-invariants of surgeries, $Y_n$, and $\Upsilon_K$, can be defined from the full complex by considering certain subcomplexes and inclusion maps. In some cases these invariants can also be extracted geometrically from the decorated curves, though in these cases the geometric descriptions are less simple than the invariants mentioned earlier and in practice one often works with a mix of geometric and algebraic techniques.

A second reason the curve picture is useful is that it isolates the concordance-relevant part of the invariant. As noted above the multicurve $\Gamma$ arising from a knot-like complex has a distinguished component $\gamma_0$, the only arc component. When considering the multicurve $\Gamma$ with its decoration $\bchain$, which loosely speaking is a subset of the self-intersection points of $\Gamma$, we say that two curve components are connected if an intersection between them is contained in $\bchain$ and we define $\Gamma_0$ to be the connected component of $(\Gamma, \bchain)$ containing $\gamma_0$. This distinguished connected component corresponds to the irreducible direct summand of the complex that supports its (vertical or horizontal) homology, which is related to the local equivalence class or stable equivalence class of the complex. This distinguished connected component (and thus in particular the distinguished curve $\gamma_0$), is a concordance invariant. The fact that many of the numerical invariants listed above are concordance invariants corresponds to the fact that they only depend on the distinguished connected component $\Gamma_0$.

There are, however, important subtleties. In the $WZ=0$ setting, decorated immersed curves can be put in a clean normal form. As a result, the immersed multicurve $\Gamma$ associated with $\CFKr(K)$ is an invariant of $K$ up to homotopy of curves in the marked strip $\overline \sS_u$. If a decoration $\bchainhat$ is required in the $WZ=0$ setting, this too is uniquely determined by $K$. In the full minus setting over $\sR$, the bounding chain decoration is more delicate; different subsets $\bchain$ of self-intersection points of $\Gamma$ may give equivalent decorated curves, and there is not always a clear preferred representative for $\bchain$. Thus the decorated immersed curve corresponding to $\CFKr(K)$, which we denote $\Gamma(K)$, is an invariant only up to equivalence of decorated curves. The notion of equivalence here comes from the Fukaya category, but does not (as of now) have a nice combinatorial description. Thus when considering a fixed multicurve $\Gamma$ with two decorations $\bchain$ and $\bchain'$, it may not be immediate to determine if $(\Gamma, \bchain)$ and $(\Gamma, \bchain')$ are equivalent. Having said that, in practice this is not an issue in small examples. In many cases the decoration $\bchain$ is uniquely determined by the complex; in fact often it is forced by the multicurve itself (in these cases the full complex over $\sR$ is determined by its $WZ = 0$ simplification). In other cases there are multiple equivalent choices of $\bchain$ but one is obviously simpler and the preferred representative (for instance, one of them may be the trivial decoration that contains no self-intersection points). Each of these scenarios occurs in examples we provide in Section \ref{sec:examples}. At worst, a given multicurve has finitely many valid decorations---usually quite few---and different choices can often be checked for equivalence by ad hoc methods. Nevertheless, it would be satisfying to have a definitive normal form for the decoration $\bchain$---whether this is possible remains an open question.

\subsection{A first example}

Before introducing the general formalism, we indicate the curve construction in a simple case. Consider a knot $K$ whose knot Floer complex admits a basis that is both horizontally and vertically simplified. For each generator of the complex the associated curve $\Gamma(K)$ has one intersection with the vertical line $\mu$ through the marked points. The Alexander grading of a generator determines the height of the corresponding intersection point. A differential term $W^n y$ from a generator $x$ to a generator $y$ is represented by an arc on the right side of $\mu$ connecting the points corresponding to $x$ and $y$, while a term $Z^n y$ is represented by an arc on the left side. These arcs necessarily pass $n$ marked points in each case (as can be seen by considering Alexander gradings). Because our basis is horizontally and vertically simplified, this construction results in an immersed multicurve with one arc component whose ends are on $\mu$---we extend the ends of this arc component through $\mu$ to the boundaries of the strip. The resulting multicurve always represents the $WZ=0$ knot Floer complex of $K$. If we further assume that $\CFKr(K)$ is a staircase complex, as occurs for an L-space knot, or a staircase complex along with $1\times 1$ box summands (e.g. if $K$ is thin), then the curve actually captures the full complex over $\sR$, no bounding chain is needed (note that in these cases there are no diagonal arrows to capture even in the full complex). This simplified construction applies, for instance, to the right-hand trefoil; the corresponding complex and immersed curve are shown in Figure \ref{fig:first-example}(a,b).

\begin{figure}
\includegraphics[scale = 1]{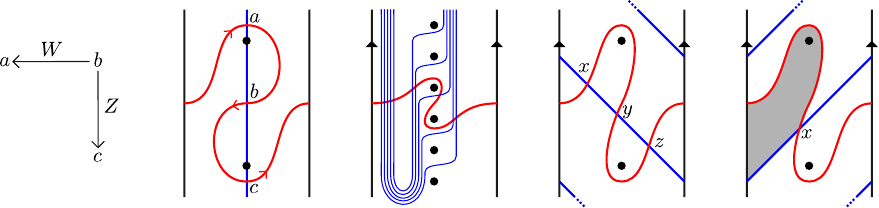}

(a) \hspace{24 mm} (b) \hspace{24 mm} (c) \hspace{24 mm} (d) \hspace{24 mm} (e) 
\caption{(a) The complex $\CFKr(K)$ for the right-hand trefoil; (b) The corresponding immersed curve $\Gamma(K)$; (c) $V_s$ counts marked points enclosed by the bigon between $\Gamma(K)$ and one of the blue curves shown, namely the one crossing the central vertical axis at height $s$; (d) $\HFminus(S^3_{-1}(K))$ is the homology of a complex with three generators and differential coming from the indicated bigons; (e) $\HFminus(S^3_{1}(K))$ has a single generator corresponding to the intersection point $x$, and the grading of this generator (and hence the $d$-invariant of this surgery) is determined by the shaded region. }
\label{fig:first-example}
\end{figure}

From this curve we can read off many invariants for the trefoil. The curve consists of only one component, an arc, which is thus $\gamma_0$. The height of the first intersection of $\gamma_0$ with $\mu$, which is also the maximum height attained by the curve, is 1 (here, and throughout the paper, we use height in a discrete sense: a point on $\mu$ between $(0, n-\frac 1 2)$ and $(0, n+\frac 1 2)$ is said to be at height $n$). It follows that $\tau(K) = g(K) = 1$. The maximum height is attained exactly once so $K$ is fibered. The curve turns downward after it first crosses $\mu$, so $\epsilon(K) = 1$, and $\nu(K) = 1 = \lceil \frac 1 2 \rceil$ because the curve first intersects $\mu$ near height $\frac 1 2$ when the curve is pulled tight. There is a single right arc of length one, oriented downward (when following $\gamma_0$ from left to right), which implies that $\phi_1(K) =1$, $\phi_i(K)=0$ for $i > 1$, and $\Ord(K) = 1$. We can recover the Alexander polynomial $\Delta_K(t) = t - 1 + t^{-1}$, where the coefficient of $t^k$ records the signed intersection of the curve with $\mu$ at height $k$. The invariants $V_s$ in this case record the number of marked points covered by the obvious bigons between $\gamma_0$ and the curves shown in Figure \ref{fig:first-example}(c), where for $V_s$ we use the curve that crosses the central vertical line $\mu$ at height $s$. We see from this picture that for the right-hand trefoil $V_{-s} = s$ for $s > 0$, $V_0 = 1$, and $V_s = 0$ for $s > 0$. The first $s$ for which $V_s = 0$ gives the invariant $\nu^-(K) = 1$. We can consider $(-1)$-surgery on $K$ by intersecting with a line of slope $-1$ instead of $\mu$ (see Figure \ref{fig:first-example}(d), which gives a Floer complex (over the single variable polynomial ring $\F[\sU]$) with three generators $x$, $y$, $z$ and differential $\partial(y) = \sU x + \sU z$ and $\partial(x) = \partial(z) = 0$; it follows that $\HFminus(S^3_{-1}(K)) \cong \F[\sU] \oplus \F$. For $(+1)$-surgery on $K$ we get a complex with a single generator and no differentials (see Figure \ref{fig:first-example}(e)), so $\HFminus(S^3_{+1}(K)) \cong \F[\sU]$.  The $d$-invariant of this surgery, which is a concordance invariant of $K$, is given by the $d$-invariant of the same surgery on the unknot (in this case, $0$) minus twice the number of marked points contained in the shaded region, so that $d( S^3_{1}(K) ) = -2$.

This example already illustrates the main philosophy of the paper. Algebraic data in $\CFKr(K)$ is converted into a curve in a marked surface, and many constructions involving the complex become questions about intersections, bigons, heights, and arcs. The rest of the paper develops this dictionary more fully and explains how to use it to recover both classical knot invariants and concordance invariants.

\subsection{Organization and reading guide}

Sections \ref{sec:CFK}-\ref{sec:correspondence} present the structural background behind the use of immersed curves in knot Floer homology by relating certain immersed curves to certain bigraded complexes.
\begin{itemize}
\item Section \ref{sec:CFK} reviews the algebraic formulation of knot Floer homology and the knot-like complexes that arise; readers familiar with knot Floer theory may skip this section but should refer to it for our notational conventions.
\item Section \ref{sec:curves} reviews the relevant category of decorated immersed curves and states a structure theorem for equivalence classes of these objects. Many details are omitted; even so, this section is the most technical, and readers may wish to skip some topics such as gradings (Section \ref{sec:gradings}) on a first reading. The train tracks discussed in Section \ref{sec:train-tracks} are only necessary to motivate the structure theorems stated in Section \ref{sec:structure-thm} and sketch their proof, but can be skipped if one primarily cares about examples and applications.
\item Section \ref{sec:correspondence} lays out the equivalence between the knot-like complexes defined in Section \ref{sec:CFK} and the knot-like curves defined in Section \ref{sec:curves}.
\end{itemize}
Section \ref{sec:examples} provides many explicit examples of knot-like complexes and corresponding knot-like curves. These examples are an important part of the exposition, and readers may find it helpful to consult them while reading Sections \ref{sec:CFK}-\ref{sec:correspondence}. Readers with some previous exposure to knot Floer immersed curves may wish to skip directly here and refer to Sections \ref{sec:CFK}-\ref{sec:correspondence} as needed.

Sections \ref{sec:simple-invariants}-\ref{sec:concordance-invariants} focus on extracting various knot invariants from immersed curves.

\begin{itemize}
\item Section \ref{sec:simple-invariants} highlights some simple invariants that can be easily extracted from the immersed curve associated to a knot without considering the decoration $\bchain$. These include classical invariants recovered by knot Floer homology such as the genus and the Alexander polynomial, as well as the invariants $\tau$, $\epsilon$, $\nu$, $\phi_i$, and the torsion order $\Ord(K)$. These invariants depend only on the curve and not on the more mysterious bounding chain decoration; as such this section serves as a good warm-up to interpreting immersed curve invariants. 

\item Section \ref{sec:subcomplexes} builds some language needed to define more sophisticated invariants. In particular, we discuss certain subcomplexes and inclusion maps that play a role in many of the invariants extracted from knot-like complexes and interpret them in terms of immersed curves. Section \ref{sec:surgery-formula} is an aside on the Dehn surgery formula, which also makes use of these inclusion maps; this section is not necessary for the rest of the paper (except when discussing the $d$-invariants of surgery in \ref{sec:invariants-from-v_s}). The standard inclusion maps $v_s$ are covered in Section \ref{sec:v_s-maps} and more obscure generalizations of them are introduced in Section \ref{sec:other-inclusion-maps}; while these generalizations are of some interest, we recommend skipping Section \ref{sec:other-inclusion-maps} if it seems intimidating.

\item Section \ref{sec:distinguished-components} defines the distinguished curve components $\gamma_0(K)$ and $\Gamma_0(K)$, which are concordance invariants, and discusses how (certain equivalence classes of) knot-like curves can be given a group structure.
\item Section \ref{sec:concordance-invariants} describes several invariants that can be extracted from the distinguished component $\Gamma_0(K)$, which in particular makes them concordance invariants. In addition to $\tau$, $\epsilon$, $\nu$, $\phi_i$ which were introduced earlier, these include $\nu^-$, $V_i$ and $H_i$, $d$-invariants of surgeries, $Y_n$, and $\Upsilon_K$.
\end{itemize}

\subsection{Topics not covered}

Knot Floer homology has led to a remarkable range of developments in the more than two decades since it was introduced, and even with the focus on the role of immersed curves in knot Floer theory it would be impossible to address all relevant topics in this survey. We briefly mention a few topics we have excluded. First, we emphasize that we restrict to knots and say nothing about links. This is more or less required by the stated focus of the article; while knot Floer homology admits versions for links, there is not (yet) an immersed curve interpretation of these invariants. We also restrict to knots in $S^3$ for simplicity. Knot Floer homology can be defined for knots in arbitrary 3-manifolds, and with a little care the immersed curve perspective extends to knots in a broader class of manifolds, but restricting to $S^3$ simplifies the discussion significantly without losing many of the key ideas, and this is the case of interest for many applications anyway.

We only discuss invariants determined by the knot Floer complex $\CFKr(K)$. There are many knot invariants related to Heegaard Floer homology that are not determined by the knot Floer complex alone. Some of these rely on enhancements of the knot Floer complex, such as invariants from involutive knot Floer homology \cite{HendricksManolescu:involutive, HendricksHom:involutive-note}, while other invariants depend on Heegaard Floer invariants of closed 3-manifolds associated to a knot (but not by surgery), such as the $d$-invariants of branched double covers \cite{ManolescuOwens}. It is an interesting question whether these invariants, particularly those of the first type, admit interpretations using immersed curves; for example, it may be possible to add extra decorations to the curves discussed here to encode the involutive knot Floer complex, but this has not yet been done and we will not consider it here.

Another omission is any discussion of satellite operations. There have been several notable applications of immersed curves to questions relating to satellite operations, owing to the fact that in many cases the immersed curve invariants for certain satellites can be geometrically obtained from the immersed curve invariants for the companion. For example, \cite{HW:cabling} gives an immersed curve interpretation of cabling operations and \cite{Chen2019} and \cite{ChenHanselman} apply immersed curves to satellite operations with $(1,1)$-patterns. Applications of these methods can be found in \cite{Shen2025NoteKnotFloer}, \cite{HomLidmanPark:cabling} and \cite{PatwardhanXiao:mazur}. We do not address satellite operations in this survey to keep the length manageable, and to avoid making any mention of bordered Heegaard Floer homology which often plays a role in satellite applications, but we encourage the reader to explore the references above.

\section{Knot Floer homology}\label{sec:CFK}

Knot Floer homology was introduced by Ozsv{\'a}th and Szab{\'o} in \cite{OzSz:knots}, and separately by Rasmussen in \cite{Ras:knot-floer}; other useful surveys of this material can be found in \cite{Manolescu:intro-to-knot-floer} and \cite{Hom:survey}. In this section we briefly review the structure of knot Floer homology in its algebraic form before turning our attention to immersed curves in Section \ref{sec:curves}. As notation for knot Floer homology has changed over the years, this section also serves to introduce the notation conventions we will use.

\subsection{The knot Floer complex}
Knot Floer homology associates to a knot $K$ in $S^3$ a finitely generated bigraded chain complex $\CFKr(K)$ over the ring $\sR = \F[W,Z]$, where $\F = \F_2$ is the field with two elements and $W$ and $Z$ are two formal variables. Note that we include the $\sR$ subscript in our notation to highlight our modern conventions, which differ from the original definition that only used one formal variable; see Remark \ref{rmk:WZ-notation}. This chain complex is a well-defined invariant of $K$ up to bigraded chain homotopy equivalence. The bigrading $\gr = (\gr_w, \gr_z)$ consists of two integer gradings and has the property that $\gr_w \equiv \gr_z \pmod 2$. Multiplication by $W$ and $Z$ have bidegrees $(-2, 0)$ and $(0, -2)$, respectively, and the differential $\partial$ on $\CFKr(K)$ has bidegree $(-1, -1)$. The grading $\gr_w$ is also called the \emph{Maslov grading} and is also denoted $M$. There is another useful grading $A$, defined by $\frac{\gr_w - \gr_z}{2}$, called the \emph{Alexander grading}. Since $\gr_w$ and $\gr_z$ agree modulo 2, $A$ is an integer grading. It is clear that $A$ is preserved by the differential, so the chain complex splits as a direct sum
$$\CFKr(K) = \bigoplus_{s \in \Z} \CFKr(K; s)$$
where $\CFKr(K; s)$ denotes the subcomplex with Alexander grading $s$. 

The construction of $\CFKr(K)$ makes use of a \emph{doubly pointed Heegaard diagram} representing the knot $K$, which consists of a genus $g$ surface $\Sigma$ decorated by two sets of $g$ disjoint and homologically linearly independent attaching circles $\balpha$ and $\bbeta$ and two basepoints $w$ and $z$ in the complement of $\balpha\cup\bbeta$ such that $(\Sigma, \balpha, \bbeta)$ is a Heegaard diagram for $S^3$ and connecting $w$ and $z$ by an arc in each handlebody that avoids the attaching disks determines the knot $K$. Given such a diagram, we consider the $g$-fold symmetric product $\Sym^g(\Sigma) = \Sigma \times \cdots \times \Sigma/S_g$ and construct half-dimensional subspaces $\mathbb{T}_\balpha $ and $\mathbb{T}_\bbeta$, which are respectively the product of all the curves in $\balpha$ and the product of all the curves in $\bbeta$. The symmetric product $\Sym^g(\Sigma)$ is a symplectic $(2g)$-manifold, and for suitable choice of symplectic structure the subspaces $\mathbb{T}_\balpha $ and $\mathbb{T}_\bbeta$ are Lagrangian \cite{Perutz:HFtori}. Then $\CFKr(K)$ is defined as the Lagrangian Floer chain complex of $\mathbb{T}_\balpha$ and $\mathbb{T}_\bbeta$: generators are intersection points in $\mathbb{T}_\balpha  \cap \mathbb{T}_\bbeta$, and the differential counts certain pseudo-holomorphic disks. The formal variables $W$ and $Z$ record how pseudo-holomorphic disks counted in the differential interact with the basepoints $w$ and $z$, respectively. More precisely, given generators $\x$ and $\y$ we let $\pi_2(\x, \y)$ denote the set of homotopy classes of disks connecting $\x$ to $\y$, and given $\phi \in \pi_2(\x, \y)$ we define $n_w(\phi)$ and $n_z(\phi)$ to be the algebraic intersection number of a representative of $\phi$ with $\{w\} \times \Sym^{g-1}(\Sigma)$ and $\{z\} \times \Sym^{g-1}(\Sigma)$, respectively. Letting $\sM(\phi)$ denote the moduli space of appropriate pseudo-holomorphic representatives of $\phi$ and $\widehat{\sM}(\phi)$ the quotient of $\sM(\phi)$ by its $\R$ action, and noting that $\widehat{\sM}(\phi)$ is compact and zero-dimensional when the Maslov index $\mu(\phi)$ is 1, the differential is defined by
$$\partial \x = \sum_{\y \in \mathbb{T}_\balpha  \cap \mathbb{T}_\bbeta} \sum_{\phi\in\pi_2(\x, \y),  \mu(\phi) = 1} \# \widehat{\sM}(\phi) W^{n_w(\phi)} Z^{n_z(\phi)} \y.$$ 
Any two generators $\x$ and $\y$ are connected by some homotopy class $\phi \in \pi_2(\x, \y)$ (though $\phi$ may not contribute to the differential). The gradings $\gr_w$ and $\gr_z$ are defined, as relative $\Z$-gradings, by the formulas
\begin{equation}\label{eq:relative-grading}
\gr_w(x) - \gr_w(y) = \mu(\phi) - 2n_w(\phi) \qquad \text{ and } \qquad \gr_z(x) - \gr_z(y) = \mu(\phi) - 2n_z(\phi).
\end{equation}
We can enhance $\gr_w$ and $\gr_z$ to absolute gradings with the convention that $\gr_w$ is 0 for the generator of the vertical complex (defined in the next paragraph) and $\gr_z$ is 0 for the generator of the horizontal complex.

It is sometimes convenient to simplify the knot Floer complex by passing to various quotients of the ring $\sR$. One important variation comes from setting the product $WZ$, which we will denote $\sU$, to zero; we let $\sRhat$ denote $\sR / (\sU = 0)$ and let $\CFKrhat(K)$ denote the resulting bigraded complex. Setting $\sU = 0$ corresponds to ignoring disks that cover both basepoints, while disks that cover either basepoint alone are allowed; this simplification may lose some information, but it is easier to compute\footnote{An effective algorithm based on methods from \cite{OzSz:algebras-with-matchings} has been implemented by Szab{\'o} and is now available through SnapPy \cite{SnapPy}, via the function \texttt{Link.knot\_floer\_homology()}.}. Further quotients come from setting $Z=0$, setting $W = 0$, or both; in the last case the resulting chain complex over $\F$ is denoted $\CFKhat(K)$ and its homology is $\HFKhat(K)$. Another simplification to the construction of $\CFKr(K)$ is forgetting a basepoint altogether, which algebraically corresponds to setting the relevant variable to 1. In particular, setting $Z=1$ (and thus $W = \sU$) in $\CFKr(K)$ defines a chain complex over $\F[W] = \F[\sU]$ that we call the \emph{vertical complex}. The vertical complex has a single $\Z$-grading induced by $\gr_w$; $\gr_z$ does not descend to a grading since multiplication by $Z$ acts nontrivially on $\gr_z$. Similarly, setting $W=1$ defines the \emph{horizontal complex} over $\F[Z] = \F[\sU]$ with grading $\gr_z$. Recall that $\CFKr(K)$ is defined using a doubly pointed Heegaard diagram for $K$ and removing either basepoint gives a pointed Heegaard diagram for $S^3$; it follows that both the vertical complex and the horizontal complex are chain homotopic to $\CFminus(S^3)$, and their homologies (called the \emph{vertical homology} and the \emph{horizontal homology}, respectively) are isomorphic to $\HFminus(S^3) \cong \F[\sU]$. In particular there is a unique generator of each homology, which is used to fix $\gr_w$ and $\gr_z$ as absolute gradings by declaring that the generators of vertical and horizontal homology have grading zero. Setting $W=0$ in the vertical complex defines the \emph{hat vertical complex} and setting $Z=0$ in the horizontal complex defines the \emph{hat horizontal complex}; both are graded complexes over $\F$ whose homology is isomorphic to $\HFhat(S^3) = \F$. The horizontal and vertical terminology is used because we often represent $\CFKr(K)$ in the plane such that $W^a Z^b$ times a generator appears at coordinates $(-a, -b)$ and differentials are represented by arrows that move leftward and/or downward (see Figure \ref{fig:trefoil-Heegaard-diagram}). In the vertical complex vertical arrows are treated as unlabeled arrows that can be cancelled in homology, and moreover in the hat vertical complex all non-vertical arrows are ignored.

\begin{remark}\label{rmk:WZ-notation}
Our notation differs from the original definition of the knot Floer complex in \cite{OzSz:knots} and \cite{Ras:knot-floer}, but has become more common in the literature. Note that each of the two formal variables induce a filtration on $\CFKminus_{\sR}(K)$ by negative powers of the variable. The original definition contained one formal variable $U$ (corresponding to our $W$) and an additional filtration; loosely speaking it is obtained from our notation by setting $Z=1$ but remembering the filtration that came from the $Z$ powers. We find the additional formal variable convenient and it highlights the symmetric behavior of the two filtrations. Another potential source of confusion is that the earliest appearances in the literature of the notation with two formal variables used $U$ and $V$ rather than $W$ and $Z$, but the present notation has the advantage that the variables match the corresponding basepoints. We use $\sU$ to denote the product $WZ$, using the script font to avoid confusion with $U$, the old notation for the current $W$.
\end{remark}

\begin{remark}
We are restricting to knots in $S^3$ and to $\F = \F_2$ coefficients for simplicity, but the knot Floer complex is defined for knots in arbitrary 3-manifolds and with more general coefficients. With more care the immersed curve interpretation of knot Floer invariants described later in this article can be extended to more general manifolds and to arbitrary field coefficients; see \cite{Hanselman:CFK}.
\end{remark}

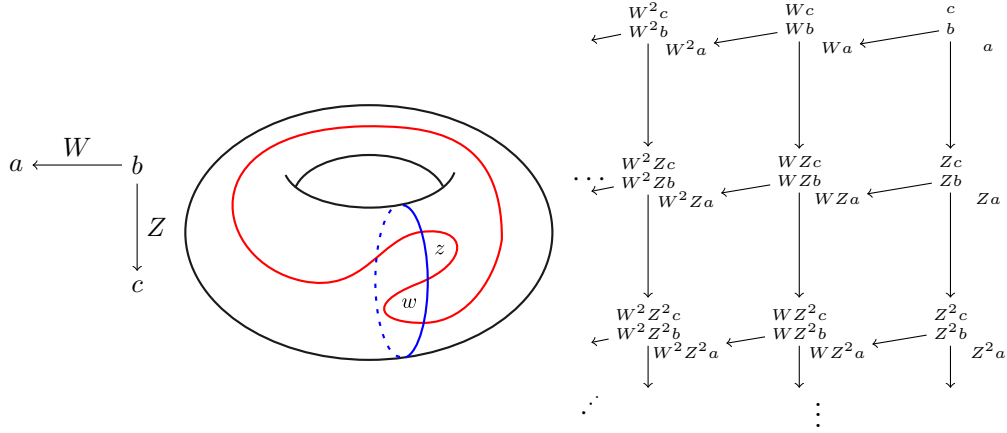
\begin{figure}
\raisebox{8mm}{
\begin{tikzpicture}[scale = .8]
    \node (a) at (0,2) {$a$};
    \node (b) at (2,2) {$b$};
    \node (c) at (2,0) {$c$};

    \draw[->] (b) -- node[above] {$W$} (a);
    \draw[->] (b) -- node[right] {$Z$} (c);
\end{tikzpicture}}
\includegraphics[scale = .8]{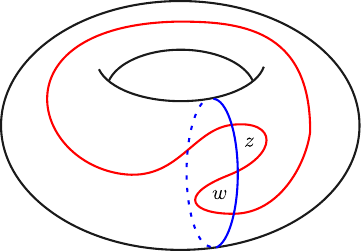}
\raisebox{-10mm}{\begin{tikzpicture}
    \tiny
    \node (a) at (.25, -.25) {$a$};
    \node (b) at (-.25,0) {$b$};
    \node (c) at (-.25,.25) {$c$};
    
    \node (Wa) at (-1.75, -.25) {$Wa$};
    \node (Wb) at (-2.25,0) {$Wb$};
    \node (Wc) at (-2.25,.25) {$Wc$};
    
    \node (WWa) at (-3.75, -.25) {$W^2a$};
    \node (WWb) at (-4.25,0) {$W^2b$};
    \node (WWc) at (-4.25,.25) {$W^2c$};
    
    \node (Za) at (.25, -2.25) {$Za$};
    \node (Zb) at (-.25,-2) {$Zb$};
    \node (Zc) at (-.25,-1.75) {$Zc$};
    
    \node (WZa) at (-1.75, -2.25) {$WZa$};
    \node (WZb) at (-2.25,-2) {$WZb$};
    \node (WZc) at (-2.25,-1.75) {$WZc$};
    
    \node (WWZa) at (-3.75, -2.25) {$W^2Za$};
    \node (WWZb) at (-4.25,-2) {$W^2Zb$};
    \node (WWZc) at (-4.25,-1.75) {$W^2Zc$};
    
    \node (ZZa) at (.25, -4.25) {$Z^2a$};
    \node (ZZb) at (-.25,-4) {$Z^2b$};
    \node (ZZc) at (-.25,-3.75) {$Z^2c$};
    
    \node (WZZa) at (-1.75, -4.25) {$WZ^2a$};
    \node (WZZb) at (-2.25,-4) {$WZ^2b$};
    \node (WZZc) at (-2.25,-3.75) {$WZ^2c$};
    
    \node (WWZZa) at (-3.75, -4.25) {$W^2Z^2a$};
    \node (WWZZb) at (-4.25,-4) {$W^2Z^2b$};
    \node (WWZZc) at (-4.25,-3.75) {$W^2Z^2c$};
    
    \large
    \node at (-2,-5) {$\vdots$};
    \node at (-5,-2) {$\cdots$};
    \node at (-5,-5) {$\cdot^{\cdot^{\cdot}}$};
    
    \draw[->] (b) -- (Wa);
    \draw[->] (Wb) --  (WWa);
    \draw[->] (Zb) -- (WZa);
    \draw[->] (WZb) -- (WWZa);
    \draw[->] (ZZb) -- (WZZa);
    \draw[->] (WZZb) -- (WWZZa);
    
    \draw[->] (WWb) --  (-5, -.15);
    \draw[->] (WWZb) --  (-5, -2.15);
    \draw[->] (WWZZb) --  (-5, -4.15);
 
    \draw[->] (b) --  (Zc);
    \draw[->] (Wb) --  (WZc);
    \draw[->] (WWb) --  (WWZc);
    \draw[->] (Zb) --  (ZZc);
    \draw[->] (WZb) --  (WZZc);
    \draw[->] (WWZb) --  (WWZZc);
    
    \draw[->] (ZZb) --  (-.25, -4.75);
    \draw[->] (WZZb) --  (-2.25, -4.75);
    \draw[->] (WWZZb) --  (-4.25, -4.75);
\end{tikzpicture}}
\caption{A doubly pointed Heegaard diagram for the right-hand trefoil. The resulting chain complex is indicated to the right, where intersection points in the diagram are labelled $a$, $b$, $c$ from top to bottom. The left shows a schematic for the complex as a finitely generated complex over $\sR$, and the right shows this complex viewed as an infinitely generated complex over $\F$.}\label{fig:trefoil-Heegaard-diagram}
\end{figure}
\begin{example}\label{ex:trefoil} A genus one doubly pointed Heegaard diagram for the right-hand trefoil is shown in Figure \ref{fig:trefoil-Heegaard-diagram}. The construction of $\CFKr(K)$ is easy to visualize in this case because the Heegaard diagram has genus one and we do not need to consider symmetric products. The complex is generated by intersection points of the single $\alpha$ curve and the single $\beta$ curve in the torus, and the differential counts bigons in the torus with one boundary on $\alpha$ and one boundary on $\beta$; in this setting we may count disks combinatorially and do not need to introduce a symplectic structure. There are three generators $a$, $b$, and $c$ and exactly two disks that contribute to the differential. One disk covers $w$ and the other covers $z$, and resulting differential gives
$$\partial(b) = Wa + Zc, \qquad \partial(a) = \partial(c) = 0.$$
If we declare that $\gr(a) = (0,-2)$, then \eqref{eq:relative-grading}  implies $\gr(b) = (-1,-1)$ and $\gr(c) = (-2, 0)$. Figure \ref{fig:trefoil-Heegaard-diagram} shows the complex  $\CFKr(K)$ represented in two ways. On the left we see the three generators of $\CFKr(K)$ with terms in the differential represented by arrows labelled by elements of $\sR$, while on the right the complex is viewed as being generated over $\F$ by terms of the form $W^a Z^b x$ for $x$ one of the generators over $\sR$.
\end{example}

While examples that do not require working with symmetric products are rare---knots admitting a genus one doubly pointed Heegaard diagram are a special class called (1,1) knots---it is instructive to have this example in mind when we consider immersed curves, which always live in the torus and for which Floer homology can always be computed combinatorially.

\subsection{Knot-like complexes} It will be helpful to define a class of bigraded chain complexes over $\sR$ (or $\sRhat$) which contains the knot Floer invariants described above.

\begin{definition}
A free finitely generated bigraded chain complex $C$ over $\sR$ is a \emph{knot-like complex} over $\sR$ if the horizontal and vertical homologies of $C$ are both isomorphic to $\F[\sU]$ and the generators of these homologies have $\gr_w = 0$ and $\gr_z = 0$, respectively. Similarly a free finitely generated bigraded chain complex $\widehat C$ over $\sRhat$ is a \emph{knot-like complex} over $\sRhat$ if the horizontal and vertical homologies are both $\F$ with the same grading condition. Let $\complexes$ and $\hatcomplexes$ denote the sets of knot-like complexes over $\sR$ and $\sRhat$.
\end{definition}

In addition to being a knot-like complex, the complex $\CFKr(K)$ for any knot $K$ in $S^3$ satisfies a symmetry condition: it is preserved, up to chain homotopy, by switching the roles $W$ and $Z$. We call a knot-like complex with this additional constraint a \emph{symmetric knot-like complex}.
Although knot Floer complexes are symmetric, the symmetry is not essential to most of the invariants discussed in this article so we will not assume this added constraint unless otherwise mentioned.

The sets of knot-like complexes $\complexes$ and $\hatcomplexes$ form categories, where morphisms are chain maps of graded $\sR$ or $\sRhat$ complexes. Generators (over $\sR$ or $\sRhat$) of $\Mor_{\sR}(C, C')$ or $\Mor_{\sRhat}(C, C')$ are elementary morphisms $(x_i \to y_j)$ that take a generator $x_i$ of $C$ to a generator $y_j$ of $C'$, and the bigrading of the generator $(x_i \to y_j)$ is $\gr(y_j) - \gr(x_i)$. Morphisms can be multiplied by elements of the coefficient ring, giving a module action, and the differential is determined by the differential on $C$ and $C'$.  We also have that $\Mor(C, C') \cong C^\vee \otimes C'$ where $C^\vee$ denotes the dual of $C$ in which gradings are negated and the direction of differential arrows is reversed.

\section{Decorated immersed curves}\label{sec:curves}

The knot Floer invariants described in the preceding section admit another interpretation in terms of decorated immersed curves. This stems from a structure theorem for knot-like complexes which says that (chain homotopy equivalence classes of) such complexes are equivalent to (homotopy classes) of decorated collections of immersed curves in a cover of a marked torus. In this section we introduce the relevant class of immersed curve objects, and the following section sketches their relationship to knot-like complexes.

\subsection{Curves in marked surfaces} We will work in the torus $T = T^2$, which we identify with $(\R/\Z)^2$, with either one or two marked points, as well as certain covers of $T$. Let $T_u$ denote the torus with a single marked point $u$ at the point $(0,\tfrac 1 2)$, and let $T_{w,z}$ denote the torus with marked points $w$ and $z$ at points $(\epsilon, \tfrac 1 2)$ and $(-\epsilon, \tfrac 1 2)$ respectively for some small $\epsilon$. We will primarily work in the $\Z$-fold cover of $T$ that is naturally identified with the infinite cylinder $(\R/\Z)\times \R$, which we denote $\barT$. Let $\barT_u$ and $\barT_{w,z}$ denote this cover equipped with the lifts of the marked point(s) in $T_u$ or $T_{w,z}$. Occasionally it is convenient to work in the universal cover $\tildeT \cong \R^2$ of $T$, denoted $\tildeT_u$ or $\tildeT_{w,z}$ when equipped with appropriate marked points.

We will consider compact immersed curves $\gamma$ in $\barT_u$, that is, immersions of $S^1$ into $\barT$ that avoid the lifts of $u$ (which occur at the points $(0,n+\tfrac 1 2)$ for integers $n$). More precisely we consider homotopy classes of such curves, where homotopies do not pass through lifts of $u$. We remark that any homotopy class of curve $\gamma$ in $\barT_u$ determines a homotopy class of curve in $\barT_{w,z}$ by first homotoping $\gamma$ to ensure it avoids an $\epsilon$ neighborhood of lifts of $u$ before replacing lifts of $u$ with lifts of $w$ and $z$. Conversely a curve in $\barT_{w,z}$ may be viewed as a curve in $\barT_u$ provided it does not cross any of the arcs from $(-\epsilon, n+\frac 1 2)$ to $(\epsilon, n + \frac 1 2)$ for integers $n$. We will ultimately be interested in collections of one or more curves $\Gamma = \{\gamma_0, \gamma_1, \ldots, \gamma_n\}$, where each $\gamma_i$ is a (homotopy class of) immersed curves in $\barT_u$ or $\barT_{w,z}$ as defined above; we will call such a collection an \emph{immersed multicurve} in $\barT_u$ or $\barT_{w,z}$.

At times we need to project curves (in either $\barT_u$ or $\barT_{w,z}$) to the torus by the projection map $p: \barT \to T$. In principle this projection loses information, but this information (a choice of lift of the curve in $T$ to $\barT$) is retained if curves are decorated with appropriate grading information (see Section \ref{sec:gradings}), so we may choose to work in either $T$ or $\barT$. A helpful philosophy is that it is most natural to view curves as living in $\barT$, where the grading information is visible geometrically, but pairing statements are cleanest in $T$ so we generally consider projections of curves to $T$ when pairing. It may also be convenient to work in the (marked) universal cover $\tildeT_u$ or $\tildeT_{w,z}$ and consider the preimage of curves under the projection map $\overline p: \tildeT \to \barT$, but this is primarily for the sake of clearer pictures.

Cutting the infinite strip $\barT \cong (\R/\Z) \times \R$ open along the line $\mu_{\frac 1 2} = \{\frac 1 2\} \times \R$ produces an infinite strip $[-\frac 1 2, \frac 1 2]\times \R$, which we will denote $\overline \sS$, and similarly cutting $T$ open along $\{\frac 1 2\} \times (\R/\Z)$ produces an annulus $\sS$. Both $\sS$ and $\overline \sS$ have two boundary components, which we refer to as the left and right boundaries. Let $\sS_u$, $\sS_{w,z}$, $\overline \sS_u$, and $\overline \sS_{w,z}$ denote versions of the annulus and infinite strip with appropriate marked points. We will also consider compact immersed curves in these surfaces, where we now allow both immersed $S^1$ components and immersed arcs with endpoints on the boundary of $\sS$ or $\overline \sS$. Collections of such curves/arcs are immersed multicurves in the appropriate marked surface. An immersed multicurve in $\barT_u$ clearly determines an immersed multicurve in $\overline \sS_u$, but the converse is not always true since we at least need the same number of arc endpoints on the left and right boundaries of $\overline \sS_u$, and if there are many endpoints on each side then gluing ends in different ways can produce different multicurves in $\barT_u$. That said, we will ultimately restrict to immersed multicurves in $\barT_u$ that intersect $\mu_{\frac1 2}$ exactly once and to multicurves in $\overline \sS_u$ that have exactly one arc component with one endpoint on each boundary of $\overline \sS_u$, which are clearly equivalent. In light of this, and by abuse of notation, we will sometimes drift freely between multicurves in $\barT_u$ and $\overline \sS_u$. The distinction is mainly relevant when we discuss morphisms in Section \ref{sec:morphisms}\footnote{The distinction is also relevant when considering knots in manifolds other than $S^3$, when curves in $\barT_u$ may cross $\mu_{\frac 1 2}$ more than once and information is lost by cutting the cylinder open, but we will not consider this case here.}.

At times we will work instead in the punctured surfaces obtained from $T_u$, $\barT_u$, $\sS_u$ or $\sS_u$ by removing the marked point(s) $u$, which we denote $T_\bullet$, $\barT_\bullet$, $\sS_\bullet$, and $\overline \sS_\bullet$. In general passing from marked surfaces to punctured surfaces will correspond to setting $\sU = 0$ and working over $\sRhat$ rather than over $\sR$. A homotopy class of curves in $\barT_\bullet$ is clearly equivalent a homotopy class of curves in $\barT_u$, the difference will come when we count bigons which may cover marked points but may not cover punctures.

\subsection{Floer complexes} Let $\Sigma_u$ denote either $\sS_u$, $\overline \sS_u$, $T_u$, or $\barT_u$ (only these will be relevant to us, but more generally $\Sigma_u$ could be any surface with one or more $u$ marked points). Given two immersed curves $\gamma$ and $\gamma'$ in $\Sigma_u$ that intersect transversely, we define a vector space $CF(\gamma, \gamma')$ generated over $\F[\sU]$ by points in $\gamma\cap\gamma'$. We further equip $CF(\gamma, \gamma')$ with a linear map $\partial$ that counts immersed bigons. Given intersection points $x$ and $y$, an immersed bigon from $x$ to $y$ is a map of the unit disk $D^2 \subset \C$ into $T$ that takes  $-i$ to $x$, $i$ to $y$, the positive real part half of $\partial D^2$ to $\gamma$ and the negative real part half of $\partial D^2$ to $\gamma'$ that is an immersion except at $\pm i$ and for which the image of a small neighborhood of either $-i$ or $i$ covers exactly one quadrant of a neighborhood of $x$ or $y$. Let $\sM(x, y)$ denote the set of equivalence classes of immersed bigons from $x$ to $y$, where bigons are considered equivalent if they differ by precomposing with a self-homeomorphism of $D^2$. One can check that in the settings we will consider $\sM(x,y)$ is finite for each pair of generators. Each equivalence class $B$ of bigons is counted with weight $\sU^{n_u(B)}$, where $n_u(B)$ is the multiplicity with which (any representative of) $B$ covers the basepoint $u$. The map $\partial$ is defined by
$$\partial x = \sum_{y \in \gamma \cap \gamma'} \sum_{B \in \sM(x,y)} \sU^{n_u(B)} y.$$
In favorable circumstances the map $\partial$ is a differential, making $CF(\gamma, \gamma')$ a chain complex, but this is not true for all pairs of immersed curves $(\gamma, \gamma')$. What can be said is that for $k \ge 0$ and any transverse collection of curves $(\gamma_0, \gamma_1, \ldots \gamma_k)$ there are operations
$$m_k: CF(\gamma_0, \gamma_1) \otimes \cdots \otimes CF(\gamma_{k-1}, \gamma_k) \to CF(\gamma_0, \gamma_k)$$
defined by counting immersed $(k+1)$-gons that these satisfy $\Ainfty$ relations. In other words, the set of immersed multicurves in $\Sigma_u$ have the structure of a curved $\Ainfty$-category; this is the Fukaya category of the surface. For each object $\Gamma$ the curvature is given by the map $m_0: \F \to CF(\Gamma, \Gamma)$ that counts immersed monogons with boundary on $\Gamma$. The curvature is determined by $m_0(1)$, which roughly speaking is a collection of self-intersection points of $\Gamma$. Note that $m_1$ is simply the map $\partial$ defined above. The relations imply that $m_0$ is an obstruction to $m_1$ squaring to 0. We say that an immersed curve is \emph{unobstructed} if $m_0(1)$ vanishes, and for any pair of unobstructed curves $\partial^2 = 0$ and $CF(\gamma, \gamma')$ is a chain complex over $\F[\sU]$. We emphasize that we do not forbid monogons, we only require that monogons with a given corner cancel in pairs. For example, the figure-eight shaped curve in Example \ref{ex:first-examples} bounds two immersed monogons with the same corner, with each one covering a single marked point; since these monogons make canceling contributions to $m_0$, the curve is unobstructed.

We can also define a Floer complex for curves $\gamma$ and $\gamma'$ in a doubly marked surface $\Sigma_{w,z}$ such as $\sS_{w,z}$, $\overline \sS_{w,z}$, $T_{w,z}$, or $\barT_{w,z}$. This is defined analogously, using $\F[W,Z]$ coefficients and counting each equivalence class $B$ of immersed bigons with weight $W^{n_w(B)} Z^{n_z(B)}$, with $n_w$ and $n_z$ denoting the multiplicity with which $B$ covers the basepoint(s) $w$ and $z$.

Setting $\sU = 0$ gives a simplified Floer complex $\widehat{CF}(\gamma, \gamma')$, which is generated over $\F$ by $\gamma \cap \gamma'$. The differential, and more generally the $m_k$ maps, only count polygons with $n_u(B) = 0$; equivalently, we count immersed polygons in the appropriate punctured surface $\sS_\bullet$, $\overline \sS_\bullet$, $T_\bullet$ or $\barT_\bullet$, where the marked points $u$ have been removed.

The Floer complex extends to immersed multicurves as a direct sum over components: for multicurves $\Gamma$ and $\Gamma'$ we have
$$CF(\Gamma, \Gamma') = \bigoplus_{\gamma\in\Gamma, \gamma'\in\Gamma'} CF(\gamma, \gamma'),$$
and likewise for $\widehat{CF}(\Gamma, \Gamma')$.

\subsection{Gradings} \label{sec:gradings} 
The curves we consider are decorated with gradings. A grading (more specifically, a $\Z$-grading) on an immersed curve $\gamma$ in an unmarked surface is a map $\tau: \gamma \to \R$ that is a lift of the tangent slope map from $\gamma$ to $\RP \cong \R/\Z$. Defining the tangent slope map requires us to fix a trivialization of the tangent bundle of the surface; in all of the surfaces we consider there is an obvious trivialization coming from the identification of $\tildeT$ with $\R^2$, and we identify the set of slopes with $\R / \Z$ such that a horizontal line has slope 0 and rotating counterclockwise by $r$ radians gives a line slope $\frac r \pi$. Note that the grading function $\tau$ on a curve is determined by the curve up to an integer shift. For an immersed multicurve, the grading function is determined up to a separate integer shift on each component. We also remark that the mod 2 reduction of the grading determines (and is determined by) an orientation on the curve by requiring the resulting map from $\gamma$ to $\R/2\Z$ to be the oriented tangent vector map. Note that not all immersed curves admit a $\Z$-grading, since a lift of the tangent slope to $\R$ may not exist. Such a lift only exists if the net rotation of the tangent slope around the curve is $0$, and in this case we say the curve is \emph{$\Z$-gradable}. If a curve is not $\Z$-gradable, we may consider lifts of the tangent slope to $\R/N\Z$ for some $N$--if such a lift exists the curve is $\Z/N\Z$-gradable and the lift is a $\Z/N\Z$-grading. Note that every immersed curve is $\Z/2\Z$-gradable, with a $\Z/2\Z$-grading corresponding to a choice of orientation on the curve.

Gradings on immersed curves in the marked surface $\barT_u$ (or $\overline \sS_u$) are defined similarly, but we would like to modify the definition so that the grading change around a closed loop encodes the number of marked points enclosed in the loop. One way to do this is to make cuts from each marked point to one of the infinite ends of $\barT$ and to make the grading function discontinuous in a prescribed way when the curve crosses a cut. Specifically, the grading jumps up by 2 when the curve crosses a cut from the left side of the cut to the right side of the cut (we view cuts as oriented from the marked point to the infinite end of the cylinder $\barT$). If the curve crosses multiple cuts at once then the grading jump is 2 times the number of cuts crossed; for this to make sense we must ensure that finitely many cuts coincide at any given point. We will fix a concrete choice of cuts on $\barT_u$: for the marked point with height $y \in \Z + \frac 1 2$, the cut is the ray $\{ (0, t) \}$ where $t$ ranges from $y$ to $\infty$ if $y>0$ or from $y$ to $-\infty$ if $y < 0$. In particular, our grading functions for curves in $\barT_u$ jump only when they cross the line $\mu = \{0\}\times \R$, and when they cross from the left side of $\mu$ to the right side of $\mu$ between the marked points at height $n-\frac 1 2$ and $n + \frac 1 2$ the grading function increases by $2n$. Note that the grading function is still continuous away from $\mu$, and its mod 2 reduction is still the oriented tangent slope function. It is still true that the grading on each component of a multicurve is determined up to an overall shift by the curve. Once again a $\Z$-grading may not exist for a given immersed curve, and if it does we say the curve is $\Z$-gradable. Given a closed curve $\gamma$ we can consider $\rot(\gamma)$, the net counterclockwise rotation of the tangent slope along the loop $\gamma$ (in radians/$\pi$) and also the total winding number $\wind(\gamma)$ around the marked points which is equal to the signed number of times the curve crosses a grading cut. It is clear that $\gamma$ is $\Z$-gradable if and only if $\rot(\gamma) = 2\wind(\gamma)$.

Curves in the doubly marked cylinder $\barT_{w,z}$ (or $\overline \sS_{w,z}$), if they are suitably gradable, can be equipped with two different grading functions $\tau_w$ and $\tau_z$ from the curve to $\R$, with each defined the same way as the grading on $\barT_u$ but using only the $w$ or only the $z$ marked points, where the grading cuts used in $\barT_u$ are shifted over to lie on either $\{\epsilon\}\times \R$ or $\{-\epsilon\}\times \R$. A curve is $\Z$-bigradable if such lifts exist. A $\Z$-gradable curve in $\barT_u$ viewed as a curve in $\barT_{w,z}$ is always $\Z$-bigradable. In this case the two gradings agree outside of the strip $(-\epsilon, \epsilon) \times \R$ and they differ within that strip by an amount related to the height.

We define gradings on curves in $T_u$ (or similarly in $\sS_u$, $T_{w,z}$, or $\sS_{w,z}$) indirectly: a grading function on a multicurve $\Gamma$ in $T_u$ is a function $\tau: \Gamma \to \R$ such that $\tau \circ p$ is a grading function on some lift of $\Gamma$ to $\barT$. In other words, $\tau$ agrees mod 2 with the oriented tangent slope and is continuous except along the curve $\{0\} \times (\R/\Z)$ at which it jumps by even amounts, and there is a consistency condition on the magnitude of the jumps in a given component of $\Gamma$. Note that a grading on a curve in $T_u$ specifies a choice of lift of that curve to $\barT_u$, since a point at which the grading jumps by $2k$ must lift to height $k$.

If two curves $\gamma$ and $ \gamma'$ in $\barT_{u}$ are equipped with gradings $\tau$ and $\tau'$ and the curves intersect transversely away from the grading cuts, the Floer complex $CF(\gamma, \gamma')$ can also be equipped with an integer grading. For a generator corresponding to an intersection point $x$ in $\gamma \cap \gamma'$ the grading is defined to be
$$\gr(x) = \lfloor \tau'(x) - \tau(x) \rfloor$$
where $\lfloor \cdot \rfloor$ denotes the greatest integer function. Equivalently, the grading is $\tau'(x) - \tau(x) - \theta$ where $\theta$ is $\frac 1 \pi$ times the angle covered when turning counterclockwise from $\gamma$ to $\gamma'$. In other words, $\gr(x)$ records how much the grading function must jump if a path makes a left turn from $\gamma$ to $\gamma'$ at $x$. Note that if the gradings $\tau$ and $\tau'$ are only $\Z/N\Z$ gradings, they determine a $\Z/N\Z$ grading on $CF(\gamma, \gamma')$. In particular there is a $\Z/2\Z$ grading on $CF(\gamma, \gamma')$ determined only by the orientations on the curves corresponding to the sign of intersection points.

We can also consider gradings related to (transverse) self-intersection points of an immersed curve $\gamma$ in $\barT_u$. Any transverse self-intersection point $p$ can be viewed as an intersection point between the two segments of the restriction of $\gamma$ to a small neighborhood of $p$. Thus $p$ can be given two gradings, both defined as above but differing in which segment plays the role of $\gamma$ and which plays the role of $\gamma'$. It is not difficult to check that these two gradings sum to $-1$; indeed, if we imagine turning leftward from one segment to the other and then leftward from that segment back to the first, the grading function returns to the starting value but we have also rotated counterclockwise by $\pi$, so the sum of the two grading jumps must have been $-1$. In particular, exactly one of the two gradings at $p$ is even; this is the ordering of segments for which making a left turn from the first to the second respects the orientation on the curve. Let us write the two gradings as $2d$ and $-1-2d$ for some integer $d$, and we define the \emph{degree} of $p$, $\deg(p)$, to be this integer $d$. The degree records half the grading jump that occurs when a left turn is made at $p$ that is consistent with the orientations on $\gamma$. For self-intersection points of curves in $\barT_{w,z}$ we could define a bidegree $(d_w, d_z)$ analogously, but this will not be necessary. For curves in $\barT_{w,z}$ coming from curves in $\barT_u$, if we assume self-intersections occur outside the strip $(-\epsilon, \epsilon) \times \R$, we always have $d_w = d_z = d$ where $d$ is the degree of the corresponding curve in $\barT_u$.

\subsection{Bounding chains}\label{sec:bounding-chains}
Immersed curves alone are not sufficient to represent all knot-like complexes, we will need to enhance our immersed curves with an additional decoration called a \emph{bounding chain}. Bounding cochains are a standard feature in the construction of Fukaya categories and Floer cohomology in the immersed setting \cite{AkahoJoyce}, we use the term bounding chain since Heegaard Floer homology is set up with homological rather than cohomological conventions. Our treatment of bounding (co)chains is also greatly simplified from the general case given that we are restricting to curves in surfaces, and we give a more combinatorial description that we hope is more accessible but may be unsatisfying to symplectic geometers. Moreover, the description of bounding chains here is particularly brief and focuses on intuition rather than rigor; the reader is referred to  \cite[Section 3.3]{Hanselman:CFK} for more details. Looking ahead to the examples in Section \ref{sec:examples}, particularly Example \ref{ex:nontrivial-bchain}, will also help in understanding this section.

Given an immersed multicurve $\Gamma$, let $\sI_{\ge 0}$ denote the set of self-intersection points of $\Gamma$ that have nonnegative degree and let $\sI_0$ be the set of degree 0 self-intersection points; a bounding chain $\bchain$ on $\Gamma$ will be a particular subset of $\sI_{\ge 0}$, and $\widehat\bchain$ will be its restriction to $\sI_0$. If multicurves $\Gamma_i$ are each equipped with a subset $\bchain_i$ of the relevant set $\sI_{\ge 0}$, we can modify the $\Ainfty$ operations $m_k$ so that the boundaries of the polygons counted are allowed to make additional left turns at self-intersection points of curves that are in some $\bchain_i$. More precisely, we define modified maps $m_k^\bchain$ similar to $m_k$ except that they count, for all nonnegative integers $n$, immersed $(k+1+n)$-gons such that $k+1$ of the corners are at intersections in $\Gamma_i\cap\Gamma_{i+1}$ or $\Gamma_0\cap \Gamma_k$ as usual and the remaining corners are at elements of $\bchain_i$ for some $i$. We call such a polygon a \emph{generalized $(k+1)$-gon} and we call the additional corners \emph{false corners} of the polygon. We require that at false corners the boundary orientation is consistent with the curve orientation---that is, the orientations either agree both before and after the turn or they are opposite both before and after the turn. As in the definition of $m_k$, each polygon counted by $m_k^{\bchain}$ contributes with weight that, in addition to the term $\sU^{n_u(B)}$ recording the marked points covered by the polygon, also has a factor $\sU^{\deg(p)}$ for each false corner $p$. Although generalized polygons may have arbitrarily many false corners, under suitable hypotheses there are finitely many generalized $(k+1)$-gons with a given set of true corners so the maps $m_k^\bchain$ are well-defined. One can also check that the maps $m_k^\bchain$ still satisfy the $\Ainfty$-relations; in particular the map $m_0^\bchain$ that counts generalized monogons is an obstruction to $m_1^\bchain$ squaring to 0. We will be interested in collections of self-intersection points $\bchain$ for which this obstruction vanishes.
\begin{definition}
A \emph{bounding chain} on an immersed multicurve $\Gamma$ is a subset of the nonnegative degree self-intersection points $\sI_{\ge 0}$ of $\Gamma$ such that $m_0^\bchain = 0$. This constraint is known as the \emph{Maurer-Cartan equation}, we will also call it the \emph{no monogon condition}. If $\bchain$ satisfies this condition we say that the pair $(\Gamma, \bchain)$ is \emph{unobstructed}.
\end{definition} 
If two multicurves $\Gamma$ and $\Gamma'$ are decorated with bounding chains $\bchain$ and $\bchain'$ then the map $\partial = m_1^\bchain$ counting generalized bigons is a differential on $CF(\Gamma, \Gamma')$.

\begin{remark}
Recall that we are restricting to $\F = \F_2$ coefficients. For general field coefficients one considers linear combinations of points in $\sI_{\ge 0}$ rather than subsets of $\sI_{\ge 0}$, and coefficients in $\bchain$ of each false corner factor in to the weight of polygons. In this setting, in addition to self-intersection points $\bchain$ should include a basepoint on each curve component with nonzero coefficient. 
\end{remark}

\begin{remark}
A bounding chain $\bchain$ is more commonly defined as an element of $CF(\Gamma, \Gamma)$, the Floer chain complex of $\Gamma$ with (a suitable perturbation of) itself. The complex $CF(\Gamma, \Gamma)$ is generated over $\F[\sU]$ by intersections between the two copies of $\Gamma$. There are a pair of intersections for each self-intersection of $\Gamma$ and an extra pair of intersections for each closed curve in $\Gamma$ (due to the required perturbation) which can be viewed as occurring near a fixed basepoint of $\Gamma$, and for each pair one of the two intersections is ruled out from being included in $\bchain$ by a grading condition, which also determines the power of $\sU$ in the coefficient of the other intersection point. Thus an element of $CF(\Gamma, \Gamma)$ with the relevant grading condition can be encoded as an $\F$-linear combination of the set containing the self-intersection points of $\Gamma$ and a basepoint on each component of $\Gamma$.
\end{remark}

We can define bounding chains similarly in the $\sU = 0$ setting. Recall that working over $\sRhat$ corresponds to working in the punctured surface $\Sigma_\bullet$, or equivalently ignoring all immersed polygons with $n_u(B) \neq 0$ when defining the operations $m^b_k$. We may also exclude any self-intersection points in $\bchain$ with strictly positive degree, since polygons using these are weighted by a positive power of $\sU$. Thus a \emph{bounding chain mod $\sU$} is a subset $\bchainhat$ of $\sI_0$ for which the count $m_0^{\bchainhat}$ of generalized immersed monogons in $\barT_\bullet$ vanishes. For any bounding chain $\bchain$ the restriction $\bchainhat$ of $\bchain$ to $\sI_0$ is automatically a bounding chain mod $\sU$, but we remark there are bounding chains mod $\sU$ that do not come from a bounding chain in the minus setting (see, for instance, Example \ref{ex:nontrivial-local-system}). If $\Gamma$ is a curve in $\Sigma_\bullet$ and $\bchainhat$ is a bounding chain mod $\sU$ for $\Gamma$, we will either say that the pair $(\Gamma, \bchainhat)$ is \emph{unobstructed mod $\sU$} or that it is an \emph{unobstructed decorated curve in $\Sigma_\bullet$}.

\subsection{The immersed Fukaya category}\label{sec:fukaya} We may now consider an $\Ainfty$-category whose objects are unobstructed decorated immersed multicurves $(\Gamma, \bchain)$ in $\Sigma_u$, or unobstructed decorated immersed multicurves $(\Gamma, \bchainhat)$ in $\Sigma_\bullet$ if we wish to work mod $\sU$; we call this the immersed Fukaya category of the relevant marked surface and denote it $\mathcal{F}(\Sigma_u)$ or $\mathcal{F}(\Sigma_\bullet)$. Given two objects $(\Gamma, \bchain)$ and $(\Gamma', \bchain')$, the morphism complex is the Lagrangian Floer complex $CF(\Gamma, \bchain) , (\Gamma', \bchain') )$, with differential $\partial = m_1^\bchain$, and its homology is $HF(\Gamma, \bchain) , (\Gamma', \bchain') )$. If $\Sigma$ is $\sS$ or $\overline \sS$ we should be more precise about what Floer homology means for arc components of curves: In $\sS$ the correct setting is the wrapped Fukaya category of the annulus, where when defining the Floer complex $CF((\Gamma, \bchain) , (\Gamma', \bchain') )$ we wrap the ends of arc components of $\Gamma'$ along the boundary circles of $\sU$ following the boundary orientation while fixing the ends of arc components of $\Gamma$, and in the cover $\overline \sS$ we fix ends of arc components of $\Gamma$ and push the ends of arc components of $\Gamma'$ arbitrarily far downward on the left boundary of $\overline \sS$ and arbitrarily far upward on the right boundary of $\overline \sS$ (see for example Figure \ref{fig:morphisms}).

\begin{figure}
\includegraphics[scale = 1]{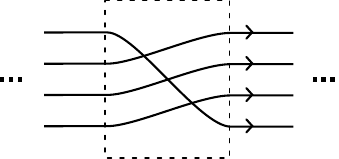}
\caption{The crossing region of a non-primitive curve, for which the strands run parallel outside of this region. The self-intersection points in this region are the local system self-intersection points.}
\label{fig:local-system-points}
\end{figure}

We will call two decorated curves \emph{Floer-equivalent} if they have isomorphic Floer homology when paired with every other object, following \cite[Definition 10.4]{Hanselman:CFK}; this is the notion of equivalence for decorated curves used throughout this paper. Note that quasi-isomorphic objects in the immersed Fukaya category are Floer-equivalent, but we do not need or assert the converse. It is helpful to, as much as possible, choose a preferred representative for each equivalence class of decorated immersed curves. In light of this, we will be able to restrict to decorated curves $(\Gamma, \bchainhat)$ in $\barT_\bullet$ for which the decoration $\bchainhat$ has a very limited form that we now describe. By applying a homotopy to the immersed curve $\Gamma$ if necessary, we will often assume that $\Gamma$ has a particularly nice form:
\begin{definition} An immersed multicurve $\Gamma$ in $\overline \sS_u$ (or in $\overline \sS_\bullet$) is in \emph{simple position} if:
\begin{itemize}
\item $\Gamma$ has transverse self-intersection points and no triple points,
\item $\Gamma$ intersects the vertical line $\mu = \{0\}\times \R$ minimally and orthogonally, with every component of $\Gamma$ intersecting $\mu$,
\item $\Gamma$ has minimal self-intersection, and
\item each non-primitive component of $\Gamma$ lies in a small neighborhood of an underlying primitive curve, and has the form of parallel copies of that curve outside of a small rectangle, in which the strands cross according to a cyclic permutation as shown in Figure \ref{fig:local-system-points}.
\end{itemize}

\end{definition}
In some settings we must modify this definition slightly because admissibility conditions (which we will not state precisely here) are needed to ensure the finiteness of polygon counts. One way to ensure admissibility is to forbid two components of $\Gamma$ from bounding an immersed annulus.

\begin{definition}\label{def:almost-simple-position}
An immersed multicurve $\Gamma$ in $\overline \sS_u$ is in \emph{almost simple position} if no two components bound an immersed annulus and it satisfies the conditions of simple position defined above except that the self-intersection of $\Gamma$ is only minimal subject to the no annuli constraint.
\end{definition}
See for instance the multicurve in Figure \ref{fig:local-system}, where the presence of two homotopic components requires the curves to be in non-minimal position to avoid an immersed annulus.

For curves in simple or almost simple position, the self-intersection points that occur in the crossing regions of any non-primitive components play a special role. These points, which always have degree 0, are called \emph{local system intersection points}. We say that $\bchainhat \subset \sI_0$ has \emph{local system type} if it is a subset of the local system intersection points of $\Gamma$. 

\begin{remark}\label{rmk:local-systems}
The significance of the terminology is the observation that a non-primitive curve $\gamma$ in $\overline \sS_\bullet$ decorated with a bounding chain $\bchainhat$ of local system type is equivalent to the underlying primitive curve decorated with a local system (see \cite[Section 6.2]{Hanselman:CFK} for more on this equivalence). Thus instead of considering decorated curves $(\Gamma, \bchainhat)$ with $\bchainhat$ of local system type, we could alternatively consider the category of unobstructed immersed curves in $\overline \sS_\bullet$ decorated with local systems; this aligns with the approach taken in \cite{HKK} and \cite{HRW}. From this perspective, the relevant objects in the minus setting would still be curves with bounding chains, but all curves would be primitive, they would be equipped with local systems, and in the bounding chains each self-intersection point would be labelled with a linear map between the vector spaces attached to the relevant curves.
\end{remark}

\subsection{Immersed train tracks} \label{sec:train-tracks}
One helpful way to interpret bounding chains is by passing from immersed curves to immersed train tracks, that is, immersed graphs such that at each vertex all incident edges are tangent. Including a self-intersection point $p$ in $\bchain$ can then be viewed as adding two new edges to the train track that, like exit ramps on a highway, allow a path to turn smoothly from one segment of the curve to the other near $p$ as in Figure \ref{fig:crossover-arrow}. Although paths are allowed to travel either direction along $\Gamma$, these new edges are directed so paths may only follow them in one direction. The new edges must be consistent with the orientation on $\Gamma$ (i.e. the directed path along the new edge must agree with the orientation at both ends or oppose the orientation at both ends) and are always directed so that following them corresponds to making a left turn. The new turning edges are also weighted by $\sU^{\deg p}$. From this perspective, the $\Ainfty$ operations are weighted counts of immersed $(k+1)$-gons, with no additional corners but for which each side of the polygon maps to a smooth path in the relevant train track. We say that a train track is unobstructed if the count $m_0$ of immersed monogons vanishes, and for two unobstructed train track counting immersed bigons defines a differential.

\begin{figure}
\includegraphics[scale = .8]{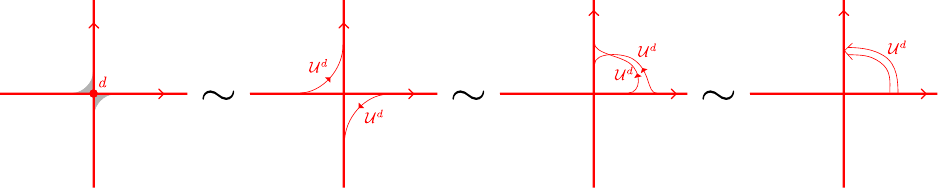}
\caption{An immersed curve with bounding chain $\bchain$ may be viewed as an immersed train track by adding two directed edges at each self-intersection point contained in $\bchain$, weighted by $\sU^d$ where $d$ is the degree of the self-intersection point. If both edges are placed in the same quadrant they cross once and are roughly parallel except for turning opposite directions at the ends. We use a bold arrow as a notational shorthand for such a pair of directed train track edges, and call this a crossover arrow.}
\label{fig:crossover-arrow}
\end{figure}

The pair of edges added at each point $p$ in $\bchain$ may be drawn in opposing quadrants or in the same quadrant as seen in Figure \ref{fig:crossover-arrow}. In the latter case they cross once and may be viewed as running roughly parallel to each other apart from turning opposite directions at the ends. Such a pair of arrows will be called a crossover arrow and is denoted by a bold arrow as a diagrammatic shorthand. We restrict to immersed train tracks that are obtained from immersed multicurves by adding crossover arrows. If all crossover arrows occur in a neighborhood of a self-intersection point and move counterclockwise in one quadrant as in Figure \ref{fig:crossover-arrow} then clearly the train track is equivalent to an immersed multicurve decorated with a bounding chain. We may also consider crossover arrows that do not lie near a self-intersection point, but we note that any crossover arrow can be made to look like those in Figure \ref{fig:crossover-arrow} after a homotopy of the immersed curve. We do place some restrictions on crossover arrows---for instance, we assume that crossover arrows are disjoint from each other---see \cite[Section 4.2]{Hanselman:CFK} for details.

\begin{figure}
\includegraphics[scale = .8]{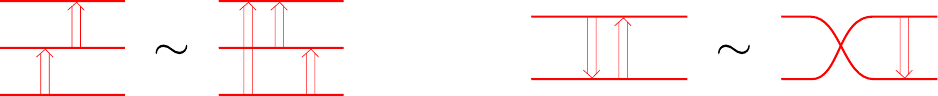}
\caption{Some local moves that preserve equivalence class of train tracks. Left: Crossover arrows may slide along immersed curves, but composite arrows must be added or removed when the head of an arrow slides past the tail of another. Right: A pair of opposite crossover arrows between the same sections of immersed curve may be replaced with one crossover arrow and one crossing.}
\label{fig:local-moves}
\end{figure}
The main advantage of train tracks and crossover arrows lies in visualizing equivalences between objects in the immersed Fukaya category, particularly when working mod $\sU$. For example, consider a decorated immersed curve $(\Gamma, \bchainhat)$ in $\barT_\bullet$ and the corresponding immersed train track. The curve $\Gamma$ along with the crossover arrows divide $\barT_\bullet$ into connected components, and if the component immediately to the left of a crossover arrow contains a puncture or is unbounded then removing that crossover arrow produces an equivalent object because no immersed polygon whose boundary follows that crossover arrow can contribute to any $m_k$ map. Moreover, in many situations sliding the ends of a crossover arrow along the immersed curve $\Gamma$ preserves Floer homological data when paired with all other train tracks and thus results in an equivalent object. Often a crossover arrow can be slid to a position where it is removable, giving a mechanism to replace $\bchain$ with a bounding chain $\bchain'$ containing fewer points such that $(\Gamma, \bchain')$ is equivalent to $(\Gamma, \bchain)$. In some situations sliding crossover arrows does not preserve the equivalence class of a train track, mainly when a slide causes multiple crossover arrows to interact, but certain more complicated local moves that do preserve the equivalence class of a train track are allowed. Some of these are shown in Figure \ref{fig:local-moves}: sliding the head of one crossover arrow past the tail of another crossover arrow requires adding a new crossover arrow obtained from composing the two, and a crossover arrow can slide past an opposing crossover arrow only at the expense of replacing it with a crossing (note that this move modifies the immersed curve $\Gamma$). A more thorough description of arrow-sliding rules for train tracks in punctured surfaces can be found in \cite[Section 3.4]{HRW} and \cite[Section 4.3]{Hanselman:CFK}.

\subsection{A structure theorem for decorated immersed curves in punctured surfaces} \label{sec:structure-thm}
Armed with the language of train tracks and the set of arrow sliding moves described above, one can seek to remove as many points from $\bchain$ as possible to provide the simplest representative for the equivalence class of $(\Gamma, \bchain)$. It turns out that when working mod $\sU$ almost all crossover arrows can be removed, and those that remain can be pushed to local-system self-intersection points. This gives rise to the following normal form for objects in the immersed Fukaya category.

\begin{theorem}\label{thm:curves-structure-theorem}
Let $\Sigma_\bullet$ be $\sS_\bullet$, $\overline \sS_\bullet$, $T_\bullet$, or $\barT_\bullet$. Any decorated immersed curve $(\Gamma, \bchainhat)$ in $\Sigma_\bullet$ is equivalent to a decorated curve $(\Gamma', \bchainhat')$ for which $\Gamma'$ is in (almost) simple position and $\bchainhat'$ has local system type. Moreover, such a representative for an equivalence class is unique in the following sense: if $(\Gamma', \bchainhat')$ and $(\Gamma'', \bchainhat'')$ are equivalent and both $\bchainhat'$ and $\bchainhat''$ are of local-system type then $\Gamma'$ is homotopic to $\Gamma''$ and $\bchainhat'$ and $\bchainhat''$ agree as subsets of self-intersection points (using the obvious identification of the local system self-intersection points of two homotopic curves in (almost) simple position).
\end{theorem}

We remark that the word \emph{almost} can be removed from the statement above, as immersed annuli between components of a multicurve do not end up causing problems in the hat setting, but we will need to work with curves in almost simple position to extend to the minus setting.

Theorem \ref{thm:curves-structure-theorem} is a special case (though using very different language) of a result for immersed Fukaya categories of arbitrary punctured surfaces \cite[Theorem 4.3]{HKK}. This structure theorem was also proved, using the more explicit train track algorithm sketched above, in the case of the punctured torus in \cite{HRW}, and that proof extends to arbitrary punctured surfaces as shown in \cite{KWZ} (the special case relevant for Theorem \ref{thm:curves-structure-theorem} is also considered in \cite{KWZ:mnemonic}).  The arrow sliding algorithm not only proves the existence of a representative of local system type but also gives a recipe for finding this representative. The uniqueness follows from pairing $(\Gamma', \bchainhat')$ and $(\Gamma'', \bchainhat'')$ with carefully chosen test objects. Note that the results in \cite{HKK}, \cite{HRW}, and \cite{KWZ} do not use the language of bounding chains but rather the language of curves decorated with local systems (see Remark \ref{rmk:local-systems}).

\begin{figure}
\includegraphics[scale = .8]{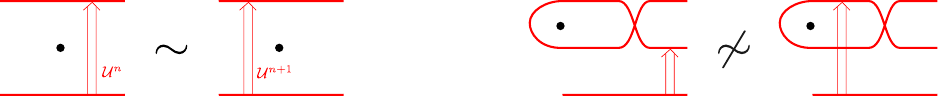}
\caption{Left: An allowed arrow slide in marked surfaces, in which an arrow slides past a marked point at the expense of modifying its weight by a factor of $\sU$. Right: Perhaps surprisingly, sliding a crossover arrow as shown does not result in an equivalent train track.}
\label{fig:minus-arrow-slides}
\end{figure}
One might hope to apply the same strategy to train tracks in marked surfaces $\Sigma_u$, but the situation is more subtle when we pass from punctures to marked points and work over $\sR$ rather than $\sRhat$. In particular, arrow sliding in marked surfaces leads to additional technical difficulties. Some useful arrow slide moves are allowed, for instance a crossover arrow may slide over a marked point on its left at the expense of multiplying the weight on the crossover arrow by $\sU$ (see the left side of Figure \ref{fig:minus-arrow-slides}), but other seemingly tame arrow slides do not result in equivalent train tracks (see the right side of Figure \ref{fig:minus-arrow-slides}). The difficulty ultimately arises from the presence of monogons, which can be assumed not to exist in punctured surfaces up to homotopy of the curves. These difficulties may not be insurmountable, but as of now there is no comprehensive arrow sliding calculus in the setting of marked surfaces.

Despite the additional difficulties, a partial structure theorem can be stated for decorated curves in $\overline \sS_u$, but rather than starting with an arbitrary decorated curve and simplifying it by a sequence of moves connecting equivalent objects the strategy relies essentially on the correspondence between curves and complexes described in the next section. We will revisit this in Section \ref{sec:minus-curves}.

\subsection{Knot-like curves} We will restrict to a subclass of decorated immersed curves in $\overline \sS_\bullet$ that contains those representing knot Floer invariants $\CFKrhat(K)$ for knots in $S^3$. We first require curves to be in almost simple position (Definition \ref{def:almost-simple-position}). In light of Theorem \ref{thm:curves-structure-theorem} we may restrict to decorated curves for which the bounding chain $\bchainhat$ has local system type. Furthermore, we will only consider multicurves with exactly one arc component and exactly one endpoint on each boundary of $\overline \sS_\bullet$. Finally, we will assume that the grading function on $\Gamma$ is normalized to be zero at the endpoints of the arc component. We remark that we could equivalently consider decorated curves in $\barT_\bullet$ that intersect the vertical line $\mu_{\frac 1 2}$ exactly once, and for which the grading function is zero at the intersection with $\mu_{\frac 1 2}$. While under these hypotheses curves in $\barT_\bullet$ and curves in $\overline \sS_\bullet$ are interchangeable, it is best to view our preferred class of curves as living in $\overline \sS_u$. We will also define a subclass of decorated immersed curves in $\overline \sS_u$ (or, equivalently, in $\barT_u$) that contains those representing knot Floer invariants $\CFKr(K)$ for knots in $S^3$; this class will be defined in terms of the mod $\sU$ reduction of the decorated curve, i.e. the same multicurve viewed as a curve in $\overline \sS_\bullet$ with the bounding chain $\bchain$ replaced by $\bchainhat = \bchain|_{\sI_0}$ .

\begin{definition}
A decorated immersed curve $(\Gamma, \bchainhat)$ in $\overline \sS_\bullet$ is a \emph{knot-like curve over $\sRhat$} if $\Gamma$ is in almost simple position, $\Gamma$ has exactly one arc component and intersects each boundary of the infinite strip once, the grading function evaluates to 0 at the endpoints of the arc component, and $\bchainhat$ has local system type. A decorated immersed curve $(\Gamma, \bchain)$ in $\overline \sS_u$ is a \emph{knot-like curve over $\sR$} if its mod $\sU$ reduction $(\Gamma, \bchainhat)$ is a knot-like curve over $\sRhat$. Let $\allcurves$ and $\allcurveshat$ denote the sets of knot-like curves over $\sR$ and $\sRhat$, respectively.
\end{definition}
 
Knot-like curves are called \emph{symmetric knot-like curves} if they are fixed, up to equivalence, by the rotation $(x,y) \to (-x,-y)$ in $\overline \sS_u$. The curves arising from knot Floer invariants will always be symmetric knot-like curves, but most of the discussion in this article does not require this symmetry.

The set $\allcurves$ (respectively $\allcurveshat$) of knot-like curves is a subcategory of the immersed Fukaya category of $\overline \sS_u$ (respectively of $\overline \sS_\bullet$). By abuse of notation we will use $\allcurves$ or $\allcurveshat$ both for the category and the set of objects. We are chiefly concerned with objects in $\allcurves$ or $\allcurveshat$ up to Floer equivalence.  Note that by Theorem \ref{thm:curves-structure-theorem} two elements of $\allcurveshat$ are only equivalent if the multicurves are homotopic and the bounding chains agree as subsets of self-intersection points. By considering the $\sU = 0$ restrictions, we see that equivalent knot-like curves in $\allcurves$ also have homotopic multicurves and identical restrictions of the bounding chain to $\sI_0$, but the equivalence for (the strictly positive degree portion of) $\bchain$ is more subtle.

At times it is necessary to project knot-like curves from $\overline \sS_u$ to $\sS_u$ by the projection map $p$; the projection of a curve in $\allcurves$, decorated with the projection of its grading function, will be called a knot-like curve in $\sS_u$ and the set of such curves will be denoted $p\allcurves$. In principle projecting loses information, but the grading function uniquely determines the correct lift of the curves to $\overline \sS_u$, so in fact projection gives a bijection from $\allcurves$ to $p\allcurves$. The curves in $p\allcurves$ are the objects of a category, which by abuse of notation we also call $p\allcurves$. By the above observation $\allcurves$ and $p\allcurves$ have the same objects, but the morphisms are different: given objects $C_1$ and $C_2$ in $\allcurves$, we must allow different vertical translations of $C_2$ relative to $C_1$ to recover all morphisms from $p(C_1)$ to $p(C_2)$; that is,
$$CF(p(C_1), p(C_2)) \cong \bigoplus_{s\in\Z} CF(C_1, C_2[s]),$$
where $C_2[s]$ is the result of translating $C_2$ up by $s$. The category $p\allcurveshat$ of knot-like curves in $\sS_\bullet$ is defined similarly.

\section{Relating knot-like complexes and knot-like curves}\label{sec:correspondence}

We are now ready to describe the relationship between knot-like complexes and knot-like curves. The following is a special case of the main result of \cite{Hanselman:CFK}:

\begin{theorem}\label{thm:bijection}
There are bijections
$$f: {\complexes}/{\sim} \to {\allcurves}/{\sim} \qquad \text{ and } \qquad \hat f: {\hatcomplexes}/{\sim} \to {\allcurveshat}/{\sim},$$
where on the left of each correspondence $\sim$ denotes chain homotopy equivalence and on the right $\sim$ denotes Floer-equivalence of decorated curves. Moreover, for any knot-like complexes $C_1$ and $C_2$ and any $s\in \Z$ there is an isomorphism
$$H_*(\Mor_\sR(C_1, C_2)_s) \cong HF(f(C_2), f(C_1)[s] ).$$
The direct sum over all Alexander gradings $s$ is obtained by applying the projection $p: \overline \sS_u \to \sS_u$, i.e.
$$H_*(\Mor_\sR(C_1, C_2)) \cong HF(p(f(C_2)), p(f(C_1))).$$
Analogous statements hold for $\hat f$ taking Floer homology in $\overline \sS_\bullet$ or $\sS_\bullet$ instead of $\overline \sS_u$ or $\sS_u$.
\end{theorem}

We remark that Theorem \ref{thm:bijection} suggests the following categorical equivalence:
\begin{conjecture}
The map $f$ induces an equivalence of the homology categories $H_*( \complexes)$ and $H_*( p \allcurves)^{op}$.
\end{conjecture}
This is a helpful philosophy to keep in mind, but strictly speaking Theorem \ref{thm:bijection} does not prove an equivalence of categories. To confirm such an equivalence, one would need to check that the correspondence between morphism homology spaces respects composition; we claim that it is straightforward to extend the argument in \cite{Hanselman:CFK} to identify compositions in $\Mor(C_1, C_2)$ with immersed triangles in an appropriate curve diagram, but we do not prove this. We also note that the proof that $H_*(\Mor_\sR(C_1, C_2)) \cong HF(p(f(C_2)), p(f(C_1)))$ relies on choosing a particular representative of an equivalence class of curves and should be extended to arbitrary objects.

Since a knot $K$ in $S^3$ determines an equivalence class of knot-like complex $\CFKr(K)$, by Theorem \ref{thm:bijection} it also determines a knot-like curve $f(\CFKr(K))$ over $\sR$; we let $\Gamma(K)$ denote this knot-like curve. That is, $\Gamma(K)$ is the decorated immersed curve in $\overline \sS_u$ representing the knot Floer complex of $K$. Note that $\Gamma(K)$ is an invariant of $K$ up to equivalence of knot-like curves. In particular the underlying multicurve is an invariant up to homotopy, and the bounding chain is invariant up to the equivalence discussed above.

In this section we sketch the main ideas of the correspondence in Theorem \ref{thm:bijection}. 

\subsection{From curves to complexes}
It is easier to define a map $g$ from ${\allcurves}/{\sim}$ to ${\complexes}/{\sim}$, which turns out to be the inverse of $f$. Given a knot-like curve $(\Gamma, \bchain)$ in $\allcurves$, we first view $(\Gamma, \bchain)$ as a decorated curve in $\overline \sS_{w,z}$ instead of $\overline \sS_u$ and then we define the complex $g((\Gamma, \bchain))$ to be the Floer chain complex $CF( (\Gamma, \bchain), \mu )$, where $\mu$ is the vertical line $\{0\}\times \R$ in $\overline \sS_{w,z}$ equipped with the constant bigrading function $(\gr_w, \gr_z) = (\frac 1 2, \frac 1 2)$. The result is clearly a bigraded complex over $\sR$, and clearly equivalent knot-like curves determine chain homotopic complexes. 

To check that the resulting complex $CF( (\Gamma, \bchain), \mu )$ is a knot-like complex, note that the vertical homology of the resulting complex  is obtained by ignoring the $z$ marked points, which means that up to equivalence we may slide $\mu$ leftward to the left edge of the strip. From this it is clear that the vertical homology has a single generator in grading $\lfloor \tau' - \tau \rfloor$, where $\tau = 0$ is the value of grading function on $\Gamma$ at both endpoints of its arc component and $\tau' = \frac 1 2$ is the grading on $\mu$. Similarly ignoring $w$ marked points and sliding $\mu$ to the right of the strip shows that the horizontal homology has a single generator in grading 0, as desired for a knot-like complex. 

Recall that the grading function $\tau$ on $\Gamma$ determines a bigrading $(\tau_w, \tau_z)$ when $\Gamma$ is viewed as a curve in $\overline \sS_{w,z}$. Note that generators of the Floer complex correspond to intersections of $\Gamma$ with $\mu$, and at each such intersection point $p_x$ with corresponding generator $x$, $\gr_w(x) = \lfloor \frac 1 2 - \tau_w(p_x)\rfloor = -\tau_w(p)$ where $\tau_w(p_x)$ agrees with the value of $\tau$ just to the left of $p_x$. Similarly, $\gr_z(x) = -\tau_z(p_x)$ where $\tau_z(p_x)$ is the value of $\tau$ just to the right of $p_x$. The Alexander grading of $x$ is half the difference $\gr_w(x) - \gr_z(x)$, or half of the jump in $\tau$ from the left side of $\mu$ to the right side of $\mu$ at $p_x$, which is simply (the nearest integer to) the height of $p_x$.

We can similarly define a map $\hat g$ from ${\allcurveshat}/{\sim}$ to ${\hatcomplexes}/{\sim}$ by taking the Floer chain complex of a knot-like curve over $\sRhat$ $(\Gamma, \widehat\bchain)$ modulo $\sU$ with $\mu$. This means that we ignore bigons that cover both $w$ and $z$ marked points. Combined with the assumption that $\Gamma$ intersects $\mu$ minimally, this implies we only need to consider bigons that lie entirely on the right side of $\mu$ (these contribute terms with some power of $W$) or entirely on the left side of $\mu$ (these contribute terms with some power of $Z$). If we ignore the local system decoration $\widehat\bchain$ then there is exactly one such bigon for each component of $\Gamma \setminus (\mu \cap \Gamma)$ that has both ends approaching $\mu$; if these segments lie on the right side of $\mu$ (we call such segments \emph{right segments} of $\Gamma$) they contribute horizontal arrows to the Floer complex and if they lie on the left of $\mu$ (such segments are called \emph{left segments} of $\Gamma$) they contribute vertical arrows. Because the intersection point $p$ has at most one contributing segment of $\Gamma$ on each side, it follows that when $\widehat\bchain = 0$ the construction of the Floer complex specifies a basis that is both horizontally and vertically simplified in the sense of \cite[Definition 11.23]{LOT:bordered}.

\subsection{From complexes to curves over $\sRhat$, given a simplified basis}

To define the map $\hat f$, given a knot-like complex $\widehat C$ over $\sRhat$ we wish to construct a knot-like curve $(\Gamma, \widehat\bchain)$ for which the Floer complex with $\mu$ recovers a complex homotopy equivalent to $\widehat C$. This turns out to be straightforward if we are given a basis for $\widehat C$ that is both horizontally and vertically simplified---we simply reverse the process above. For each generator $x$ we add a point $p_x$ on $\mu$ between $(0, A(x)-\frac 1 2)$ and $(0, A(x)+\frac 1 2)$, for each horizontal arrow from $x$ to $y$ in $\widehat C$ we add a segment on the right of $\mu$ connecting $p_x$ to $p_y$, and for each vertical arrow from $x$ to $y$ we add a segment connecting $p_x$ to $p_y$ on the left of $\mu$. Finally, for any point $p_x$ with no segment on the right we add a segment to the right of $\mu$ connecting $p_x$ to the point $(\frac 1 2, 0)$ and if there is no segment to the left of $p_x$ we add a segment on the left of $\mu$ connecting $p_x$ to $(-\frac 1 2, 0)$. The fact that the basis we used for $\widehat C$ was horizontally and vertically simplified implies that each point $p_x$ has exactly one segment attached on each side, so the result is an immersed multicurve in $\overline \sS_\bullet$. Since $\widehat C$ is a knot-like complex and has rank one horizontal and vertical homology the immersed multicurve must intersect each boundary of $\overline \sS_\bullet$ exactly once. We can check that there are no immersed monogons that do not cover a puncture in $\barT_*$, so $\Gamma$ is unobstructed on its own and we can set $\widehat\bchain = 0$. The grading function $\tau$ on $\Gamma$ is defined to agree with $-\gr_w$ just to the left of each intersection of $\Gamma$ with $\mu$ and to agree with $-\gr_z$ just to the right of $\mu$, and as such it must evaluate to 0 at the endpoints of the arc component. It is straightforward to check the multicurve is a knot-like curve in $\overline \sS_\bullet$ and the Floer chain complex of this curve with $\mu$ recovers the initial complex $\widehat C$ exactly.

In all of the examples in Section \ref{sec:examples} except Example \ref{ex:nontrivial-local-system} the complexes are given with respect to a horizontally and vertically simplified basis, so the above procedure determines the corresponding immersed multicurve $\Gamma$.

\subsection{From complexes to curves over $\sRhat$, general case}
The difficulty in the general case is that a knot-like complex $\widehat C$ over $\sRhat$ need not have a basis that is both horizontally and vertically simplified, and even if it does finding such a basis may be difficult. Without this nice basis constructing the correspondence $\hat f$ is more subtle. Given an arbitrary basis for $\widehat C$, attempting to follow the procedure above by adding a point $p_x$ for each generator $x$ and a left segment or right segment for each term in the differential produces an immersed train track. We can improve this to a train track that consists of an immersed multicurve with a collection of crossover arrows by choosing (separately) a horizontally simplified basis and a vertically simplified basis for $\widehat C$. We then draw the right side of the train track with respect to the horizontally simplified basis, we draw the left side with respect to the vertically generated basis, and we connect the two sides with a collection of horizontal arcs with crossover arrows between them determined by a change of basis between the two bases. One can check that this results in a train track of the desired form such that the Floer complex with $\mu$ recovers $\widehat C$ (see \cite[Section 7.1]{Hanselman:CFK}). Noting that a collection of crossover arrows on an immersed multicurve can be viewed as a bounding chain on the multicurve, possibly after applying a homotopy to the curve, we can then invoke Theorem \ref{thm:curves-structure-theorem} to find an equivalent decorated curve that is of local system type, and the result is a knot-like curve in $\allcurveshat$ that we define to be $\hat f(\widehat C)$. The uniqueness claim in Theorem \ref{thm:curves-structure-theorem} ensures that $\hat f(\widehat C)$ is well-defined (it depends only on the homotopy equivalence class of $\widehat C$) and that $\hat f \circ \hat g$ is the identity. Note that to actually compute $\hat f(\widehat C)$ requires applying the arrow sliding algorithm used in the proof of Theorem \ref{thm:curves-structure-theorem}; this is simple in small examples but computationally involved in general.

\subsection{From complexes to curves over $\sR$}\label{sec:minus-curves}

The construction of curves representing complexes over $\sR$ uses the $\sRhat$ case as a foundation. Given a knot-like complex $C$ over $\sR$, we first reduce mod $\sU$ to obtain a complex $\widehat C$ over $\sRhat$ and construct the corresponding knot-like curve $\hat f(\widehat C) = (\Gamma, \bchainhat)$ as described above. 

We fix the curve $\Gamma$ from here on and initially set $\bchain = \bchainhat$. We now consider the Floer complex of $(\Gamma, \bchain)$ with $\mu$ in $\overline \sS_{w,z}$. By assumption this recovers $\widehat C$ over $\sRhat$, but over $\sR$ it may not recover $C$ and in fact $CF((\Gamma, \bchain), \mu)$ may not even give a chain complex since $(\Gamma, \bchain)$ may not be unobstructed (we say still that $CF((\Gamma, \bchain), \mu)$ is a \emph{precomplex}, an object like a chain complex without the requirement that $\partial^2 = 0$). The strategy is to systematically modify $\bchain$ by adding or removing self-intersection points of strictly positive degree until the precomplex $CF( (\Gamma, \bchain), \mu)$ agrees with $C$. This can be done by considering one potential diagonal arrow at a time, and if $C$ and $CF((\Gamma, \bchain), \mu)$ disagree for this arrow it can be shown that a crossover arrow may be added that modifies $CF((\Gamma, \bchain), \mu)$ to agree with $C$ without affecting any arrows that have been considered previously. This crossover arrow can be placed at a self-intersection point (possibly after an arrow slide that may change the preferred basis of $C$) and thus interpreted as a new point in $\bchain$. This construction is described in more detail in \cite[Section 9.2]{Hanselman:CFK}. We remark that this construction requires repeatedly comparing the intermediate decorated curves $(\Gamma, \bchain)$, and the resulting precomplexes $CF((\Gamma, \bchain), \mu)$, to the complex $C$, and the geometric object may not correctly encode $C$ until the last step.

The procedure above produces an unobstructed decorated immersed curve $(\Gamma, \bchain)$ in $\overline \sS_u$ that represents $C$, in the sense that $g( (\Gamma, \bchain) ) = C$. Because the curve $\Gamma$ is not changed throughout the process of modifying $\bchain$ and the restriction $\bchainhat$ of $\bchain$ to $\sI_0$ is never modified, we have by construction that the mod $\sU$ reduction $(\Gamma, \bchainhat)$ is a knot-like curve over $\sRhat$. It follows that $(\Gamma, \bchain)$ is a knot-like curve over $\sR$. We define $f(C)$ to be the equivalence class of knot-like curve determined by $(\Gamma, \bchain)$. Recall that the multicurve $\Gamma$ in a representative of such an equivalence class is well-defined up to homotopy and that $\bchainhat$ is well-defined as a subset of the local system self-intersection points of $\Gamma$, but we caution that the subset $\bchain$ is not in general uniquely determined. To summarize, we have the following:

\begin{theorem}\label{thm:curves-structure-theorem-minus}
Any equivalence class of knot-like complex $C$ over $\sR$ can be represented by an equivalence class of knot-like curve $(\Gamma, \bchain)$ in $\overline \sS_u$. Moreover, for such a representative the mod $\sU$ reduction $(\Gamma, \bchainhat)$ is well-defined up to homotopy.
\end{theorem}

\subsection{Morphisms}\label{sec:morphisms}

To complete Theorem \ref{thm:bijection} we must show that the correspondence of objects defined by $f$ or $\hat f$ respects morphisms in the corresponding categories. Consider knot-like complexes $C$ and $C'$ with corresponding knot-like curves $f(C) = (\Gamma, \bchain)$ and $f(C') = (\Gamma', \bchain')$. The morphism space $\Mor_{\sR}(C, C')$ is a bigraded chain complex over $\sR$, but it splits as a direct sum over Alexander gradings
$$\Mor_{\sR}(C, C') = \bigoplus_{s\in\Z} \Mor_{\sR}(C, C')_s$$
and each summand $\Mor_{\sR}(C, C')_s$ may be viewed as a chain complex over $\F[\sU]$.

We claim that the space of Alexander grading preserving morphisms $\Mor_{\sR}(C, C')_0$ is chain homotopy equivalent to the Lagrangian Floer complex $CF( (\Gamma', \bchain'), (\Gamma, \bchain) )$. This can be seen by perturbing the curves into a particular position so that the Floer complex agrees with $\Mor_{\sR}(C, C')_0$ at the chain level. Figure \ref{fig:morphisms} shows the example of $\Mor_\sR(C, C')_0$ where $C = \CFKr(RHT)$ and $C' = \CFKr(LHT)$. Recall that when defining this Floer complex in $\overline \sS_u$ we must push the endpoints of $\Gamma$ downward on the left boundary of $\overline \sS_u$ and upward on the right boundary. We may arrange $\Gamma$ and $\Gamma'$ to intersect the strip $[-\frac 1 4, \frac 1 4]\times \R$ in a collection of horizontal segments and then perturb $\Gamma$ by pushing it downward on the left and upward on the right so that each perturbed horizontal segment of $\Gamma$ (corresponding to a generator of $C$) intersects each horizontal segment of $\Gamma'$ (corresponding to a generator of $C'$). Each resulting intersection point corresponds to the elementary morphism taking the relevant generator of $C$ to the relevant generator of $C'$ times an appropriate power of $W$ or $Z$ determined by the height difference between the two generators. Thus there is a bijection between generators of $CF( (\Gamma', \bchain'), (\Gamma, \bchain) )$ and generators of $\Mor_{\sR}(C, C')_0$. It is not difficult to check that the differentials also agree; see \cite[Proposition 10.1]{Hanselman:CFK}.
\begin{figure}
\includegraphics[scale = 1]{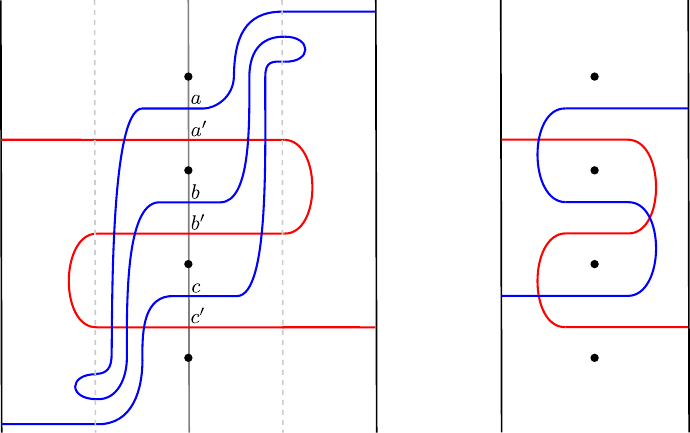}
\caption{When curves are arranged as on the left, the Floer chain complex of $(\Gamma', \bchain')$  (red) and $(\Gamma, \bchain)$ (blue) agrees with the complex $\Mor(C, C')$. Putting curves in minimal position as on the right simplifies the computation of $H_* \Mor(C, C')$.}
\label{fig:morphisms}
\end{figure}

Because shifting one curve upward or downward corresponds to shifting the Alexander grading of the corresponding complex, the summands $\Mor_\sR(C, C')_s$ correspond to the Floer complex of $(\Gamma', \bchain')$ with the result of translating $(\Gamma, \bchain)$ upward by $s$. To compute all summands at once and obtain the homology of $\Mor_\sR(C, C')$, we can project the curves to $\sS_u$ and take the (wrapped) Lagrangian Floer homology.

We remark that with curves perturbed as described above, the Floer complex $CF( (\Gamma', \bchain'), (\Gamma, \bchain) )$ can also be identified with the complex $C^\vee \otimes C'$, where $C^\vee$ is the dual complex to $C$. Since the curve representing $C^\vee$ is the vertical reflection of the curve representing $C$, we can also say that the complex $C \otimes C'$ is homotopy equivalent to the complex $CF( (\Gamma', \bchain'), m(\Gamma, \bchain) )$ where $m$ denotes the vertical mirror image of a decorated curve.

\section{Examples}\label{sec:examples}

In this selection we give several examples of the decorated immersed curve $\Gamma(K)$ for knots $K$ in $S^3$, as well as curves representing knot-like complexes over $\sR$ which may or may not be realized as invariants of knots in $S^3$.

\begin{figure}
\begin{tabular}{cccc}
$LHT$ & $4_1$ & $C_{2,-1}(LHT)$ & $C_{2,-1}(4_1)$\\
\toprule
$\begin{array}{l}
\gr(a) = (2,0) \\
\gr(b) = (1,1) \\
\gr(c) = (0,2) \end{array}$  & 
$\begin{array}{l}
\gr(a) = (1,-1) \\
\gr(b) = (0,0) \\
\gr(c) = (0,0) \\
\gr(d) = (0,0) \\
\gr(e) = (-1,1) \end{array}$ & 
$\begin{array}{l}
\gr(a) = (4,0) \\
\gr(b) = (3,1) \\
\gr(c) = (2,0) \\
\gr(d) = (1,1) \\
\gr(e) = (0,2) \\
\gr(f) = (1,3) \\
\gr(g) = (0,4)\end{array}$ & 
$\begin{array}{l}
\gr(a) = (3,-1) \\
\gr(b) = (2,0) \\
\gr(c) = (1,-1) \\
\gr(d) = (0,0) \\
\gr(e) = (0,0) \\
\gr(f) = (0,0) \\
\gr(g) = (-1,1) \\
\gr(h) = (0,2)\\
\gr(i) = (-1,3) \end{array}$\\

\midrule 
\\[-.7 em]
$\begin{array}{l}
\partial  a = Zb  \\
\partial b = 0 \\
\partial c = Wb \end{array}$ &
$\begin{array}{l}
\partial  a = Zc  \\
\partial b = Wa + Ze \\
\partial c = 0 \\
\partial d = 0 \\
\partial e = Wc \end{array}$  & 
$\begin{array}{l}
\partial a = Zb  \\
\partial b = 0 \\
\partial c = WZb + Z^2 f \\
\partial d = Wc + Ze \\
\partial e = W^2b + WZf \\
\partial f = 0 \\
\partial g = Wf \end{array}$ & 
$\begin{array}{l}
\partial a = Zb  \\
\partial b = 0 \\
\partial c = Z^2 h \\
\partial d = Wc + Z^2i\\
\partial e = 0\\
\partial f = W^2 a + Zg \\
\partial g = W^2b \\
\partial h = 0 \\
\partial i = Wh \end{array}$\\
\midrule 
\\[-.5 em]
\raisebox{8mm}{
\begin{tikzpicture}[scale = .8]
    \node (a) at (0,2) {$a$};
    \node (b) at (2,2) {$b$};
    \node (c) at (2,0) {$c$};

    \draw[->] (b) -- node[above] {$W$} (a);
    \draw[->] (b) -- node[right] {$Z$} (c);
\end{tikzpicture}} & 
\raisebox{6mm}{
\begin{tikzpicture}[scale = .8]
    \node (a) at (0,2) {$a$};
    \node (b) at (2,2) {$b$};
    \node (c) at (0,0) {$c$};
    \node (d) at (-.3,-.3) {$d$};
    \node (e) at (2,0) {$e$};

    \draw[->] (b) -- node[above] {$W$} (a);
    \draw[->] (b) -- node[right] {$Z$} (e); 
    \draw[->] (e) -- node[above] {$W$} (c);
    \draw[->] (a) -- node[right] {$Z$} (c); 
    \end{tikzpicture}} & 
\raisebox{0mm}{
\begin{tikzpicture}[scale = .8]
    \node (a) at (0,4) {$a$};
    \node (b) at (0,2) {$b$};
    \node (c) at (4,2) {$c$};
    \node (d) at (4,4) {$d$};
    \node (e) at (2,4) {$e$};
    \node (f) at (2,0) {$f$};
    \node (g) at (4,0) {$g$};

    \draw[->] (a) -- node[xshift = 6pt] {$Z$} (b); 
    \draw[->] (c) -- node[pos = .15, yshift = 6pt] {$W^2$} (b);
    \draw[->] (d) -- node[xshift = 6pt] {$Z$} (c);
    \draw[->] (d) -- node[above] {$W$} (e);
    \draw[->] (e) -- node[pos=.15, xshift = 7pt] {$Z^2$} (f);
    \draw[->] (g) -- node[above] {$W$} (f);
    \draw[->] (e) -- node[yshift = 15pt] {$WZ$} (b);
    \draw[->] (c) -- node[xshift = 15pt] {$WZ$} (f);

    \end{tikzpicture}}  & 
\raisebox{0mm}{
\begin{tikzpicture}[scale = .8]
    \node (a) at (0,4.2) {$a$};
    \node (b) at (0,2) {$b$};
    \node (c) at (2,3.8) {$c$};
    \node (d) at (3.8,3.8) {$d$};
    \node (e) at (0,0) {$e$};
    \node (f) at (4.2,4.2) {$f$};
    \node (g) at (4.2,2) {$g$};
    \node (h) at (2,0) {$h$};
    \node (i) at (3.8,0) {$i$};

    \draw[->] (a) -- node[xshift = 6pt] {$Z$} (b); 
    \draw[->] (f) -- node[xshift = 6pt] {$Z$} (g); 
    \draw[->] (f) -- node[yshift = 8pt] {$W^2$} (a); 
    \draw[->] (g) -- node[pos = .75, yshift = 6pt] {$W^2$} (b); 
    
    \draw[->] (d) -- node[below] {$W$} (c);
    \draw[->] (i) -- node[above] {$W$} (h);
    \draw[->] (c) -- node[pos = .75, xshift = -8pt] {$Z^2$} (h);
    \draw[->] (d) -- node[pos = .75, xshift = 7pt] {$Z^2$} (i);
    \end{tikzpicture}} \vspace{2mm} \\

\midrule
\\
\phantom{--} \raisebox{13.22mm}{\includegraphics{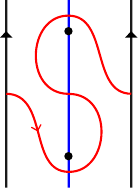}} \phantom{--} & 
\phantom{--} \raisebox{13.22mm}{\includegraphics{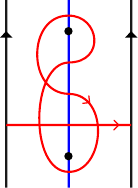}} \phantom{--} & 
\phantom{--} \includegraphics{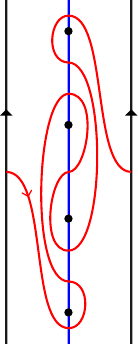} \phantom{--} & 
\phantom{--} \includegraphics{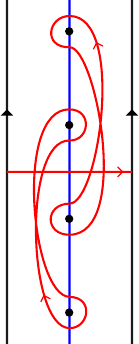} \phantom{--} \\ \bottomrule
\end{tabular}
\caption{Knot Floer complexes $\CFKr(K)$ and corresponding immersed curves $\Gamma(K)$ in the punctured strip $\barT_u$ for several knots: the left-hand trefoil, the figure-eight knot, and the $(2,-1)$ cables of these. The complex is drawn as a graph in the plane and the corresponding immersed curve is shown. In all four examples the bounding chain $\bchain$ is 0. The complexes can be read off the curves by taking Floer homology with the vertical line $\mu$, after passing from $\barT_u$ to $\barT_{w,z}$ by replacing each marked point with a pair of marked points. In particular, generators correspond to intersections with $\mu$ and should be labelled by letters $\{a,b,c,\ldots \}$ ordered from top to bottom. The value of the grading function at each intersection point is indicated in gray, and the orientation on the curves is determined by the grading function mod 2.}
\label{fig:first_examples}

\end{figure}

\begin{example}\label{ex:first-examples}
Figure \ref{fig:first_examples} shows knot Floer complexes and corresponding immersed curves for the left-hand trefoil ($LHT$), the figure-eight knot ($4_1$), and the $(2,-1)$ cable of each of these. The complexes are recorded with respect to bases that are both horizontally and vertically simplified, so it is straightforward to obtain the curves $\Gamma(K)$ representing $\CFKrhat(K)$: there is a point on $\mu$ for each generator, with height given by the grading $A = \frac{\gr_w - \gr_z}{2}$, and we add a right (resp. left) arc connecting points corresponding to $x$ and $y$ if $\partial x$ has a $W^k y$ term (resp. a $Z^k y$ term) for some $k$. It also follows from the fact that the bases are horizontally and vertically simplified that the local system decoration $\widehat\bchain$ is trivial. Since Floer homology of $\Gamma(K)$ with $\mu$ over $\sR$ correctly recovers the full complex $\CFKr$, we also have that $\bchain = 0$ in each case. This fact is also forced by the curves, since $\bchain$ is a linear combination of the positive degree self-intersection points of $\Gamma(K)$: for the trefoil and its cable there are no self-intersection points of $\Gamma(K)$, and for the figure-eight and its cable all self-intersection points have degree 0.
\end{example}

\begin{remark}\label{rmk:curve-notation} We note that instead of drawing pictures, the immersed curves $\Gamma(K)$ can be efficiently recorded by listing the (integer) heights of the sequence of points at which each component crosses $\mu$. For the single component of $\Gamma(K)$ that meets $\mu_{\frac 1 2}$ traverse the curve starting at $\mu_{\frac 1 2}$ following the orientation (which always moves rightward initially), but for all other components there is not a well-defined starting point so the resulting lists of integers are defined only up to cyclic permutation. We will use $[ \cdots ]$ for the list representing the distinguished component and $( \cdots )$ for the lists representing other components. Here we adopt the (arbitrary) convention to pick the largest of all cyclic reorderings of a list ordered lexicographically; in particular, we start traversing each component at its highest point. The grading function can be specified by recording its value at the intersection with $\mu$ corresponding to the first element in the list; we will denote this with a subscript. The first grading at the first intersection with $\mu$ on the distinguished component is always 0, so the subscript on that component can be omitted. With these shorthand notational conventions in place, the invariants in Example \ref{ex:first-examples} can be stated as follows:
\begin{align*}
\Gamma(LHT) &= [-1,0,1] \\
\Gamma(4_1) &= [0] \sqcup (1,0,-1,0)_1\\
\Gamma(C_{2,-1}(LHT)) &= [-2, -1, 1, 0, -1, 1, 2] \\
\Gamma(C_{2,-1}(4_1)) &= [0] \sqcup (2,1,-1,0)_1 \sqcup (1,-1,-2,0)_1
\end{align*}
It is not straightforward to represent the bounding chain decoration $\bchain$ using this shorthand for $\Gamma(K)$, but since $\bchain$ is very often trivial or determined by $\Gamma(K)$, the shorthand notation is useful in practice.
\end{remark}

For readers seeking more examples from knots, the immersed curves $\Gamma(K)$ for all prime knots with at most 16 crossings have been computed and are publicly available (with notation similar to that introduced in Remark \ref{rmk:curve-notation}) \cite{Hanselman:CFK-code}. It is also easy to obtain the curves for any cable from the curves for the companion as described in \cite{HW:cabling}. Connected sums and some other satellite operations are reasonable to compute as a source of more examples.

The examples so far have trivial bounding chain decoration. The next example will have $\bchain \neq 0$, though $\bchain$ is still uniquely determined by the curve $\Gamma(K)$. This example arises as one component of the invariant for a knot; we will ignore the other components for simplicity.

\begin{figure}
$\begin{array}{l}
\gr(a) = (2, -6) \\ [1 em]
\gr(b) = (1, -5) \\ [1 em]
\gr(c) = (0, -4) \\ [1 em]
\gr(d) = (1, -1) \\
\gr(e) = (-1, -3) \\ [1 em]
\gr(f) = (0, 0) \\
\gr(g) = (-2, -2) \\
\gr(h) = (0, 0) \\ [1 em]
\gr(i) = (-3, -1) \\
\gr(j) = (-1, 1) \\ [1 em]
\gr(k) = (-4, 0) \\ [1 em]
\gr(l) = (-5, 1) \\ [1 em]
\gr(m) = (-6, 2)
\end{array}$ \hspace{1 cm}
$\begin{array}{l}
\partial a = Z^3d \\ [1 em]
\partial b = Wa + Zc\\ [1 em]
\partial c = WZ^2 d\\ [1 em]
\partial d = 0\\
\partial e = Wc + WZ^2 f + Zg\\ [1 em]
\partial f = Wd \\
\partial g = 0\\
\partial h = Zj\\ [1 em]
\partial i = Wg + W^2 Z h + Z k\\
\partial j = 0\\ [1 em]
\partial k = W^2 Z j\\ [1 em]
\partial l = Wk + Zm\\ [1 em]
\partial m = W^3 j
\end{array}$ \hspace{1 cm}
\raisebox{-4cm}{\includegraphics{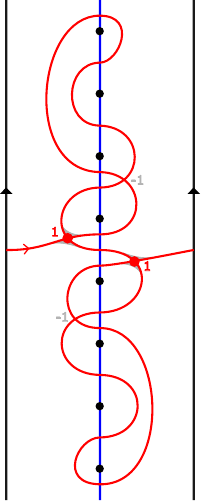}}
\caption{A component of the knot Floer complex $\CFKr(T_{2,9}\# - T_{2,3;2,5})$, along with the corresponding decorated curve $(\Gamma, \bchain)$. Intersections of $\Gamma$ with $\mu$ correspond to generators $a, \ldots, m$ from top to bottom. Each self-intersection point is labelled by its degree. The bounding chain decoration $\bchain$ is nontrivial, with the two indicated self-intersection points having nonzero coefficient. These points must be included in $\bchain$ to recover the $WZ^2f$ term in $\partial e$ and the $W^2Zh$ term in $\partial i$.}
\label{fig:nontrivial-bchain}
\end{figure}
\begin{example}\label{ex:nontrivial-bchain}
Figure \ref{fig:nontrivial-bchain} shows one direct summand of the knot Floer complex $\CFKr(K)$ where $K = T_{2,9}\# - T_{2,3;2,5}$, and the corresponding component of $\Gamma(K)$. The immersed curve has four self-intersection points, which are labelled in the diagram by their degree. Recall that the degrees are determined by the grading function: when making a left turn at a self-intersection consistent with the orientations, the grading function increases by $\theta + 2d$, where $\theta$ is the angle between the two tangent lines of the curve at the self-intersection point and $d$ is the degree. For example, just to the right of the point in $\Gamma \cap \mu$ corresponding to the generator $f$ the grading function has value $-\gr_z(f) = 0$ and just to the right of the point corresponding to $e$ the grading function has value $-\gr_z(e) = 3$; from this one can check that the higher self-intersection point on the right side of $\mu$ has degree $-1$. The bounding chain $\bchain$ is a subset of (equivalently, an $\F$-linear combination of) the two positive degree self-intersection points. In principle this gives four possibilities, but in fact only one is a valid bounding chain. This is because each degree $-1$ self-intersection point is the corner of a monogon covering two marked points, and in each case the only possible canceling monogon is a (generalized) monogon covering one marked point and making a turn at a degree 1 self-intersection point. Thus for $m_0^\bchain$ to vanish both degree 1 self-intersection points must be included in $\bchain$. One can check that when $\Gamma$ is decorated with this $\bchain$, the Floer complex with $\mu$ completely recovers the complex over $\sR$.
\end{example}

Example \ref{ex:nontrivial-bchain} has a nontrivial bounding chain decoration $\bchain$, but that decoration is uniquely determined by the curve $\Gamma(K)$. This does not always happen, as the next two examples show.

\begin{figure}
$\begin{array}{l}
\gr(a) = (2,0) \\
\\
\gr(b) = (1,1) \\
\gr(c) = (1,1) \\
\gr(d) = (0,0) \\
\\
\gr(e) = (0,2) \end{array}$  \hspace{1 cm}
\raisebox{-15 mm}{\includegraphics{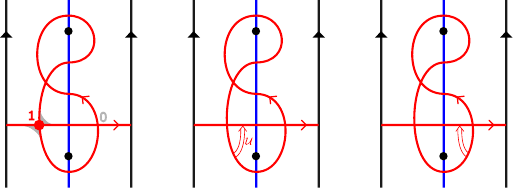}}
\caption{A multicurve with two distinct valid bounding chain decorations, $\bchain$ is either empty or contains the indicated intersection point. These correspond to two complexes that have the same $\sU = 0$ quotient but different diagonal arrows. The two choices of $\bchain$ are equivalent (the corresponding complexes are homotopy equivalent), which can be seen geometrically by interpreting the nontrivial choice of $\bchain$ as a crossover arrow and sliding the arrow over a marked point before removing it.}
\end{figure}
\begin{example}\label{ex:equivalent-bchains}
Consider a complex $C$ like $\CFKr(4_1)$ from Example \ref{ex:first-examples}, but with the bigrading of all generators but $d$ shifted up by $(1,1)$. Consider also $C'$ which is the same as $C$ but with a $WZ d$ term added to $\partial b$. $C$ and $C'$ agree mod $\sU$, and the multicurve $\Gamma$ representing either complex over $\sRhat$ is the same as for $4_1$ except that the grading function increases by one (and the orientation reverses) on the figure-eight component. Now $\Gamma$ has a self-intersection point $p$ of degree 1, and we may consider $\bchain = \emptyset$ or $\bchain' = \{p\}$. Since including $p$ does not introduce any generalized monogons, both choices are valid bounding chains. It is simple to see that $(\Gamma, \bchain)$ represents the complex $C$ while $(\Gamma, \bchain')$ represents the complex $C'$.

This shows that the bounding chain decoration, when viewed as a subset of the self-intersection points, is not uniquely determined by the immersed curve. However, in this example the two choices are in fact equivalent because the corresponding complexes $C$ and $C'$ are homotopy equivalent. In fact, $C'$ is obtained from $C$ by the change of basis replacing $e$ with $e + W d$. The equivalence of $\bchain$ and $\bchain'$ can be obtained by interpreting $(\Gamma, \bchain')$ as a train track with a crossover arrow near $p$ labeled by $\sU^{\deg(p)} = \sU$ as in the figure. We can then slide the crossover arrow across the marked point, decreasing the power of $\sU$ by one in the process, which corresponds to the change of basis stated above. Finally, after sliding past the other self-intersection point the (now unlabelled) crossover arrow moves left-to-right and can be removed.
\end{example}

Example \ref{ex:equivalent-bchains} highlights the subtlety in the bounding chain decoration. The decoration $\bchain$ is unique up to equivalence, but equivalence is a fairly abstract notion and some work is required to check if two different subsets of self-intersection points are equivalent bounding chain decorations. One might hope that there is a preferred simplest representative of each equivalent class of decorations $\bchain$; for instance, in Example \ref{ex:equivalent-bchains} the empty decoration is clearly simpler and it makes sense to declare that $\bchain$, and not $\bchain'$, is the standard decoration used to represent this equivalence class of complexes. If such a normal form for bounding chains exists, then two decorations on the same curve can be compared by ensuring both are in normal form and then simply comparing them as subsets of the self-intersection points. Unfortunately it is not yet clear if such normal form exists in general. A reasonable approach would be to interpret bounding chains as crossover arrows and removing them if possible by arrow sliding moves, showing that there is a unique arrangement in which all removable arrows have been removed, and this seems to work in practice, but sliding arrows in the minus setting can be subtle (see Figure \ref{fig:minus-arrow-slides}) and a complete arrow sliding calculus for curves in marked surfaces has not been worked out.

The next example shows that for some multicurves $\Gamma$ the bounding chain is not uniquely determined even up to equivalence. Thus it is possible for the decorated curve $(\Gamma, \bchain)$ to carry strictly more information than the curve $\Gamma$ alone. The example we give (taken from \cite[Proposition 9.4]{OSS:upsilon}) is an abstract complex, we do not know if this occurs in the knot Floer complex of any knot.

\begin{figure}
$\begin{array}{l}
\gr(a) = (2,-4) \\
\\
\gr(b) = (-3,-3) \\
\gr(c) = (1,1) \\
\gr(d) = (0,0) \\
\\
\gr(e) = (-4,2) \end{array}$  \hspace{1 cm}
$\begin{array}{l}
\partial a = Z^3 c \\
\\
\partial b = W^3 a + Z^3 d + x W^2 Z^2 e \\
\partial c = 0 \\
\partial d = W^3 c \\
\\
\partial e = y WZ c \end{array}$ \hspace{1 cm}
\raisebox{-30 mm}{\includegraphics{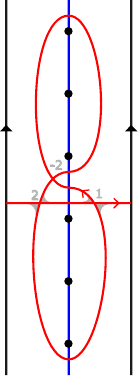}}
\caption{A knot-like complex depending on parameters $x, y \in \F$, at least one of which is 0, and the corresponding decorated immersed curve. The complexes agree over $\sRhat$, so the immersed curve $\Gamma$ is the same for each complex, but there are three (non-equivalent) choices of bounding chain decoration based on $x$ and $y$. In particular $\bchain$ includes the degree 2 self-intersection point when $x = 1$ and the degree 1 self-intersection point when $y=1$.}
\label{fig:different-bchains}
\end{figure}
\begin{example}\label{ex:different-bchains}
Figure \ref{fig:different-bchains} shows a complex which can include at most one of two optional terms in the differential (including both would violate $\partial^2 = 0$), resulting in three distinct complexes over $\sR$ that all agree mod $\sU$. The multicurve $\Gamma$ representing the $\sRhat$ quotient of the complex is also shown. The curve has two positive degree self-intersection points, one with degree 2 and the other with degree 1, and representing the complex over $\sR$ requires decorating the curve with a subset of these two points. Including the first corresponds to including the term with coefficient $x$ in the differential, and including the other corresponds to including the term with coefficient $y$; note that including both points in $\bchain$ is not allowed since this would create a generalized monogon with its corner at the degree $-2$ self-intersection point.

The three choices of bounding chain are not equivalent because the corresponding complexes over $\sR$ are not homotopy equivalent, a fact that can be checked by extracting certain numerical invariants from each complex (see Section \ref{sec:invariants-from-v_s}). Note that we could represent the bounding chain as a crossover arrow and attempt to remove it as in Example \ref{ex:equivalent-bchains}, but since arrows can only slide past a puncture on their right at the expense of decreasing the power of $\sU$ on the arrow and negative powers are not allowed, the arrow is not able to slide past all three punctures and is thus not removable.
\end{example}

All examples so far have trivial local system decoration $\widehat\bchain$---we will next consider an example of a knot-like complex with nontrivial $\widehat\bchain$. We note that this example is an abstract knot-like complex, it is currently unknown whether there is a knot for which the associated immersed curve carries a nontrivial local system decoration.

\begin{figure}
$\begin{array}{l}
\gr(a) = (2, -2) \\
\gr(b) = (2, -2) \\
\\
\\
\gr(c) = (-1,-1 ) \\
\gr(d) = (-1, -1) \\
\gr(e) = (1, 1) \\
\gr(f) = (1, 1) \\
\gr(g) = (0, 0) \\
\\
\\
\gr(h) = (-2, 2) \\
\gr(i) = (-2, 2) \\
\end{array}$  \hspace{1 cm}
$\begin{array}{l}
\partial a = Z^2 f \\
\partial b = Z^2 e \\
\\
\\
\partial c = W^2 b + Z^2 i \\
\partial d = W^2 a + Z^2 h \\
\partial e = 0 \\
\partial f = 0 \\
\partial g = 0 \\
\\
\\
\partial h = W^2 e + W^2 f\\
\partial i = W^2 f
 \end{array}$ \hspace{1 cm}
\raisebox{-35 mm}{\includegraphics{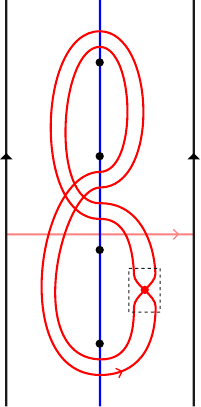}}
\caption{A knot-like complex over $\sRhat$ for which the corresponding immersed curve has nontrivial local system decoration. The self-intersection point in the dotted rectangle is the local system self-intersection point for the non-primitive curve component, and it is included in $\widehat\bchain$.}
\label{fig:local-system-hat}
\end{figure}

\begin{figure}
$\begin{array}{l}
\gr(a) = (3, -1) \\
\gr(b) = (3, -1) \\
\gr(c) = (2, -2) \\
\gr(d) = (2, -2) \\
\\
\\
\gr(e) = (-1, -1) \\
\gr(f) = (-1, -1) \\
\gr(g) = (0, 0) \\
\gr(h) = (0, 0) \\
\gr(i) = (1, 1) \\
\gr(j) = (1, 1) \\
\gr(k) = (2, 2) \\
\gr(l) = (2, 2) \\
\gr(m) = (0, 0) \\
\\
\\
\gr(n) = (-1, 3) \\
\gr(o) = (-1, 3) \\
\gr(p) = (-2, 2) \\
\gr(q) = (-2, 2) \end{array}$  \hspace{5 mm}
$\begin{array}{l}
\partial a = Z^2 l \\
\partial b = Z^2 k \\
\partial c = Z^2 j + WZa \\
\partial d = Z^2 i + WZb \\
\\
\\
\partial e = W^2 d + WZg + Z^2 o \\
\partial f = W^2 c + WZ h + Z^2 n \\
\partial g = W^2 b + WZi + WZj + Z^2 q \\
\partial h = W^2 a + WZi + Z^2 p\\
\partial i = WZ k\\
\partial j = WZ l\\
\partial k = 0 \\
\partial l = 0 \\
\partial m = 0\\
\\
\\
\partial n = W^2 i + W^2 j + WZ p \\
\partial o = W^2 j + WZq \\
\partial p = W^2 k + W^2 l \\
\partial q = W^2 l 
 \end{array}$ \hspace{5 mm}
\raisebox{-42 mm}{\includegraphics{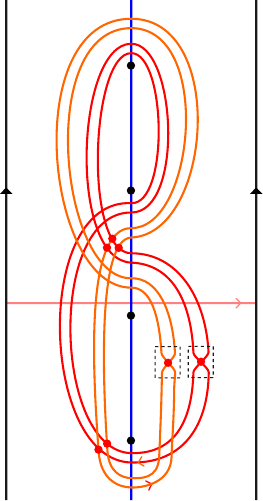}}
\caption{A knot-like complex over $\sR$ and corresponding immersed curve with nontrivial local system decoration. The self-intersection points in the dotted rectangles are the local system self-intersection points for the two non-primitive curve components, and both are included in $\widehat\bchain$; these as always have degree 0. The other self-intersection points included in $\bchain$ all have degree 1.}
\label{fig:local-system}
\end{figure}
\begin{example}\label{ex:nontrivial-local-system}
Figure \ref{fig:local-system-hat} gives a complex over $\sRhat$ for which the corresponding curve has $\widehat\bchain \neq \emptyset$. The complex does not have a horizontally simplified basis, as there are two horizontal arrows starting at $h$. Note that one component of the immersed curve is not primitive, it is homotopic to two copies of a figure-eight shaped primitive curve. The self-intersection point in the dotted box is a local system intersection point, and by the normal form for decorated immersed curves this is the only point that may be included in a bounding chain decoration $\bchainhat$ over $\sRhat$. The term $W^2 f$ in $\partial h$ forces this point to be included in $\bchainhat$, and the same immersed curve decorated with the trivial bounding chain $\bchainhat$ would produce this complex with the $W^2 f$ term removed from $\partial h$. Since a complex with a horizontally and vertically simplified basis immediately gives rise to a curve with trivial local system over $\sRhat$, the fact that this local system decoration is nontrivial means that this complex does not admit a basis that is both horizontally and vertically simplified.

This complex over $\sRhat$ on its own is not the $\sU=0$ restriction of a knot-like complex over $\sR$, as there is no way to add diagonal arrows such that $\partial^2 = 0$ when working over $\sR$. However, it is possible to construct a complex over $\sR$ whose $\sU=0$ reduction has this complex as a direct summand. Such a complex, taken from Example P in \cite{Popovic}, is shown in Figure \ref{fig:local-system}. The complex mod $\sU$ determines the three curve components and the decoration $\bchainhat = \bchain|_{\sU = 0}$, which consists of the local system intersection point for each non-primitive component. Note that since the two non-primitive curves are homotopic to each other and thus bound an immersed annulus in minimal position, we must introduce extra self-intersection points to put the curve in almost simple position and guarantee that the minus complex can be captured by choosing a bounding chain $\bchain$. One can check that including the remaining self-intersection points shown in the figure each of which is degree 1, along with the points in $\bchainhat$ represents the complex correctly over $\sR$.
\end{example}

Our final example is an interesting knot-like complex described by Popovi{\'c}, who gave some conditions on the algebraic realizability of certain components in knot-like complexes \cite{Popovic}. Once again it is unknown whether the knot Floer complex for any knot in $S^3$ has this as a summand.
\begin{figure}
$\begin{array}{l}
\partial a = Z^4 g \\
\partial b = WZc + Z^2f \\
\partial c = WZa + Z^2d \\
\\
\\
\partial d = WZ^3g \\
\partial e = W^2c + WZf + W^2 Z^4 g + Z^3j \\
\partial f = W^2 a + WZe \\ 
\\
\\
\partial g = 0\\
\partial h = W^4b + W^3Ze + WZ^4k + Z^5m \\
\\
\partial i = WZj + Z^5o \\
\partial j = 0 \\
\\
\\
\partial k = W^2 j + Z^2 l \\
\\
\\
\partial l = 0 \\
\partial m = W^5 g + WZ l \\
\\
\partial n = W^5i + W^4Zk + WZ^3q + Z^4t \\
\partial o = 0 \\
\\
\\
\partial p = WZ r + Z^2u\\
\partial q = W^3 l + W^4Z^2 o + WZp + Z^2s \\
\partial r = W^3 Z o \\
\\
\\
\partial s = W^2 r + WZu \\
\partial t = W^2 p + WZs \\
\partial u = W^4 o
\end{array}$\hspace{1 cm}
\raisebox{-78 mm}{\includegraphics{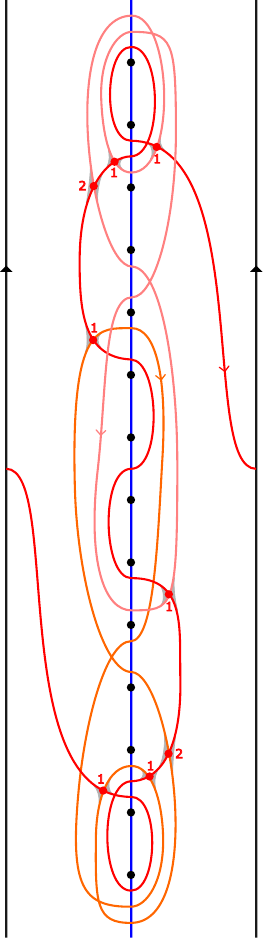}}
\caption{A knot-like complex and corresponding decorated immersed multicurve, which has three components. Self-intersection points are labeled with their degree, and those included in the subset $\bchain$ are indicated by red dots.}
\label{fig:popovic}
\end{figure}
\begin{example}\label{ex:popovic}
The complex in Figure \ref{fig:popovic} is Example D in \cite{Popovic}. As the complex has a horizontally and vertically simplified basis it is simple to construct the immersed multicurve $\Gamma$, which has three components, and note that $\bchainhat = 0$. We then systematically consider diagonal arrows, in order of increasing length, and add self-intersection points to $\bchain$ as needed to capture these diagonal arrows. We leave it as an exercise to see that when $\bchain$ is the collection of six degree 1 self-intersection points shown in the figure the decorated curve $(\Gamma, \bchain)$ represents the correct complex. Gradings are not indicated in the figure, but the grading function on $\Gamma$ is determined once the degrees of the indicated intersection points are known.

An interesting feature of this complex is that each generator is connected by a sequence of arrows to $\sU$ times itself, and as a result when viewed as a complex over $\F$ (with generators $W^a Z^b x$ for $a,b \ge 0$ and $x$ a generator of the complex over $\sR$) each Alexander grading gives a single infinitely generated complex that does not decompose as a direct sum. In contrast, for the complexes in Example \ref{ex:first-examples} restricting to a single Alexander grading and viewing as complex over $\F$ yields a direct sum of infinitely many finitely generated complexes. Moreover, after ignoring finitely many summands with low powers of $W$ and $Z$ the complex can be written $\F[\sU] \otimes C$ for a finitely generated complex $C$. There are no known examples of knots for which the knot Floer complex does not split in this way. The fact that the complex in this example does not split into finitely generated complexes over $\F$ corresponds to a geometric feature of the corresponding curve: the sequence of (forward or backward) horizontal and vertical arrows that comes from following one full loop of the orange curve component $\gamma_1$ connects a generator $x$ to $\sU^{\wind(\gamma_1)}$ times $x$, so the complex does not split because the total winding number $\wind(\gamma_1)$ of $\gamma_1$ around the punctures is nonzero. Note that since the curve is $\Z$-gradable this is equivalent to the net rotation $\rot(\gamma_1)$ around the curve being nonzero.
\end{example}

\section{Extracting simple invariants from $\Gamma(K)$}\label{sec:simple-invariants}

Although at times it makes sense to work with the full knot Floer complex $\CFKr(K)$, or equivalently the corresponding knot-like curve $\Gamma(K)$, this is often more information than required to distinguish two knots or identify a desired feature of a knot, and it can be conceptually or computationally advantageous to focus on some partial information extracted from the full invariant. The rest of the article will be devoted to various simpler invariants that can be defined or recovered using $\Gamma(K)$. In this section we discuss how some classical knot invariants are detected by $\Gamma(K)$, namely the Seifert genus, fiberedness, and the Alexander polynomial, and we also discuss some invariants defined using knot Floer homology that are particularly simple to extract from  the immersed curve $\Gamma(K)$. More sophisticated invariants will be discussed in Section \ref{sec:concordance-invariants}, but these require more technical background covered in Section \ref{sec:subcomplexes} and Section \ref{sec:distinguished-components}.

\subsection{Genus and fiberedness} A well known result of Ozsv{\'a}th and Szab{\'o} is that the knot Floer homology of a knot $K$ detects the Seifert genus $g(K)$ \cite{OzSz:genus-bounds}. In particular, it was shown that $g(K)$ is the maximal Alexander grading in which $\HFKhat(K;s)$ is nonzero, where $\HFKhat(K)$ is the homology of the finitely generated chain complex over $\F$ obtained from $\CFKr(K)$ by setting $W = Z = 0$ and $\HFKhat(K;s$ is the summand in Alexander grading $s$. This chain complex is the Floer complex of $\widehat\Gamma(K)$ with $\mu$ in the doubly punctured strip $\overline \sS_{\bullet, \bullet}$ obtained from $\overline \sS_{w,z}$ by removing all basepoints, and the Alexander grading of a generator is given by its height. Since we assume $\widehat\Gamma(K)$ intersects $\mu$ minimally there are no bigons and $\HFKhat(K)$ is generated by the intersection points of $\widehat\Gamma(K)$ with $\mu$. It follows that $g(K)$ is simply the maximal height at which the immersed curve $\widehat\Gamma(K)$ intersects $\mu$.

It was also shown by Ghiggini \cite{Ghiggini:detects-fiberedness} and Ni \cite{Ni:detects-fiberedness} that knot Floer homology detects whether or not a knot $K$ is fibered. More specifically, $K$ is fibered if and only if the dimension of $\HFKhat(K)$ restricted to the top Alexander grading $A = g(K)$ is 1. From the discussion above, it is clear that $K$ is fibered if and only if the immersed curve $\widehat\Gamma(K)$ attains its maximum height exactly once.

\subsection{Alexander polynomial}
Knot Floer homology is a categorification of the Alexander polynomial, where the Alexander polynomial $\Delta_K(t)$ can be recovered from $\HFKhat(K)$ by taking a graded Euler characteristic. More precisely, we have
$$\Delta_K(t) = \sum_{s \in \Z} \chi(\HFKhat(K;s)) t^s,$$
where $\chi$ is the Euler characteristic with respect to the mod 2 grading. Recall that generators of $\HFKhat(K;s)$ correspond to intersections of $\widehat \Gamma$ with the segment of $\mu$ between $(0,s-\frac 1 2)$ and $(0,s+\frac 1 2)$, and mod 2 grading corresponds to the sign of the intersection point (using the orientation on $\widehat \Gamma(K)$ determined by its grading function and orienting $\mu$ upwards). Thus the coefficient of $t^s$ in $\Delta_K(t)$ is simply the algebraic intersection number of $\widehat\Gamma(K)$ with the segment of $\mu$ at height $s$.

\subsection{The distinguished curve component $\gamma_0$ and related invariants} \label{sec:gamma_0-invariants}
In the decorated multicurve $\Gamma(K)$, the underlying immersed multicurve $\Gamma$ is a collection $\{\gamma_0, \ldots, \gamma_n\}$ of compact immersed curves in $\overline \sS_\bullet$. By the definition of a knot-like curve it has exactly one arc component and by convention we will always label components so that this is $\gamma_0$. It follows that the homotopy class of the arc $\gamma_0$ is an invariant of the knot $K$. We will say more about this distinguished curve component in Section \ref{sec:distinguished-components}, where we will see that it is in fact a concordance invariant. For now we will explain how several numerical invariants can be easily extracted from the shape of $\gamma_0$. In particular, we have the following:

\begin{itemize}
\item the integer $\tau(C)$ is the height of the first intersection of $\gamma_0$ with $\mu$. Here as usual we interpret height in a discrete way---that is, the height of a point on $\mu$ in $\overline \sS_\bullet$ is $n$ if it lies between marked points at $(0,n-\frac 1 2)$ and $(0,n+\frac 1 2)$. By first intersection we mean the first one reached traveling along $\gamma_0$ starting from its intersection with $\mu_{\frac 1 2}$ moving rightward.

\item the invariant $\epsilon(C)$  takes a value in $\{-1, 0, 1\}$ and records what the curve does immediately after its first intersection with $\mu$: if $\gamma_0$ turns downwards (i.e. the second intersection with $\mu$ occurs lower than the first) then $\epsilon(\gamma_0) = 1$, if $\gamma_0$ turns upwards (i.e. the second intersection with $\mu$ occurs higher than the first) then $\epsilon(\gamma_0) = -1$, and if the $\gamma_0$ continues straight (i.e. there is no second intersection with $\mu$) then $\epsilon(\gamma_0) = 0$. Note that for symmetric knot-like complexes, which includes knot Floer complexes, $\epsilon$ can only be 0 if $\tau = 0$.

\item the integer $\nu(C)$ is closely related to $\tau(C)$ and also measures the height at which $\gamma_0$ first intersects $\mu$ but in a different way. The first intersection occurs somewhere between $(0, \tau(C) - \frac 1 2)$ and $(0, \tau(C)+\frac 1 2)$, but if $\gamma_0$ is ``pulled tight" the intersection will lie near $(0, \tau(C)-\frac 1 2)$ if $\epsilon(C) = 1$ and near $(0, \tau(C)+\frac 1 2)$ if $\epsilon(C) = -1$. If $\epsilon(C) = 0$, so that $\gamma_0$ is homotopic to a horizontal curve, we adopt the convention that $\gamma_0$ should be pushed downward to cross $\mu$ near $(0, \tau(C)-\frac 1 2)$. We then define $\nu(C)$ to be $\frac 1 2$ plus the height of the puncture nearest the first intersection of $\gamma_0$ with $\mu$. In other words, $\nu(C) = \tau(C)$ if $\epsilon(C) \in \{0,1\}$ and $\nu(C) = \tau(C)+1$ otherwise.

\item for each positive integer $i$, the integer valued invariant $\phi_i(C)$ is the signed count of right-segments of $\gamma_0$ of length $i$ (that is, segments of $\gamma_0$ between successive intersections with $\mu$ at heights differing by $i$). Right-segments count positively if they are oriented downwards and negatively if they are oriented upwards.
\end{itemize}
The knot invariant $\tau$ was defined by Ozsv{\'a}th and Szab{\'o} in \cite{OzSz:4ball-genus}, $\epsilon$ was defined by Hom in \cite{Hom:CFK-and-smooth-concordance}, $\nu$ was defined by Ozsv{\'a}th and Szab{\'o} in \cite{OzSz:rational-surgeries}, and $\phi_i$ was defined by Dai, Hom, Stoffregen and Truong  in \cite{DHST:more-concordance}. Each invariant was defined using the algebraic structure of $\CFKr$, but they are indeed equivalent to the definitions given above using $\gamma_0$.

\begin{proposition}
The definitions of $\tau$, $\epsilon$, $\nu$, and $\phi_i$ above using $\gamma_0$ recover the corresponding knot invariants.
\end{proposition}
\begin{proof}
Given a knot $K$ let $C = \CFKr(K)$.  $\tau(K)$ can be defined as the minimal $s$ for which the inclusion $(C_s) |_{W = 0} \to (C^-_s) |_{W = 0}$ induces a surjection on homology. If $C$ has a vertically simplified basis this is simply the Alexander grading of the generator of vertical homology. The immersed curve $(\Gamma, \bchain)$ determines such a basis for $C$, and the generator of vertical homology corresponds to the intersection of $\Gamma$ with $\mu$ that is connected by a segment on its left to $\mu_{\frac 1 2}$. This is precisely the first intersection of $\gamma_0$ with $\mu$, and its Alexander grading is its height.

$\epsilon(K)$ has a simple description when $\widehat C$ has a basis that is both horizontally and vertically simplified, as described in \cite[Section 3.2.1]{Hom:survey}. Since the immersed curve determines a basis for $\widehat C$ that is horizontally and vertically simplified when restricted to the summand corresponding to $\gamma_0$, this description applies. Thus $\epsilon(K) = 1$ if the generator of vertical homology is the end of a horizontal arrow, or equivalently if the first intersection of $\gamma_0$ with $\mu$ is the top of a right-segment in $\gamma_0$, meaning that $\gamma_0$ bends downward at that point. Similarly $\epsilon(K) = -1$ if the generator of vertical homology is the start of a horizontal arrow, or equivalently if the first intersection of $\gamma_0$ with $\mu$ is the bottom of a right-segment in $\gamma_0$. If neither of these is true then $\epsilon(K)=0$ and $\gamma_0$ only intersects $\mu$ once and continues rightward after this intersection.

The invariant $\nu(K)$ satisfies the relationship with $\tau$ and $\epsilon$ noted above, according to \cite[Definition 3.5]{Hom:survey} so clearly matches $\nu(C)$ defined above. Alternatively, we can compare to the original definition of $\nu$ as the minimum $s$ for which $\hat v_{s,*}$ is surjective, where $\hat v_{s,*}$ is the inclusion map of a subcomplex that will be discussed in  Section \ref{sec:v_s-maps}; the geometric description of $\nu$ will also follow from the geometric description of those maps.

The invariants $\phi_i$ were defined using the standard complex representative of $C$, which corresponds to the distinguished component $\gamma_0$ as discussed in Remark \ref{rmk:standard_complex}. Under this identification it is clear that $\phi_i(C)$ counts oriented right-segments as claimed (or equivalently, since knot Floer complexes are symmetric, oriented left-segments).
\end{proof}

Recall that $\gamma_0$ can be concisely recorded as a list of integers (see Remark \ref{rmk:curve-notation}). With this notation $\tau$ is simply the first integer in the list, $\epsilon$ is 0 if the list has only one element and otherwise records whether the second entry is greater than or less than the first, and $\phi_i$ is half the (signed) count of terms at which the sequence decreases (or increases, counted with a minus sign) by $i$. The table in Section \ref{sec:simple-invariants-examples} lists these invariants for several examples.

Even though $\gamma_0$ captures all of the other invariants, each of these numerical invariants has properties that make them valuable on their own. For example, $\tau(K)$ is additive under connected sum, so it gives a homomorphism from the smooth concordance group to $\Z$ \cite[Theorem 1.2]{OzSz:4ball-genus}. It also gives a lower bound on the slice genus and the unknotting number of $K$ \cite[Corollary 1.3]{OzSz:4ball-genus}. The invariants $\phi_i$ also give concordance homomorphisms and can be used to identify a $\Z^\infty$ summand in the subgroup of the smooth concordance group generated by topologically slice knots \cite{DHST:more-concordance}. The invariant $\epsilon$ is related to a total ordering on the set of connected symmetric knot-like curves in $\overline \sS_\bullet$ (that is, the set of possible $\gamma_0$ curves); $\epsilon$ is $1$, $0$, or $-1$ corresponds to the curve being less than, equal to, or greater than the trivial curve with respect to this ordering. Geometrically, $\gamma_0 < \gamma_0'$ if the $\gamma_0$ is not homotopic to $\gamma'$ and the first time the curves diverge $\gamma_0$ lies to the left of $\gamma'$. The ordering on $\gamma_0$ curves determined by $\epsilon$ will be discussed further in Section \ref{sec:distinguished-components}. The invariant $\nu$, which can also be defined using the maps $\hat v_s$ discussed in Section \ref{sec:subcomplexes}, has a natural generalization using the maps $v_s$ that we will discuss in \ref{sec:concordance-invariants}.

\subsection{Torsion order}\label{sec:torsion-order}

The torsion order of knot Floer homology was introduced in \cite{AlishahiEftekhary:torsion} and \cite{JuhaszMillerZemke:torsion}. It is defined by considering $\HFKminus(K)$, which in our notation is the homology of the complex $C_{Z=0}$ obtained from $\CFKr$ by setting $Z = 0$. This homology, a module over $\F[W]$, splits as a copy of $\F[W]$ plus torsion summands. The torsion order is the minimum power of $W$ that kills all elements in the torsion summands, i.e.
$$\Ord(K) = \min\{ n \ge 0 | W^n \cdot \Tor( H_*(C_{Z=0} )) = 0 \}.$$

Consider the decorated curve $\Gamma(K)$ representing $\CFKr$. By homotoping the curve if necessary, we may assume that any local system intersection points lie to the left of $\mu$, so that the complex determined by $\Gamma(K)$ is naturally equipped with a horizontally simplified basis. The generators can be labelled $\{x_0, \ldots, x_{2n}\}$ such that $x_0$ is the generator of horizontal homology and for each $i$ the generators $x_{2i-1}$ and $x_{2i}$ correspond to points on $\mu$ connected by a right arc of $\Gamma(K)$ of some length $k_i$. The differential in $C_{Z=0}$ is given by $\partial(x_{2i-1}) = W^{k_i} x_{2i}$ and $\partial(x_j) = 0$ for $j$ even, and thus
$$H_*(C_{Z=0} ) \cong \F[W] \oplus \F[W]/W^{k_1} \oplus \cdots \oplus \F[W]/W^{k_n}.$$
Clearly the torsion order is the maximum of the integers $k_i$.

\begin{proposition}
The torsion order $\Ord(K)$ is the maximum length of a right arc in $\Gamma(K)$. \qed
\end{proposition}

It was shown in \cite{AlishahiEftekhary:torsion} that $\Ord(K)$ gives a lower bound for the unknotting number of $K$. The torsion order is also involved in some bounds related to ribbon concordance and ribbon cobordisms \cite{JuhaszMillerZemke:torsion}.

\subsection{Examples}\label{sec:simple-invariants-examples}

The following table gives the values of the invariants from Sections \ref{sec:gamma_0-invariants} and Sections \ref{sec:torsion-order} for several of the example curves introduced in Section \ref{sec:examples}.
\begin{center}
\begin{tabular}{>{$}l<{$}  >{$}c<{$}   >{$}c<{$}  >{$}c<{$}  >{$}c<{$}   >{$}c<{$}   >{$}c<{$}  >{$}c<{$} >{$}c<{$}  }
\toprule \\[-.8em]
 & \gamma_0 & \tau & \epsilon & \nu & \phi_1 & \phi_2 & \phi_3 & \Ord \\ [.5 em]
 \midrule \\[-.7em]
\text{unknot} & [0] & 0 & 0 & 0 & 0 & 0 & 0 & 0\\
4_1 & [0] & 0 & 0 & 0 & 0 & 0 & 0 & 1\\
C_{2,-1}(4_1) & [0] & 0 & 0 & 0 & 0 & 0 & 0 & 2\\
RHT & [1,0,-1] & 1 & 1 & 1 & 1 & 0 & 0 & 1\\
LHT & [-1,0,1] & -1 & -1 & 0 & -1 & 0 & 0 & 1\\
C_{2,-1}(LHT) & [-2,-1,1,0,-1,1,2] & -2 & -1 & -1 & 0 & -1 & 0 & 2\\
C \text{ from Figure } \ref{fig:nontrivial-bchain} & [0,1,4,3,2,1,0,-1,-2,-3,-4,-1,0] & 0 & -1 & 1 & 3 & 0 & -1 & 3 \\
C \text{ from Figure } \ref{fig:popovic} & [-5, -7, -5, -2, 0, 2, 5, 7, 5] & -5 & 1 & -5 & 0 & -1 & -1 & 5 \\ [.3 em]
\bottomrule
\end{tabular}
\end{center}

\section{Subcomplexes, inclusion maps, and the surgery formula} \label{sec:subcomplexes}

Other invariants can be extracted from the knot Floer immersed curves, but they require more work. In particular, many involve considering certain subcomplexes of $\CFKr(K)$ and the inclusion maps of these complexes into $\CFKr(K)$. In this section we develop language for the relevant subcomplex and inclusion maps and interpret them in terms of immersed curves. This technical setup will be needed to define some of the invariants discussed in Section \ref{sec:concordance-invariants}. It also plays a key role in the Dehn surgery formula, which we discuss in Section \ref{sec:surgery-formula}. The reader may wish to skip Section \ref{sec:other-inclusion-maps} on a first reading, as it applies to less common generalizations of the usual subcomplexes covered in Section \ref{sec:v_s-maps}.

\begin{figure}
\begin{tikzpicture}
    \fill[color = gray, fill opacity = .2] (-2.75,-2.75) rectangle (.5, .5);
    \fill[color = gray, fill opacity = .3] (-2.75,-2.75) rectangle (-.5, -.5);
    \shade[left color = gray!0, right color = gray!20, draw= none] (-3, -.5) rectangle (-2.75, .5);
    \shade[left color = gray!0, right color = gray!50, draw= none] (-3, -2.75) rectangle (-2.75, -.5);
    \shade[top color = gray!50, bottom color = gray!0, draw= none] (-2.75, -3) rectangle (-.5, -2.75);
    \shade[top color = gray!20, bottom color = gray!0, draw= none] (-.5, -3) rectangle (.5, -2.75);
    \shade[top color = gray!40, bottom color = gray!0, draw= none, shading angle=-45] (-3, -3) rectangle (-2.75, -2.75);

    \tiny

    \node (za) at (0, 1) {$Z^{-1}a$};
    \node (b) at (0,0) {$b$};
    \node (Zc) at (0,-1) {$Zc$};
    
    \node (Wa) at (-1, 0) {$Wa$};
    \node (WZb) at (-1,-1) {$WZb$};
    \node (WZZc) at (-1,-2) {$WZ^2c$};
    
    \node (WWZa) at (-2.1, -1) {$W^2Za$};
    \node (WWZZb) at (-2.1,-2) {$W^2Z^2b$};
      
    \large
    \node at (-3.3,-3.3) {\color{gray} $\cdot^{\cdot^{\cdot}}$};
    \node at (-3.3,-1.7) {\color{gray} $\cdots$};
    \node at (-3.3,0) {\color{gray} $\cdots$};
    \node at (-1.7, -3.3) {\color{gray} $\vdots$};
    \node at (0, -3.3) {\color{gray} $\vdots$};
    
    \draw[->] (b) -- (Wa);
    \draw[->] (WZb) -- (WWZa);
    \draw[->] (WWZZb) -- (-3, -2);

    \draw[->] (b) --  (Zc);
    \draw[->] (WZb) --  (WZZc);
    \draw[->] (WWZZb) -- (-2.1, -2.75);
\end{tikzpicture}
\hspace{1cm}
\begin{tikzpicture}
    \fill[color = gray, fill opacity = .2] (-2.75,-2.75) rectangle (.5, .5);
    \fill[color = gray, fill opacity = .3] (-2.75,-2.75) rectangle (-.5, -.5);
    \shade[left color = gray!0, right color = gray!20, draw= none] (-3, -.5) rectangle (-2.75, .5);
    \shade[left color = gray!0, right color = gray!50, draw= none] (-3, -2.75) rectangle (-2.75, -.5);
    \shade[top color = gray!50, bottom color = gray!0, draw= none] (-2.75, -3) rectangle (-.5, -2.75);
    \shade[top color = gray!20, bottom color = gray!0, draw= none] (-.5, -3) rectangle (.5, -2.75);
    \shade[top color = gray!40, bottom color = gray!0, draw= none, shading angle=-45] (-3, -3) rectangle (-2.75, -2.75);

    \tiny
    \node (zza) at (.1, 2) {$Z^{-2}a$};
    \node (zb) at (.1,1) {$Z^{-1}b$};
    \node (c) at (.1,0) {$c$};
    
    \node (Wza) at (-1.1, 1) {$WZ^{-1}a$};
    \node (Wb) at (-1,0) {$Wb$};
    \node (WZc) at (-1,-1) {$WZc$};
    
    \node (WWa) at (-2.1, 0) {$W^2a$};
    \node (WWZb) at (-2.1,-1) {$W^2Zb$};
    \node (WWZZc) at (-2.1,-2) {$W^2Z^2c$};
      
    \large
    \node at (-3.3,-3.3) {\color{gray} $\cdot^{\cdot^{\cdot}}$};
    \node at (-3.3,-1.7) {\color{gray} $\cdots$};
    \node at (-3.3,0) {\color{gray} $\cdots$};
    \node at (-1.7, -3.3) {\color{gray} $\vdots$};
    \node at (0, -3.3) {\color{gray} $\vdots$};
    
    \draw[->] (zb) -- (Wza);
    \draw[->] (Wb) -- (WWa);
    \draw[->] (WWZb) -- (-3, -1);

    \draw[->] (zb) --  (c);
    \draw[->] (Wb) --  (WZc);
    \draw[->] (WWZb) -- (WWZZc);
\end{tikzpicture}
\hspace{1cm}
\begin{tikzpicture}
    \fill[color = gray, fill opacity = .2] (-2.75,-2.75) rectangle (.5, -.5);
    \fill[color = gray, fill opacity = .3] (-2.75,-2.75) rectangle (-.5, -1.5);
    \shade[left color = gray!0, right color = gray!20, draw= none] (-3, -1.5) rectangle (-2.75, -.5);
    \shade[left color = gray!0, right color = gray!50, draw= none] (-3, -2.75) rectangle (-2.75, -1.5);
    \shade[top color = gray!50, bottom color = gray!0, draw= none] (-2.75, -3) rectangle (-.5, -2.75);
    \shade[top color = gray!20, bottom color = gray!0, draw= none] (-.5, -3) rectangle (.5, -2.75);
    \shade[top color = gray!40, bottom color = gray!0, draw= none, shading angle=-45] (-3, -3) rectangle (-2.75, -2.75);

    \tiny

    \node (za) at (0, 1) {$Z^{-1}a$};
    \node (b) at (0,0) {$b$};
    \node (Zc) at (0,-1) {$Zc$};
    
    \node (Wa) at (-1, 0) {$Wa$};
    \node (WZb) at (-1,-1) {$WZb$};
    \node (WZZc) at (-1,-2) {$WZ^2c$};
    
    \node (WWZa) at (-2.1, -1) {$W^2Za$};
    \node (WWZZb) at (-2.1,-2) {$W^2Z^2b$};
      
    \large
    \node at (-3.3,-3.3) {\color{gray} $\cdot^{\cdot^{\cdot}}$};
    \node at (-3.3,-1.7) {\color{gray} $\cdots$};
    \node at (-3.3,0) {\color{gray} $\cdots$};
    \node at (-1.7, -3.3) {\color{gray} $\vdots$};
    \node at (0, -3.3) {\color{gray} $\vdots$};
    
    \draw[->] (b) -- (Wa);
    \draw[->] (WZb) -- (WWZa);
    \draw[->] (WWZZb) -- (-3, -2);

    \draw[->] (b) --  (Zc);
    \draw[->] (WZb) --  (WZZc);
    \draw[->] (WWZZb) -- (-2.1, -2.75);
\end{tikzpicture}
\caption{Left: The complex $C^-_0$, for $s = 0$, with the subcomplex $C_0$ shaded, where $C$ is the knot Floer complex of the right-hand trefoil from Example \ref{ex:trefoil}; Middle: the complex $C^-_{-1}$ with the subcomplex $C_{-1}$ shaded; Right: the complex $C_{-1}$ shaded as a subcomplex of $C^-_0$. In each case, restricting to the right column only gives $\widehat C^-_s$ for the appropriate value of $s$ and restricting to the light gray portion gives $\widehat C_s$.}
\label{fig:Cs_complex}
\end{figure}
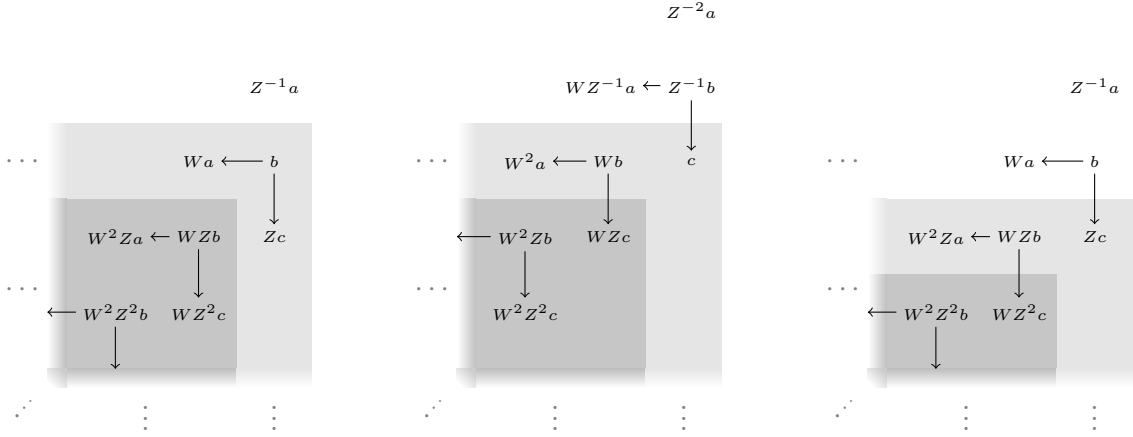

\subsection{The inclusion maps $v_s$}\label{sec:v_s-maps}
Let $C$ denote a knot-like complex over $\sR$, and suppose $\{x_i\}_{i=1}^n$ is a set of generators over $\sR$. Let $C_s$ denote the subcomplex of $C$ in Alexander grading $s$, so that $C = \bigoplus_{s\in\Z} C_s$. Note that $C_s$ is a module over $\F[\sU]$ and is generated over $\F[\sU]$ by $\{W^{a_i} Z^{b_i} x_i\}_{i=1}^n$ where $a_i = \max\{0, A(x_i)-s\}$ and $b_i = s + a_i - A(x_i)$.

It is sometimes useful to allow negative powers $Z$ by considering $C \otimes \F[Z, Z^{-1}]$, which we will denote $C^-$. The subcomplex of $C^-$ in Alexander grading $s$ is denoted $C^-_s$. Although $C^-$ is not finitely generated over $\sR$, each $C^-_s$ is a finitely generated complex over $\F[\sU]$ with generators $\{Z^{s - A(x_i)} x_i\}_{i=1}^n$. Note that $C_s$ is clearly a subcomplex of $C^-_s$, and the inclusion map is denoted
$$v_s: C_s \into C^-_s.$$
The complexes $C^-_0$ and $C^-_{-1}$ when $C = \CFKr(RHT)$ are shown in Figure \ref{fig:Cs_complex}, and the shaded portions are the complexes $C_0$ and $C_{-1}$. In each case $C^-_s$ is generated over $\F[\sU]$ by the elements in the rightmost column, and $C_s$ is generated by the elements in the lightly shaded band.

There is a clear bijection between the generating sets of $C_s$ and $C^-_s$, and $v_s$ takes each generator to a power of $\sU$ times the corresponding generator. The map $v_s$, or more precisely the induced map $v_{s,*}$ on homology, is used to define several numerical invariants associated to a knot-like complex as we will see in Section \ref{sec:concordance-invariants}. In the example in Figure \ref{fig:Cs_complex} the map $v_0$ takes the generator $Wa$ of $C_0$ to the $Wa = \sU \cdot (Z^{-1}a)$, or $\sU$ times the corresponding generator of $C^-$, while $v_0$ takes the generators $b$ and $Zc$ to themselves. Since $Z^{-1}a$ generates the homology of $C^-$, the induced map $v_{0,*}$ is not surjective but has image $\sU \cdot H_*(C^-_0)$. 

The map $v_s$ has an analog when working over $\sRhat$; the map of $\F$-vector spaces $\hat v_s: \widehat C_s \to \widehat C^-_s$ is obtained from the map $v_s$ of $\F[\sU]$ modules by setting $\sU = 0$. Equivalently, $\hat v_s$ is the composition of the projection of $C_s$ to the span of the generators $\{ Z^{s-A(x)} x | A(x) \le s\}$ followed by inclusion into $\widehat C^-_s$. In the example in Figure \ref{fig:Cs_complex} the map $\hat v_0$ takes $Wa$ to 0 and acts as the identity on the other two generators.

\begin{remark} Note that the subcomplexes $C^-_s$ and $C^-_{s'}$ are isomorphic for any choice of $s$ and $s'$, with the isomorphism given by multiplying by the appropriate power of $Z$. For this reason, it is convenient to work with only one copy of $C^-_s$, namely $C^-_0$. Note that the subcomplex $C_s$ is obtained from $C^-_s$ (when both are viewed as complexes over $\F$) by restricting to generators $Z^{s - A(x_i) + k} W^k x_i$ that have nonnegative power of $Z$, but multiplying by $Z^{-s}$ gives an isomorphism between this and the subcomplex of $C^-_0$ generated by those generators $Z^{- A(x_i) + k} W^k x_i$ with power of $Z$ at least $-s$. By abuse of notation, we will often consider $C_s$ to be this subcomplex of $C^-_0$, that is we will write $C_s$ interchangeably with $Z^{-s}C_s$, with this identification understood. Figure \ref{fig:Cs_complex} shows $C_{-1}$ both as a subcomplex of $C^-_{-1}$ and as a subcomplex of $C^-_0$. Moreover, if $C = \CFKr(K)$ for a knot $K$ then $C^-_0$ isomorphic to $\CFKminus(K)$ as first defined in \cite{OzSz:knots} (recall that the original definition did not have the formal variable $Z$ and gave a filtered complex over $\F[\sU]$, see Remark \ref{rmk:WZ-notation}). Under this identification, the subcomplexes $C_s \subset C^-_0$ are precisely the subcomplexes $A_s^- \subset \CFKminus(K)$ defined when constructing the surgery formula as in \cite{OzSz:integer-surgeries}. \end{remark}

\begin{figure}
\includegraphics[scale = 1]{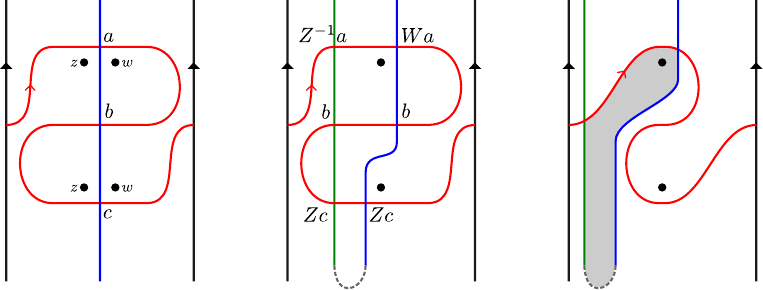}
\caption{Left: Pairing $(\Gamma, \bchain)$ with $\mu$ in $\barT_{w,z}$ gives the complex $C$ over $\sR$. Middle: Pairing $(\Gamma, \bchain)$ with $\mu_{-2\epsilon}$ (green) in $\barT_u$ gives the $\F[\sU]$ chain complex $C^-_s$, pairing $(\Gamma, \bchain)$ with $\alpha_s$ (blue) gives the $\F[\sU]$ chain complex $C_s$, and when the lower ends of $\mu_{-2\epsilon}$ and $\alpha_s$ are connected by a cap (dotted line) then counting bigons between $(\Gamma, \bchain)$ and the combined curve whose boundary covers the cap determines the map $v_s: C_s \to C^-_s$. Right: Pulling tight gives a quick way to compute the map on homology $v_{s,*}$, which in this case is multiplication by $\sU$ since the generators are connected by a single bigon (shaded) covering the marked point once.}
\label{fig:v-map-curves}
\end{figure}
The maps $v_s$ can be visualized using immersed curves. Let $(\Gamma, \bchain)$ be a knot-like curve representing the complex $C$, and assume $C$ is equipped with the basis naturally arising from this correspondence. Recall that this means the Floer complex of $(\Gamma, \bchain)$ with the vertical line $\mu$ in the doubly marked cylinder $\barT_{w,z}$  is precisely the complex $C$. To obtain the complex $C^-_0$ we may push the vertical line $\mu$ to the left of the $z$ marked points; in other words, we take the Floer complex of $(\Gamma, \bchain)$ with the line $\mu_{-2\epsilon} = \{-2\epsilon\}\times \R$ (it is clear that we may now merge the $z$ and $w$ marked points and work in $\barT_u$). There is an obvious bijection between the generators of $CF(\Gamma, \mu)$ and the generators of $CF(\Gamma, \mu_{-2\epsilon})$, but we relabel the generators by adding a power of $Z$ determined by the height of the generator. Figure \ref{fig:v-map-curves} shows the example of the right-hand trefoil, where the green vertical line is $\mu_{-2\epsilon}$. The complex $C_s$ (more precisely, $Z^{-s}C_s$) is obtained by taking the Floer complex with a curve $\alpha_s$ that follows $\mu_{-\epsilon} = \{\-epsilon\}\times\R$ upward until height $s$, moves rightward to $\mu_{\epsilon} = \{\epsilon\}\times\R$, then follows $\mu_{\epsilon}$ upward; see the blue curve in Figure \ref{fig:v-map-curves}. Once again there is a bijection between generators of this and generators of $C$, but we add nonnegative powers of $W$ or $Z$ based on how many $w$ or $z$ marked points lie between $\mu$ and $\alpha$ between the height of the generator and height $s$. Both $\mu_{-2\epsilon}$ and $\alpha_s$ extend to the infinite ends of $\barT_u$, but since $\Gamma$ is a compact curve we may truncate the curves below a sufficiently low height. We then attach the lower ends of the truncated curves by an arc as in Figure \ref{fig:v-map-curves}, resulting in a curve we will denote $\mu\alpha_s$; following the figure, we will refer to the portions of $\mu\alpha_s$ coming from $\mu_{-2\epsilon}$ and $\alpha_s$ as the green and blue portions of $\mu\alpha_s$. Now the Floer complex of $(\Gamma, \bchain)$ with $\mu\alpha_s$ is precisely the mapping cone of the map $v_s$; that is, the Floer complex consists of the direct sum of the complexes $C_s$ and $C^-_s$ along with additional differentials from generators of $C_s$ to generators of $C^-_s$ arising from bigons that involve the arc connecting the two parts of $\mu\alpha_s$, and it is not difficult to see that the map obtained by counting these bigons is precisely the map $v_s$. Indeed, assuming $(\Gamma, \bchain)$ is arranged to have only horizontal segments within the strip $[-2\epsilon, \epsilon]\times \R$, as in the figure, then there is exactly one bigon contributing to $v_s$ for each generator of $C_s$, taking that generator to the corresponding generator of $C^-_s$ with the appropriate power of $\sU$. The middle of Figure \ref{fig:v-map-curves} illustrates $v_0$ with the right-hand trefoil curve, which counts bigons connecting generators of $C_s$ (intersections of red and blue) to the corresponding generators of $C^-_s$ (intersections of red and green).

An advantage of the immersed curve description is that the induced maps on homology are unaffected by homotopy of curves. In particular, we may apply a homotopy to both the green portion and the blue portion of $\mu\alpha_s$ so that each portion intersects $\Gamma$ minimally, though we maintain the fact that the two segments are joined by a small arc far below $\Gamma$. This transformation modifies the complexes $C_s$ and $C^-_s$ and the map $v_s$ represented by the picture at the chain level, but the data is unchanged on the homology level. The right side of Figure \ref{fig:v-map-curves} performs such a homotopy, resulting in fewer intersections and simpler complexes compared to the curves in the middle of the Figure. It is clear from this picture that the homologies of $C_0$ and $C^-_0$ are both $\F[\sU]$ and the map $v_{0,*}$ is multiplication by $\sU$. The maps $\hat v_s$ and $\hat v_{s,*}$ can be interpreted in the same way if we work in the punctured surface $\barT_\bullet$; thus in Figure \ref{fig:v-map-curves} the map $\hat v_{0,*}:\F \to \F$ is trivial because there are no bigons that do not cover a puncture.

Using the map $\hat v_s$ we can revisit the definition of the invariant $\nu$ from section \ref{sec:gamma_0-invariants}, which as originally defined is the minimum $s$ for which the induced map on homology $\hat v_{s,*}$ is surjective. In terms of immersed curves, this is the minimum $s$ for which there is a bigon not covering the puncture from the $\alpha_s$ portion of $\mu\alpha_s$ to the $\mu_{-2\epsilon}$ portion of $\mu\alpha_s$, provided both portions are arranged to intersect $\Gamma$ minimally. Since only $\gamma_0$ intersects $\mu_{-2\epsilon}$, such a bigon would have to have the $\Gamma$ part of its boundary on the distinguished component $\gamma_0$. Clearly such a bigon does not exist if $\alpha_s$ crosses $\mu$ strictly lower than $\gamma_0$ first crosses $\mu$ (i.e. if $s < \tau$), such a bigon does exist if $\alpha_s$ crosses $\mu$ strictly higher than $\gamma_0$ first crosses $\mu$ (i.e. if $s > \tau$), and if $s = \tau$ such a bigon exists if $\gamma_0$ does not turn upward (i.e. if $\epsilon \neq -1$).

In addition to the maps $v_s$ there are also maps $h_s$ defined similarly reversing the roles of $Z$ and $W$. Due to the symmetry of knot Floer complexes, these carry the same information as $v_s$ for knots in $S^3$ and are less often used to define invariants, but both the $v_s$ and $h_s$ appear in the Dehn surgery formula. To define $h_s$ we allow negative powers of $W$ rather than of $Z$ and we define $h_s$ to be the inclusion map from $C_s$ to $(C \otimes \F[W, W^{-1}])_s$. The latter complex, a chain complex over $\F[\sU]$ is realized as the Floer complex of $(\Gamma, \bchain)$ with $\mu_{2\epsilon} = \{2\epsilon\}\times \R$, and the map $h_s$ counts bigons involving the capping arc when $\mu_{2\epsilon}$ and $\alpha_s$ are truncated at a sufficiently large height and joined by an arc.

\subsection{The surgery formula}\label{sec:surgery-formula} 
Ozsv{\'a}th and Szab{\'o} proved that Heegaard Floer invariants of Dehn surgeries on a knot $K$ can be computed from the knot Floer invariant of $K$ using a mapping cone formula \cite{OzSz:integer-surgeries, OzSz:rational-surgeries}. This result has a nice interpretation in terms of immersed curves: letting $S^3_{p/q}(K)$ denote the $\frac p q$-surgery on $K \subset S^3$, for each spin$^c$-structure $\spin$ of $S^3_{p/q}(K)$ the $\F[\sU]$-module $\HFminus(S^3_{p/q}(K;\spin))$ is the Floer homology in $\barT_u$ of $\Gamma(K)$ with a line of slope $p/q$. More precisely, recall that spin$^c$-structures $S^3_{p/q}(K)$ have a canonical indexing by elements of $\Z/p\Z$. For each $i \in \Z/p\Z$, let $\ell_{p,q;i}$ be a line of slope $p/q$ in $\barT_u$ (that is, a curve that lifts to a straight line of slope $p/q$ in the marked plane $\tildeT_u$) that intersects $\mu$ at the point $(0, \frac{2i + 1}{2q} - \frac 1 2)$. The surgery formula then takes the following form:

\begin{theorem}\label{thm:surgery}
For a knot $K\subset S^3$ with immersed curve invariant $\Gamma(K)$ and any $p/q\in\Q$, for each spin$^c$-structure $\spin_i$ of $S^3_{p/q}(K)$ we have an isomorphism of relatively graded $\F[U]$-modules
$$\HFminus(S^3_{p/q}(K;\spin_i)) \cong HF( \Gamma(K), \ell_{p,q;i} ),$$
where $HF$ on the right refers to Floer homology in the marked surface $\barT_u$. The Floer homology of the same curves in the punctured surface $\barT_\bullet$ recovers $\HFhat(S^3_{p/q}(K;\spin_i)$.
\end{theorem}

The right side of Figure \ref{fig:surgery} shows the surgery formula applied to $-1$ surgery on the right-hand trefoil. The resulting Floer complex has three generators. In the punctured surface $\barT_\bullet$ there are no differentials, so $\HFhat(S^3_{-1}(RHT))$ is isomorphic to $\F^3$. If we work instead in $\barT_\sU$ we have $\partial a = \partial c = 0$ and $\partial b = \sU a + \sU c$ so $\HFminus(S^3_{-1}(RHT))$ is isomorphic to $\F[\sU] \oplus \F$.

Note that the original mapping cone formula describes $\HFminus(S^3_{p/q}(K;\spin_i))$ as the homology of the mapping cone of a particular chain map from $\mathbb A = \bigoplus_{j \in \Z} C_{s_j}$ to $\mathbb B = \bigoplus_{j \in \Z} C^-_{s_j}$, where the sequence of integers $\{s_j\}$ depends on $p, q$, and the spin$^c$-structure $\spin_i$ (precisely, we have $s_j = \lfloor \frac{i + jp}{q} \rfloor$). The map is the direct sum of the maps
$$v_{s_j} : C_{s_j} \to C^-_{s_j} \quad \text{ and } \quad h_{s_j}: C_{s_j} \to C^-_{s_{j+1}}$$
connecting the appropriate summands. This infinitely generated complex can be realized as the Floer complex of $\Gamma(K)$ with a curve $\gamma_{p,q,\spin_i}$ made from infinitely many curves $\alpha_{s_j}$ and infinitely many copies of $\mu_{\frac 1 2}$, truncated at sufficiently high and low heights and connected to each other by arcs as shown on the left side of Figure \ref{fig:surgery} (the figure is drawn in the covering space $\tildeT_u$ of $\barT_u$). The Floer homology with each $\alpha_{s_j}$ portion of the curve gives a copy of $C_{s_j}$, the Floer homology with each copy of $\mu_{\frac 1 2}$ gives a copy of $C^-_{s_j}$, and bigons involving the dotted arcs on the bottom give the $v$ maps and bigons involving the dotted arcs on the top give the $h$ maps. Theorem \ref{thm:surgery} follows from this picture by observing that the curve $\gamma_{p,q,\spin_i}$ is isotopic to a straight line of slope $p/q$ meeting $\mu$ at the appropriate height. For more details, see \cite[Section 11]{Hanselman:CFK}.
\begin{figure}
\includegraphics[scale = .8]{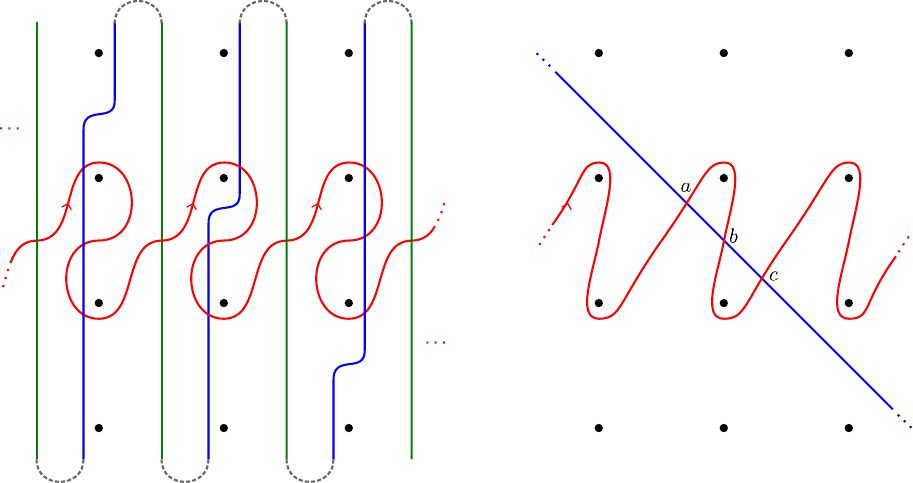}
\caption{Computing $\HFminus(S^3_{-1}(RHT))$ using immersed curves. The arrangement on the left recovers the mapping cone formula as defined by Ozsv{\'a}th and Szab{\'o} on the chain level. Pulling tight, which preserves Floer homology, gives the simpler picture on the right.}
\label{fig:surgery}
\end{figure}

\subsection{Inclusion maps of other subcomplexes}\label{sec:other-inclusion-maps}

When we draw the complex $C^-_0$ on the half of the integer lattice with non-positive $x$-coordinate as in Figure \ref{fig:Cs_complex}, the subcomplex $C_s$ is obtained by restricting to all lattice points with $y$-coordinate at most $s$. More generally, we obtain a subcomplex by restricting to any subset of the half-lattice that is closed under moving downward or leftward. Such a region is bounded above by a staircase shaped boundary, which we will keep track of by listing the corners. It is enough to record the integer given by the $x$-coordinate minus the $y$-coordinate for each corner; the sequence of such integers is strictly increasing if corners are listed from the bottom right to the top left of the staircase boundary, and this sequence determines the region since the first corner necessarily has $x$-coordinate 0. Given an increasing sequence of integers $(s_1, \ldots, s_k)$, let $R_{(s_1, \ldots, s_k)}$ denote the subset of the half-lattice whose corners determine the sequence $(s_1, \ldots, s_k)$ in this way, and let $C_{(s_1, \ldots, s_k)}$ denote the corresponding subcomplex of $C^-_0$. Figure \ref{fig:staircase-subcomplex} shows the subcomplex $C_{(-2,-1,1,3,4)}$ when $C$ is the knot Floer complex for $T_{2,5}$. We may allow the region with no corners, $R_{()}$, which is the entire half-lattice so $C_{()} = C^-_0$. The complexes $C_s$ also arise as special cases coming from regions with a single corner, where $C_s = C_{(s)}$. Regions $R_{(s_1, \ldots, s_k)}$ with $k$ odd are bounded above since their final corner is the beginning of a horizontal portion of the staircase boundary, while regions with an even number of corners are unbounded above. We may also allow infinitely many corners---for example, $R_{(-1,0,1,2,3,4,\ldots)}$ contains all points in the half-lattice below the line $x+y = -1$---however because $C$ is finitely generated it always suffices to consider a finite truncation of such an infinite list of corners. Note that as a complex over $\F[\sU]$, $C_{(s_1, \ldots, s_k)}$ is finitely generated and has a basis given by appropriate powers of $\sU$ times the basis elements of $C^-_0$. The inclusion map $C_{(s_1, \ldots, s_k)} \into C^-_0$ will be denoted $v_{(s_1, \ldots, s_k)} $.

It is not difficult to see that the complex $C_{(s_1, \ldots, s_k)}$ is realized as the Floer complex of the knot-like curve $(\Gamma, \bchain)$ representing $C$ with a nearly vertical curve $\alpha_{(s_1, \ldots, s_k)}$ that coincides with $\mu_{-\epsilon}$ below height $s_1$ and above that weaves through the punctures crossing $\mu$ once at each height in $\{s_1, \ldots, s_k\}$. Above height $s_k$ the curve coincides with $\mu_{\epsilon}$ if $k$ is odd and $\mu_{-\epsilon}$ if $k$ is even (if we allow infinite lists of corners, the line continues weaving through the punctures indefinitely). See the right side of Figure \ref{fig:staircase-subcomplex} for an example of $C_{(-2,-1,1,3,4)}$. Similar to the geometric description of $v_s$, the map $v_{(s_1, \ldots, s_k)}$ is realized by counting bigons between $(\Gamma, \bchain)$ and the curve $\mu\alpha_{(s_1, \ldots, s_k)}$ formed by connecting $\alpha_{(s_1, \ldots, s_k)}$ and $\mu_{-2\epsilon}$, truncated at sufficiently low height, by an arc. The Floer complexes of each portion of $\mu\alpha_{(s_1, \ldots, s_k)}$ give $C_{(s_1, \ldots, s_k)}$ and $C^-_0$ and bigons involving the connecting arc contribute to the map $v_{(s_1, \ldots, s_k)}$. As with the $v_s$ maps we are free to homotope each portion of $\mu\alpha_{(s_1, \ldots, s_k)}$ if we only care about the induced map on homology $v_{(s_1, \ldots, s_k),*}$. In Figure \ref{fig:staircase-subcomplex} putting the curves in minimal position makes it clear that $H_* C_{(-2,-1,1,3,4)} \cong \F[\sU] \oplus \F$ and the map $v_{(s_1, \ldots, s_k),*}$ is multiplication by $\sU$ on the tower summand and zero on the torsion summand.

Given a sequence $(s_1, \ldots, s_k)$ corresponding to the corners of a region of the lattice, with $k$ odd, there is a corresponding knot-like complex $S_{(s_1, \ldots, s_k)}$ with generators $\{x_1, \ldots, x_k\}$ and differential
$$\partial x_i = \begin{cases} 
0 & \text{ if } i \text{ even}\\
W^{s_2-s_1} x_2 & \text{ if } i=1 \\
Z^{s_k-s_{k-1}} x_{k-1} & \text{ if } i=k \\
W^{s_{i+1}-s_{i}} x_{i+1} + Z^{s_i - s_{i-1}} x_{i-1} & \text{otherwise}
\end{cases}.$$
The bigrading on $S_{(s_1, \ldots, s_k)}$ is determined by declaring that $\gr_w(x_1) = \gr_z(x_k) = 0$. This complex $S_{(s_1, \ldots, s_k)}$ is called a staircase complex; examples include the knot Floer complexes of L-space knots. Note that in particular $S_{(s_1, \ldots, s_k)}$ is the complex corresponding to the staircase shape given by the union of the non-infinite line segments in the boundary of $R_{(s_1, \ldots, s_k)}$. The knot-like curve corresponding to the complex $S_{(s_1, \ldots, s_k)}$ has a single component, which we denote $\gamma_{(s_1, \ldots, s_k)}$, that crosses $\mu$ exactly once at each of the heights $\{s_1, \ldots, s_k\}$. Note that $\gamma_{(s_1, \ldots, s_k})$ is closely related to the curve $\alpha_{(s_1, \ldots, s_k)}$, with the former obtained from the latter by truncating both ends, taking each loose end to a fixed point on $\mu_{\frac1 2}$ without crossing $\mu$, and gluing the loose ends to form a closed curve. In particular, if we work in the strip obtained from $\barT_u$ by cutting open along $\mu_{\frac 1 2}$ and slide the left endpoint of the arc resulting from $\gamma_{(s_1, \ldots, s_k)}$ arbitrarily far downward and the right endpoint arbitrarily far upward, then the Floer complex of any curve with this is chain homotopy equivalent to the Floer complex of the same curve with $\alpha_{(s_1, \ldots, s_k)}$. This, along with the discussion in Section \ref{sec:morphisms}, tells us that the complex $C_{(s_1, \ldots, s_k)}$ is homotopy equivalent both to the complex $C^\vee \otimes S_{(s_1, \ldots, s_k)}$ and to the morphism space $\Mor(S_{(s_1, \ldots, s_k)}, C)$.

\begin{figure}
\raisebox{2mm}{\includegraphics[scale = 1]{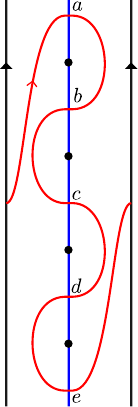}}
\hspace{1 cm}
\begin{tikzpicture}

    \draw[help lines] (-4.5,-4.5) grid (0,2.5);
    
    \fill[color = gray, fill opacity = .2] (-.85,-4.5) rectangle (.15, -1.85);
    \fill[color = gray, fill opacity = .2] (-2.85,-4.5) rectangle (-.85, .15);
    \fill[color = gray, fill opacity = .2] (-4.5,-4.5) rectangle (-2.85, 1.15);
   
   \shade[top color = gray!20, bottom color = gray!0, draw= none] (-4.5, -4.75) rectangle (.15, -4.5);
   \shade[right color = gray!20, left color = gray!0, draw= none] (-4.75, -4.5) rectangle (-4.5, 1.15);
   \shade[top color = gray!10, bottom color = gray!0, draw= none, shading angle=-45] (-4.75, -4.75) rectangle (-4.5, -4.5);

    \tiny

    \node[label={[xshift = 5mm, yshift = -3 mm] $Z^{-2}a$}] (zza) at (0, 2) {$\bullet$};
    \node[label={[xshift = 5mm, yshift = -3 mm] $Z^{-1}b$}] (zb) at (0, 1) {$\bullet$};
    \node[label={[xshift = 5mm, yshift = -3 mm] $c$}] (c) at (0, 0) {$\bullet$};
    \node[label={[xshift = 5mm, yshift = -3 mm] $Zd$}] (Zd) at (0, -1) {$\bullet$};
    \node[label={[xshift = 5mm, yshift = -3 mm] $Z^2e$}] (ZZe) at (0, -2) {$\bullet$};

    \node (Wza) at (-1, 1) {$\bullet$};
    \node (Wb) at (-1, 0) {$\bullet$};
    \node (WZc) at (-1, -1) {$\bullet$};
    \node (WZZd) at (-1, -2) {$\bullet$};
    \node (WZZZe) at (-1, -3) {$\bullet$};

    \node (WWa) at (-2, 0) {$\bullet$};
    \node (WWZb) at (-2, -1) {$\bullet$};
    \node (WWZZc) at (-2, -2) {$\bullet$};
    \node (WWZZZd) at (-2, -3) {$\bullet$};
    \node (WWZZZZe) at (-2, -4) {$\bullet$};
    
    \node (WWWZa) at (-3, -1) {$\bullet$};
    \node (WWWZZb) at (-3, -2) {$\bullet$};
    \node (WWWZZZc) at (-3, -3) {$\bullet$};
    \node (WWWZZZZd) at (-3, -4) {$\bullet$};

    \node (WWWWZZa) at (-4, -2) {$\bullet$};
    \node (WWWWZZZb) at (-4, -3) {$\bullet$};
    \node (WWWWZZZZc) at (-4, -4) {$\bullet$};

    \draw[->] (zb) -- (Wza);
    \draw[->] (Wb) -- (WWa);
    \draw[->] (WWZb) -- (WWWZa);
    \draw[->] (WWWZZb) -- (WWWWZZa);
    \draw[->] (WWWWZZZb) -- (-4.65, -3);

    \draw[->] (Zd) -- (WZc);
    \draw[->] (WZZd) -- (WWZZc);
    \draw[->] (WWZZZd) -- (WWWZZZc);
    \draw[->] (WWWZZZZd) -- (WWWWZZZZc);

    \draw[->] (zb) --  (c);
    \draw[->] (Wb) --  (WZc);
    \draw[->] (WWZb) -- (WWZZc);
    \draw[->] (WWWZZb) -- (WWWZZZc);
    \draw[->] (WWWWZZZb) -- (WWWWZZZZc);

    \draw[->] (Zd) --  (ZZe);
    \draw[->] (WZZd) --  (WZZZe);
    \draw[->] (WWZZZd) -- (WWZZZZe);
    \draw[->] (WWWZZZZd) -- (-3, -4.65);

\end{tikzpicture}
\hspace{1 cm}
\raisebox{0 mm }{\includegraphics[scale = .9]{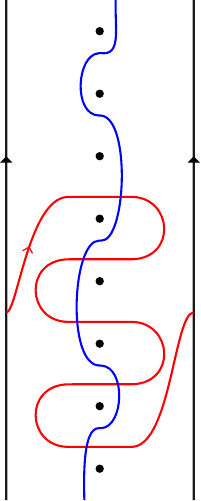} }
\caption{Left: the immersed curve $\Gamma$ corresponding to the complex $C = \CFKr(T_{2,5})$; Middle: The complex $C^-_0$ with the subcomplex $C_{(-2,-1,1,3,4)}$ shaded; Right: the complex $C_{(-2,-1,1,3,4)}$ realized as the Floer complex of $\Gamma$ and $\alpha_{(-2,-1,1,3,4)}$.}
\label{fig:staircase-subcomplex}
\end{figure}
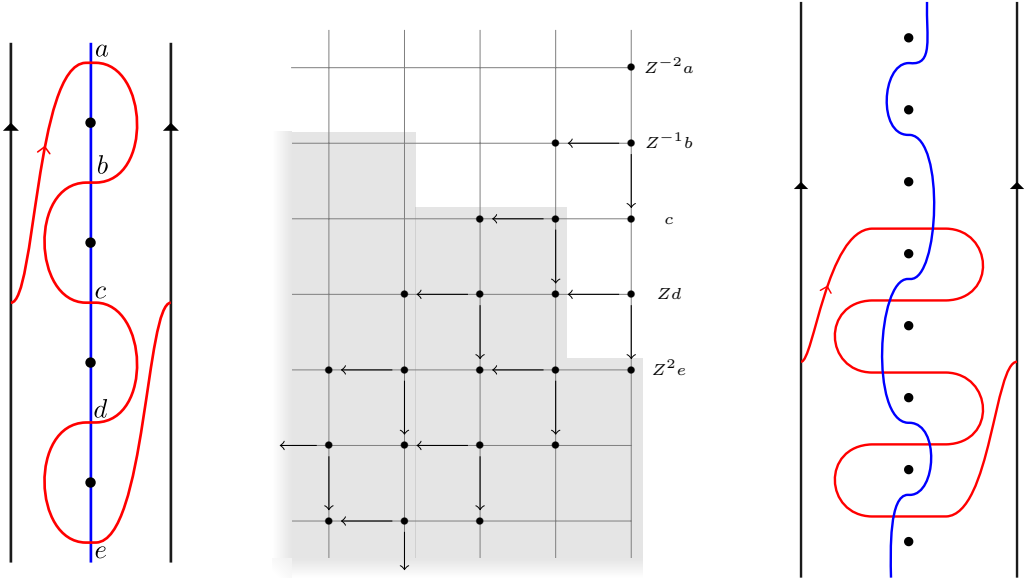

\begin{remark}
If $k$ is odd, there is always a diffeomorphism of the cylinder $\barT_u$ permuting the marked points such that $\alpha_{(s_1, \ldots, s_k)}$ is taken to $\alpha_{s'}$ for some $s'$. By applying the same diffeomorphism to the curve corresponding to $C$, we see that $C_{(s_1, \ldots, s_k)}$ is isomorphic to $C'_{s'}$ for some other complex $C'$ (in fact, there are many such diffeomorphisms and so many choices for the modified complex $C'$). See for example Figure \ref{fig:curve-twisting-vs}. Because the diffeomorphism can be supported away from portion of $\alpha_{(s_1, \ldots, s_k)}$ corresponding to $C^-_0$, it is clear that on the level of homology the map $v_{(s_1, \ldots, s_k), *}$ for $C$ is simply $v_{s', *}$ for the modified complex $C'$. This geometric observation seems noteworthy, though it is unclear whether this is useful in geometrically extracting information related to the $C{(s_1, \ldots, s_k)}$ complexes.
\end{remark}

\begin{figure}
\includegraphics[scale = .8]{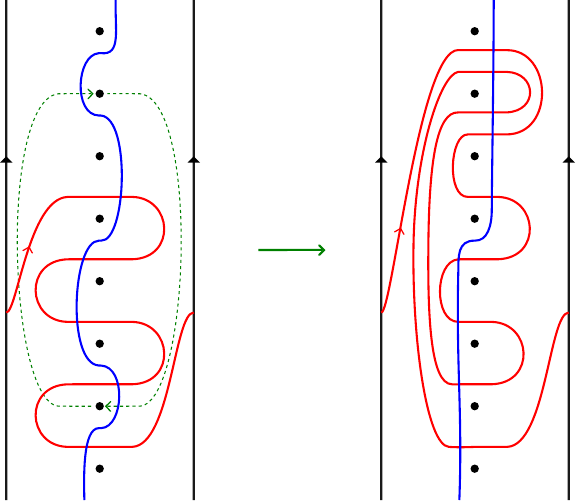}
\caption{Applying a diffeomorphism of $\barT_u$ shows that $C_{(-2,-1,1,3,4)}$ is isomorphic to $C'_1$ for a different complex $C'$.}
\label{fig:curve-twisting-vs}
\end{figure}

\section{Distinguished components and concordance invariants}\label{sec:distinguished-components}

Though the multicurve $\Gamma(K)$ in $\overline \sS_u$ may have many components, we have already seen that it has exactly one arc component $\gamma_0$ and that this component contains a wealth of useful information. In this section we expand on the distinguished component by giving a suitable generalization in the minus setting and relating these components to knot concordance. We also discuss in Section \ref{sec:group-structure} how these concepts give rise to a group structure on curves and explain, in the $\sU = 0$ setting, how to compute the immersed curves for connected sums of knots. That section is not essential to the rest of the paper.

\subsection{The invariant $\gamma_0(K)$}\label{sec:gamma_0}
Recall that for a knot-like curve $(\Gamma, \bchain)$ in $\allcurves$ the underlying immersed multicurve $\Gamma = \{\gamma_0, \gamma_1, \ldots, \gamma_n\}$ has exactly one arc component, which by convention is the one labelled $\gamma_0$. One way to extract partial information from knot Floer homology is to consider only this distinguished component. Note that the multicurve $\Gamma$, and hence the curve $\gamma_0$, depends only on the complex mod $\sU$. No decorations are needed for $\gamma_0$: the grading on $\gamma_0$ is completely determined by the curve (recall that the grading function is assumed to evaluate to $0$ at the endpoints of the component), and arc components never require local system decorations.
\begin{definition} For a knot $K \subset S^3$, $\gamma_0(K)$ is (the homotopy class of) the distinguished component of the knot-like curve in $\barT_\bullet$ representing the knot-like complex $\CFKrhat(K)$.
\end{definition}
Following Remark \ref{rmk:curve-notation}, the invariant $\gamma_0(K)$ can be recorded as a list of integers. For example, $\gamma_0$ is $[0]$ for the unknot and for the figure eight knot, $[1,0,-1]$ for the right-hand trefoil, and $[-5, -7, -5, -2, 0, 2, 5, 7, 5]$ for the complex in Example \ref{ex:popovic}.

Even though $\gamma_0(K)$ throws out all other components of the immersed multicurve and the local system decoration, it still captures a significant amount of useful information about the knot $K$. Moreover, the curve $\gamma_0(K)$ turns out to be a concordance invariant.

\begin{proposition}
The distinguished curve component $\gamma_0(K)$ is equivalent to the \emph{stable equivalence class} of $\widehat C$ as an $\sRhat$-complex as defined in \cite{Hom:survey}, or to the \emph{local equivalence class} of the knot-like complex $\widehat C = \CFKrhat(K)$ over $\sRhat$ as defined in \cite{DHST:more-concordance} (which follows the definition of local equivalence in the involutive setting from \cite[Section 2.3]{Zemke:connected-sums-involutive}, ignoring the involutive part). In particular, $\gamma_0$ is a concordance invariant.
\end{proposition}
\begin{proof}
Two knot-like complexes over $\sRhat$ are stably equivalent, in the sense of \cite{Hom:survey}, if they are homotopy equivalent to each other after taking the direct sum of each with acyclic complexes. In other words, if each complex is decomposed into irreducible direct summands, the summands from each complex that generate the horizontal homology are homotopy equivalent. Complexes are locally equivalent if there are appropriate chain maps between them inducing isomorphisms on vertical homology, and it is straightforward to see that this happens exactly when direct summands supporting homology in each complex are homotopy equivalent.

Let $(\Gamma, \bchainhat)$ be the standard knot-like curve in $\barT_\bullet$ representing $\widehat C$. Note that the complex determined by $(\Gamma, \bchainhat)$ is the direct sum of complexes determined by the individual components $\gamma_i$ (equipped with the restriction of $\bchainhat$ to each component). Moreover, these summands cannot decompose further as direct sums since this would violate the uniqueness of the knot-like curve representing $\widehat C$. The summand supporting homology is clearly the one corresponding to $\gamma_0$. Note also that the restriction to $\bchainhat$ is always trivial since $\gamma_0$ must be primitive---that is, the component $\gamma_0$ never carries a nontrivial local system.  The stable or local equivalence class of $\widehat C$ is equivalent to the homotopy equivalence class of this summand, which determines and is determined by the homotopy class of the curve $\gamma_0$.
\end{proof}

\begin{remark}\label{rmk:standard_complex}
In \cite{DHST:more-concordance} local equivalence classes of knot-like complexes over $\sRhat$ are studied by showing that all complexes are locally equivalent to a \emph{standard complex} which has a horizontally and vertically simplified basis and $2n+1$ generators that are connected by a single sequence of alternating horizontal and vertical arrows. Since $\gamma_0$ is an immersed curve with no local system decoration, the $\sRhat$-complex it determines has precisely this form. That is, the complex associated to $\gamma_0$ is precisely the standard complex representative of the local equivalence class of $\widehat C$. In \cite{DHST:more-concordance} standard complexes are also represented by a string of integers, though their notation is different from the convention established here for $\gamma_0$ in Remark \ref{rmk:curve-notation}; rather than recording the sequence of heights at which $\gamma_0$ crosses $\mu$, they list the sequence of height differences between successive intersections of $\gamma_0$ with $\mu$.
\end{remark}

\subsection{The invariant $\Gamma_0(K)$}
The invariant $\gamma_0(K)$ depends only on the knot Floer complex of $K$ mod $\sU$, but working with the full complex over $\sR$ gives a potentially stronger concordance invariant. Let $C$ be a knot-like complex and let $(\Gamma, \bchain)$ be a knot-like curve in $\barT_u$ representing $C$. Algebraically we still consider either the local equivalence class of $C$ or of the stable equivalence class of $C$, but now as complexes over $\sR$. Once again this amounts to the homotopy equivalence class of the irreducible direct summand of $C$ that supports the vertical homology.

We would like to interpret this in terms of curves. Note that the complex $C$ is not necessarily the direct sum of complexes determined by each component of $\Gamma$, since the bounding chain $\bchain$ may contain intersection points between components that introduce differential between the summands. In light of this, we must define a suitable notion of connected components for knot-like curves. For a decorated immersed curve $(\Gamma, \bchain)$ with fixed decoration $\bchain$, this is simple to do: we partition the components of $\Gamma$ into the maximal number of subsets subject to the constraint that for any self-intersection point $p$ of $\Gamma$ contained in $\bchain$ the two (not necessarily distinct) components of $\Gamma$ meeting at $p$ lie in the same subset, and then we define the connected components of $(\Gamma, \bchain)$ to be the multicurves corresponding to the subsets of components decorated with the restriction of $\bchain$ to self-intersection points between those components. Unfortunately the definition we need is more subtle because the complex $C$ only determines a knot-like curve up to equivalence---recall that the multicurve $\Gamma$ is well-defined up to homotopy but the bounding chain $\bchain$ is only defined up to an (rather abstract) notion of equivalence in the immersed Fukaya category, it is not well-defined as a subset of self-intersection points. For instance, in the knot-like curve in Example \ref{ex:equivalent-bchains} the bounding chains $\bchain = \emptyset$ and $\bchain = \{p\}$ are equivalent, but with the definition above $(\Gamma, \emptyset)$ has two connected components while $(\Gamma, \{p\})$ has one. To account for this, we must first choose a particular representative $(\Gamma, \bchain)$ of the equivalence class of curves that has the maximal number of connected components. With this terminology in place, we define $\Gamma_0(C)$ to be the connected component that supports the vertical homology, which is necessarily the connected component that contains the distinguished curve component $\gamma_0$.

\begin{definition} For a knot-like complex $C$, $\Gamma_0(C)$ is the connected component of the corresponding knot-like curve containing the distinguished curve component $\gamma_0(C)$. $\Gamma_0(C)$ is also a knot-like curve in $\barT_{u}$ and is defined up to equivalence of knot-like curves. For a knot $K \subset S^3$, we let $\Gamma_0(K) = \Gamma_0( \CFKr(K) )$.\end{definition}

For the complexes in Example \ref{ex:first-examples} $\Gamma_0$ consists only of the distinguished curve $\gamma_0$ since $\bchain$ is trivial in each case. In Example \ref{ex:nontrivial-bchain} $\Gamma_0$ again only has a single curve component though it is decorated with the bounding chain $\bchain$ (recall that $\bchain$ was uniquely determined by the curve, so in this case $\Gamma_0$ does not contain more information than $\gamma_0$). In Example \ref{ex:different-bchains}, each complex has a different $\Gamma_0$, even though $\gamma_0$ is the same for all of them. If $\bchain = \emptyset$ then $\Gamma_0$ contains only the component $\gamma_0$, while in the other two cases $\Gamma_0$ is the whole complex consisting of both curve components decorated with the chosen $\bchain$.

As with $\gamma_0(C)$, $\Gamma_0(C)$ captures the local or stable equivalence class of $C$ and $\Gamma_0(K)$ is a concordance invariant.
\begin{proposition}
The distinguished connected component $\Gamma_0(K)$ is equivalent to the \emph{stable equivalence class}, or equivalently the \emph{local equivalence class}, of $C = \CFKr(K)$ as a chain complex over $\sR$. In particular, $\Gamma_0(K)$ is a concordance invariant.
\end{proposition}
\begin{proof}
The proof is similar to the $\sU = 0$ case. The stable or local equivalence class of $C$ is determined by the irreducible direct summand of $C$ that supports vertical homology, and by construction the irreducible direct summands of $C$ are the complexes determined by the connected components of the knot-like curve representing $C$.
\end{proof}

\subsection{Groups of immersed curves}\label{sec:group-structure}

Restricting knot-like curves to their distinguished components allows us to define a group structure on curves. More precisely, we say that two knot-like curves in $\allcurves$ in $\overline \sS_u$ are $\Gamma_0$-equivalent if they have the same distinguished connected component $\Gamma_0$ and let $\allcurves_0$ denote the set of $\Gamma_0$-equivalence classes of knot-like curves $\allcurves$; equivalently, $\allcurves_0$ can be identified with the subset of $\allcurves$ given by \emph{connected} knot-like curves in $\overline \sS_u$. Similarly, knot-like curves in $\overline \sS_\bullet$ are $\gamma_0$-equivalent if they have the same distinguished curve component $\gamma_0$ and we let $\allcurveshat_0$ be the quotient of $\allcurveshat$ modulo the $\gamma_0$-equivalence, or equivalently the subset of $\allcurveshat$ consisting of curves with a single component. The sets $\allcurves_0$ and $\allcurveshat_0$ have the structure of abelian groups, as we now describe.

We first note that the sets of knot-like complex $\complexes$ and $\hatcomplexes$ (over $\sR$ and $\sRhat$, respectively) are monoids, with operation given by the tensor product. Given the bijection between $\complexes$ and $\allcurves$, or between $\hatcomplexes$ and $\allcurveshat$, given in Theorem \ref{thm:bijection}, we can use the tensor product of complexes to define a corresponding monoid operation on knot-like curves. In analogy to the definition of $\allcurves_0$ and $\allcurveshat_0$ we define $\complexes_0$ and $\hatcomplexes_0$ to be the quotients of $\complexes$ and $\hatcomplexes$ under local equivalence (or if we prefer stable equivalence), which may also be identified with the subsets of $\complexes$ and $\hatcomplexes$ consisting of knot-like complexes that are irreducible with respect to direct sum. It is clear that the bijections $f$ and $\hat f$ in Theorem \ref{thm:bijection} restrict to bijections from $\complexes_0$ to $\allcurves_0$ and from $\hatcomplexes_0$ to $\allcurveshat_0$. It follows from \cite[Proposition 2.6]{Zemke:connected-sums-involutive} that $\complexes_0$ is an abelian group with operation $\otimes$, and the same holds for $\hatcomplexes_0$. The complex with a single generator in bigrading $(0,0)$ and vanishing differential is the identity element for the group $\complexes_0$ (or $\hatcomplexes_0$). Correspondingly, the sets of connected knot-like curves $\allcurves_0$ and $\allcurveshat_0$ also form groups; the operation is determined by $\otimes$ and the identity element is the multicurve $\Gamma$ consisting of the single curve $\gamma_{triv} = (\R/\Z)\times \{0\}$ graded by the constant function 0 and with no bounding chain. 

\begin{figure}
\includegraphics[scale = .8]{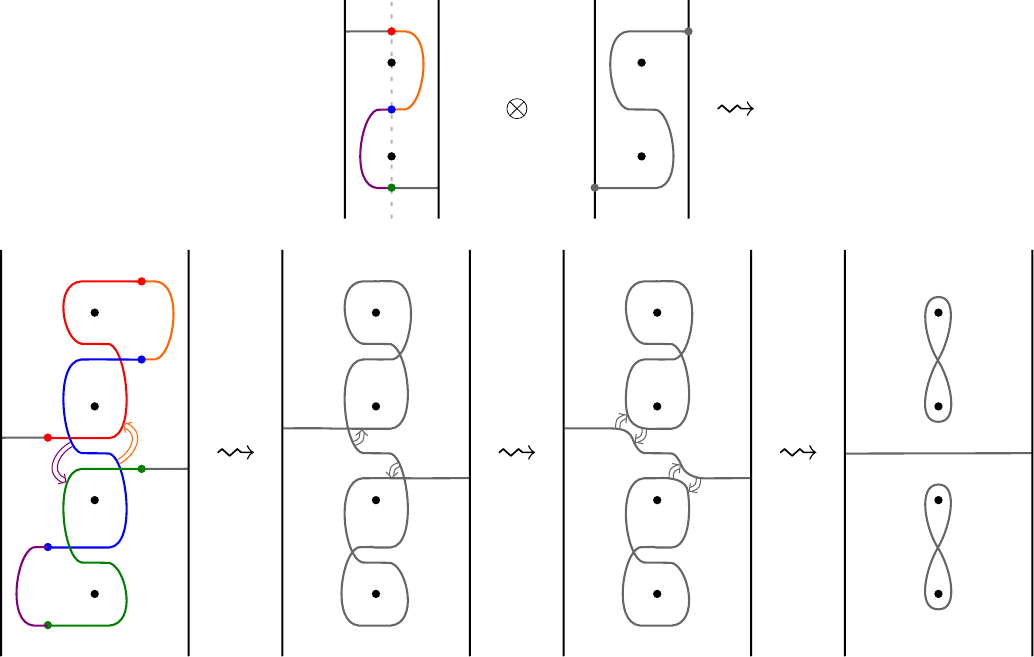}
\caption{The tensor product of two curves $\Gamma$ and $\Gamma'$ in $\allcurveshat$. The resulting multicurve with crossover arrows is formed by combining portions related to parts of $\Gamma$ as color coded in the figure; each intersection of $\Gamma$ with $\mu$ gives rise to a shifted copy of $\Gamma'$, and each right/left segment of $\Gamma$ gives rise to one right/left segment and also a crossover arrow for each right/left segment of $\Gamma'$. The rest of the figure depicts the arrow sliding algorithm in which crossover arrows are removed, at the expense of resolving two crossings, to produce a multicurve in $\allcurveshat$.}
\label{fig:tensor_product_example}
\end{figure}
In the $\sU=0$ case, it is also possible to define the group operation on $\allcurveshat_0$ geometrically. 

\begin{definition}\label{def:geometric-tensor} Given connected knot-like curves $\gamma$ and $\gamma'$ in $\allcurveshat_0$, we construct a new curve as follows:
\begin{enumerate}
\item For each $x$ in $\gamma \cap \mu$ draw a copy of $\gamma'$, shifted upwards by (the nearest integer to) the height of $x$, and decorated by the grading function $\tau'$ shifted upward by $\tau(x)$ (recall that $\tau(x)$ is an integer since $\gamma$ is in almost simple position and thus horizontal at intersections with $\mu$). Draw these copies of $\gamma'$ compressed horizontally (that is apply the map from the strip $\overline \sS_\bullet$ to itself that takes $(x,y)$ to $(\frac x 2, y)$) so that the loose ends are on $\{\pm\frac 1 4\}\times \R$.
\item For each right (resp. left) segment in $\gamma$ connecting points $x$ and $y$ in $\gamma \cap \mu$, we add a right (resp. left) segment in $[\frac 1 4, \frac 1 2) \times \R$ (resp. in $(-\frac 1 2, -\frac 1 4] \times \R$) connecting the right (resp. left) loose ends on the copies of $\gamma'$ corresponding to $x$ and $y$.
\item For each right (resp. left) segment of $\gamma$ connecting points $x$ and $y$ and each right (resp. left) segment $s$ of $\gamma'$, add a crossover arrow connecting segment $s$ in the copy of $\gamma'$ corresponding to $x$ and the segment $s$ in the copy of $\gamma'$ corresponding to $y$.
\item Run the arrow sliding algorithm mentioned in Section \ref{sec:structure-thm} to remove crossover arrows, resulting in a knot-like curve in $\allcurveshat$.
\item keep only the distinguished component $\gamma_0$ of the resulting curve to get a connected knot-like curve in $\allcurveshat_0$.
\end{enumerate}
\end{definition}
This procedure is demonstrated in Figure \ref{fig:tensor_product_example}. It is not difficult to check that train track produced by steps (1)-(3) precisely determines the tensor product of the complexes corresponding to $\gamma$ and $\gamma'$, so this definition agrees with defining $\otimes$ on $\allcurveshat_0$ via the operation $\otimes$ on $\hatcomplexes_0$.

\begin{remark}
Steps (1)-(4) described above in fact describe the monoid operation $\otimes$ on $\allcurveshat$, except that one also must describe how to handle nontrivial local systems, that is, nontrivial decorations $\bchainhat$; since $\bchainhat$ always vanishes when restricting to $\gamma_0$ this complication can be avoided when restricting the operations to $\allcurveshat_0$.
\end{remark}

The geometric description of $\otimes$ on $\allcurveshat_0$ is somewhat cumbersome due to its reliance on the arrow sliding algorithm. In particular, while there are some intuitive relationships between the geometric features of the input curves and features of the train track produced by steps (1)-(3), it is unclear what features survive step (4) which may include resolving self-intersection points. Moreover, due to the complications with arrow sliding in the minus setting (see Figure \ref{fig:minus-arrow-slides}) this description has yet to be extended to $\allcurves$ and there may not be a nice curve-based description of $\otimes$ on $\allcurves$. Thus in the minus setting, and even sometimes in the hat setting, we must resort to the bijection between curves and complexes and use algebraic methods when proving things about the group structure of curves. That said, it is an interesting exercise to prove what properties we can geometrically in the hat setting. First, we show that the group is abelian. Note that this is trivial in the algebraic setting, a curve-based proof is not difficult but requires understanding equivalences of train tracks.

\begin{proposition}\label{prop:tensor-product-commutes}
For any $\gamma$ and $\gamma'$ in $\allcurveshat_0$, $\gamma \otimes \gamma' \simeq \gamma' \otimes \gamma$.
\end{proposition}
\begin{proof}
There is a clear correspondence between $(\gamma \otimes \gamma')\cap\mu$ and $(\gamma' \otimes \gamma)\cap\mu$; in each case these intersection points correspond with pairs $\{x,y\}$ for $x$ in $\gamma\cap\mu$ and $y$ in $\gamma'\cap\mu$, we will use $p_{\{x,y\}}$ to denote the intersection point corresponding to $\{x,y\}$. If neither $x$ nor  $y$ is the end of a right segment, then $p_{\{x,y\}}$ is a loose end that is connected to the right boundary of the strip in either $\gamma \otimes \gamma'$ or $\gamma' \otimes \gamma$. If $x$ is a point in $\gamma\cap\mu$ that is not the end of a right arc in $\gamma$ and $y_1$ and $y_2$ are points in $\Gamma'\cap\mu$ connected by a right segment, then in either $\gamma \otimes \gamma'$ or $\gamma' \otimes \gamma$ the points $p_{\{x,y_1\}}$ and $p_{\{x,y_2\}}$ are connected by a right segment (coming from step (1) in $\gamma \otimes \gamma'$ and from step (2) in $\gamma'\otimes\gamma$) and no other right segments meet either point. The case of $x_1$ and $x_2$ in $\gamma\cap\mu$ connected by a right arc and $y$ in $\gamma'\cap\mu$ not on a right arc is similar. Now consider $x_1$ and $x_2$ in $\gamma\cap\mu$ that are connected by a right segment and $y_1$ and $y_2$ in $\gamma'\cap\mu$ that are connected by a right segment. In $\gamma \otimes \gamma'$ we have right segments connecting $p_{\{x_1,y_1\}}$ to $p_{\{x_1,y_2\}}$ and $p_{\{x_2,y_1\}}$ to $p_{\{x_2,y_2\}}$ and a crossover arrow connecting these two right segments, while in $\gamma' \otimes \gamma$ we have right segments connecting $p_{\{x_1,y_1\}}$ to $p_{\{x_2,y_1\}}$ and $p_{\{x_1,y_2\}}$ to $p_{\{x_2,y_2\}}$ and a crossover arrow connecting these two right segments. These two arrangements are equivalent by train track moves, since a crossover arrow moving the other direction can be added (this arrow is left-to-right) and the pair of crossover arrows can be replaced with a crossing and a crossover arrow as shown in Figure \ref{fig:commutativity}. This shows the right sides of the train tracks constructed to represent $\gamma \otimes \gamma'$ and $\gamma' \otimes \gamma$ are equivalent as train tracks, and analogous arguments apply to the left segments. Since the train tracks are equivalent, simplifying each in step (4) will result in the same element of $\allcurveshat$.
\end{proof}

\begin{figure}
\includegraphics[scale = .9]{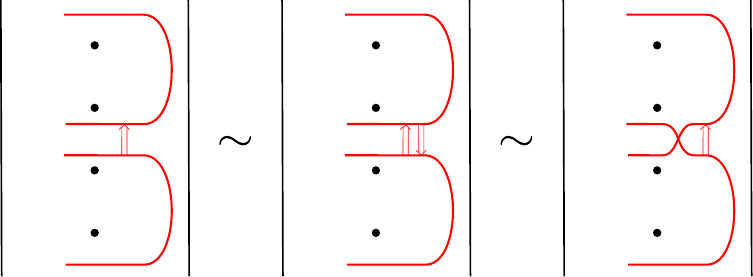}
\caption{Arrow sliding moves showing the equivalence of the train tracks constructed when computing $\gamma \otimes \gamma'$ or $\gamma' \otimes \gamma$.}
\label{fig:commutativity}
\end{figure}

We next consider inverses. In the example in Figure \ref{fig:tensor_product_example}, tensoring a curve with its mirror image gives rise to the identity curve in $\allcurveshat_0$; the fact that this always holds follows from Proposition \ref{prop:tensor-product-commutes}. For a curve $\gamma$ in $\allcurveshat_0$, let $m(\gamma)$ denote the curve obtained by reflecting $\gamma$ across the horizontal line $S^1 \times \{0\}$ and multiplying the grading function by $-1$.

\begin{proposition}\label{prop:trivial-component-with-mirror}
For any $\gamma$ in $\allcurveshat_0$, $m(\gamma)$ is the inverse of $\gamma$.
\end{proposition}
\begin{proof}
If both inputs to $\otimes$ are mirrored, the resulting train track from steps (1)-(3) in Definition \ref{def:geometric-tensor} is mirrored with the direction of crossover arrows reversed. One can check that the arrow sliding algorithm runs the same on the mirrored train track (the crossover arrows having reversed direction accounts for the fact that the role of left and right are switched when strands diverge), so $m(\gamma \otimes m(\gamma)) = m(\gamma) \otimes m(m(\gamma)) = m(\Gamma) \otimes \gamma$. This is equivalent to $\gamma \otimes m(\gamma)$ by Proposition \ref{prop:tensor-product-commutes}, so $\gamma \otimes m(\gamma)$ is fixed by mirroring. The identity element is the only connected immersed curve in $\allcurveshat_0$ with this property.
\end{proof}

In the hat setting, the group of connected immersed curves admits a total ordering. This ordering, introduced for the corresponding complexes by Hom in \cite{Hom:CFK-and-smooth-concordance}, can be described geometrically as an ordering from more leftward turning to more rightward turning curves.
\begin{proposition}\label{prop:ordering-on-curves}
The group $\allcurveshat_0$ admits a total ordering $<$ where $\gamma < \gamma'$ if, when $\gamma$ and $\gamma'$ are homotoped to coincide through as many intersections with $\mu$ as possible (starting from their left endpoints) and after that to intersect minimally with each other, $\gamma$ lies to the left of $\gamma'$ when they first diverge.\qed
\end{proposition}

Note that analogs of Proposition \ref{prop:tensor-product-commutes} and \ref{prop:trivial-component-with-mirror} hold for $\allcurves$, though as noted above one must appeal to the corresponding facts for $\complexes$ to prove them. Proposition \ref{prop:ordering-on-curves} does not extend to $\allcurves$, as the resulting ordering is only a partial ordering.

Because $\CFKr(K)$ satisfies a Kunneth formula---$\CFKr(K_1 \# K_2) \cong \CFKr(K_1) \otimes \CFKr(K_2)$---the map from the smooth concordance group to $\complexes_0$ given by the local equivalence class of $\CFKr(K)$ is in fact a homomorphism, and thus there is also a homomorphism from the smooth concordance group to $\allcurves_0$. In Section \ref{sec:concordance-invariants} we will see that many numerical invariants can be extracted from a curve in $\allcurves_0$ (or from a complex in $\complexes_0$), which are necessarily concordance invariants. In some cases these numerical invariants of knot-like curves actually give homomorphisms from $\complexes_0$ to $\Z$, and thus the corresponding concordance invariants give concordance homomorphisms.

\section{Concordance invariants extracted from $\Gamma_0$.}\label{sec:concordance-invariants}

In Section \ref{sec:gamma_0-invariants} we saw that many numerical invariants can be extracted from the distinguished curve component $\gamma_0(K)$, including the Ozsv{\'a}th-Szab{\'o} $\tau$ invariant, Hom's $\epsilon$ invariant, the related $\nu$ invariant, and the $\phi_i$ invariants. Since we have now seen that $\gamma_0(K)$ is a concordance invariant, the same is true of all of these as well. In this section we survey further concordance invariants that can be extracted from the minus theory version of the distinguished component, $\Gamma_0(K)$, and where possible we give geometric interpretations of them. Note that these invariants rely essentially on the full knot-like curve over $\sR$, and in particular on the bounding chain $\bchain$ in the decorated curve $\Gamma(K) = (\Gamma, \bchain)$ associated to the knot. This allows us to distinguish, for instance, the three knot-like complexes in Example \ref{ex:different-bchains} that restrict to the same complex over $\sRhat$, but it also makes computing the invariants harder both algebraically and geometrically. We caution that, unlike in the hat case, the geometric description of these invariants is still informal and in some cases somewhat speculative, we believe the descriptions here provide useful intuition for these invariants but more work needs to be done to make the geometric descriptions precise and see how useful they will be for studying these invariants.

\subsection{Invariants related to the $v_s$ maps}\label{sec:invariants-from-v_s}

While $\gamma_0$, which depends only on the $\sU = 0$ reduction $\widehat C$ of a knot-like complex $C$, is in general easier to compute and easier to work with than $\CFKr$, more information can be gleaned by working over $\sR$. In particular, the maps of $\F[\sU]$ complexes $v_s: C_s \to C^-_0$, or more precisely the induced maps on homology $v_{s,*}: H_* C_s \to H_* C^-_0 \cong \F[\sU]$, can be used to define invariants of the homotopy class of $C$. Note that since we only rely on the maps at the level of homology we may restrict to the summand of $C$ that supports its homology; that is, these invariants depend only on the distinguished component $\Gamma_0(C)$. The maps $v_{s,*}$ give rise to the following numerical invariants.
\begin{itemize}
\item For each integer $s$ we define the nonnegative integer $V_s(C)$ to be $\rank_\F(\coker v_{s,*})$ \cite{OzSz:integer-surgeries}.

\item We define $\nu^-(C)$ analogously to $\nu(C)$ using $v_{s,*}$ instead of $\hat v_{s,*}$. That is, 
$$\nu^-(C) = \min \{s | v_{s,*} \text{ is surjective}\}.$$
In other words, $\nu^-(C)$ is the minimum $s$ such that $V_s = 0$. Note that an analogous $\nu^+$ was first defined in \cite{HomWu}, and $\nu^-$ is equivalent to $\nu^+$ by \cite[Proposition 2.13]{OSS:upsilon}.

\item There are also invariants $H_s$ defined similarly using the maps $h_s$ instead of $v_s$, but we will not discuss them. For \emph{symmetric} knot-like complexes $H_s$ carries no further information since $H_s(C) = V_{-s}(C)$.
\end{itemize}

Note that $H_* C_s$ is always isomorphic to a single $\F[\sU]$ tower plus some number of $\sU$-torsion summands; this follows from the fact setting $\sU = 1$ in $C_*$ (which corresponds to ignoring the marked points in $\barT_u$) gives rise to rank 1 homology $H_* C_s$ (since $\alpha_s$ and $\Gamma$ intersect once algebraically). Also, $v_{s,*}$ cannot be the zero map but it necessarily restricts to the zero map on $\sU$-torsion summands, so we must have that the restriction to the tower summand is multiplication by $\sU^{V_s}$.

For symmetric knot-like complexes $C$ (i.e. those arising from knots) the invariants $V_s$ and $\nu^-$ satisfy the following properties:

\begin{itemize}
\item The sequence $\{V_s(C)\}$ is non-increasing \cite[Lemma 2.4]{NiWu}; 
\item $V_s (C) > 0$ for all $s < 0$ and thus $\nu^-(C) \ge 0$ with equality if and only if $V_0(C) = 0$ \cite[Proposition 2.3]{HomWu};
\item $\nu(C) \le \nu^-(C)$, and the difference $\nu^-(C) - \nu(C)$ can be made arbitrarily large \cite[Theorem 1]{HomWu};
\item If $C = \CFKr(K)$ for a knot $K$ then $\nu^-(C) \le g_4(K)$, where $g_4(K)$ is the 4-ball genus of $K$ \cite[Corollary 7.4]{Ras:knot-floer}.
\end{itemize}

For the rest of this section we will give describe a geometric procedure for extracting the invariants $V_s$ and $\nu^-$ from curves, using the geometric description of $v_{s,*}$ described in Section \ref{sec:v_s-maps}. This interpretation has not been described previously, but it is simply a translation of the algebraic definitions into curve pictures. We summarize the procedure informally below and illustrate it through examples, this procedure should be viewed as a heuristic for computing the invariants. The usefulness of this heuristic remains to be seen; it is likely to be useful in sufficiently nice examples, but unfortunately the geometric interpretation is not as clean as it is for $\nu$. The main difference is that the homology $H_* C_s$ is not always easy to read off of an immersed curve picture since bigons that cover the marked points contribute to the differential but cannot be removed by homotopy of the curves (in contrast, in the punctured surface $\barT_\bullet$ we can remove all bigons by putting curves in minimal position, so the intersection points of $\alpha_s$ and $\Gamma$ are precisely the generators of $\widehat H_* C_s$). It is still the case that elements of homology are represented by cycles in the Floer chain complex, which correspond to $\F[\sU]$-linear combinations of intersection points of $\Gamma$ with $\alpha_s$, and one of these cycles generates the tower in $H_* C_s$. The difficult step is identifying the intersection point(s) corresponding to the tower generator; in some cases this is immediately evident from the curve picture, but in general it may not be clear geometrically and may require algebraic effort. If we can identify the tower generator, then $V_s$ may be computed by counting the number of punctures covered by a bigon from a relevant intersection point to the unique intersection point on the green portion of $\mu\alpha_s$, and adding $k$ where the relevant intersection point occurs with coefficient $\sU^k$ in a cycle representing the tower in $H_* C_s$.

We demonstrate this procedure in several examples. First, Figure \ref{fig:nu-RHT} shows the computation of $V_s(C)$ when $C = \CFKr(RHT)$ and $-1 \le s \le 1$. In each case $\alpha_s$ can be homotoped to intersect $\Gamma$ exactly once, forming no differentials, so $H_* C_s$ is just a tower whose generator corresponds to the unique intersection point of the blue part of $\mu\alpha_s$ with $\Gamma$. In each case there is a bigon from this intersection to the unique intersection point of the green part of $\mu\alpha_s$ with $\Gamma$, and the number of marked points covered is $V_s$. We see that for this complex $V_{-1} = V_0 = 1$ and $V_1 = 0$. The generalization of these cases to arbitrary $s$ is clear: $V_s = -s$ if $s < 0$ and $V_s = 0$ if $s > 0$. From this information we see that $\nu^-(C) = 1$.

The next example, an abstract complex whose corresponding curve $\Gamma$ is shown in Figure \ref{fig:nu-example}, is slightly more complicated. If $s$ is sufficiently large or small there is a single intersection between $\alpha_s$ and $\Gamma$ and a single bigon to consider, as in the case of the trefoil. For example, when $s = -3$ there is a single bigon covering three marked points, so $V_{-3} = 3$. However when $-2 \le s \le 2$ there are three intersection points of $\alpha_s$ with $\Gamma$. Suppose these are labelled $a$, $b$, $c$ from left to right as in the figure, and let $x$ be the intersection of the green part of $\mu\alpha_s$ with $\Gamma$. When $s = -2$, the Floer chain complex of $\alpha_s$ with $\Gamma$ has $\partial a = \sU b$, $\partial b = 0$, and $\partial c = \sU^2 b$. The homology is isomorphic to a tower generated by $[c + \sU a]$ along with a copy of $\F$ generated by $b$. Note that in the Figure there is a bigon from $a$ to $x$ covering one marked point, so $v_{-2}(a) = \sU x$, but this does not mean that $\sU x$ is in the image of $v_{s,*}$ since $a$ is not a cycle. In fact $V_{-2}$ is determined by $v_{-2}( c + \sU a ) = v_{-2}( \sU a ) = \sU^2 x$. Similarly we compute $V_{-1}$ by considering how $v_{-1}$ acts on $\sU a + c$, whose homology class represents the generator of the tower in $H_* C_{-1}$. In this case there is a bigon from $a$ to $x$ covering zero marked points, so $V_{-1} = 1$. For $s=0$ we see that the tower is generated by $[a +c]$ and for $s = 1$ it is generated by $[a+\sU c]$; in each case there is a bigon from $a$ to $x$ with no marked points, so $V_0 = V_1 = 0$. It follows that $\nu^-(C) = 0$ in this example.

We remark that this example, while more complicated than the right-hand trefoil, is still fairly easy geometrically since $\Gamma_0 = \gamma_0$, $\gamma_0$ is embedded, and the intersections of it with $\alpha_s$ appear in the same order when following either $\gamma_0$ or $\alpha_s$. In such cases there are an odd number of generators connected by a chain of bigons that alternate being below or above $\alpha_s$, and we know that the generator of the tower in homology will be the class of the sum of every other generator (i.e. $a, c, e, \ldots$ if generators are labelled as above) each multiplied by some power of $\sU$, and $V_s$ is determined by a single bigon from $a$ to $x$. Thus it suffices to determine the power of $\sU$ attached to $a$ and add this to the number of marked points covered by the bigon from $a$ to $x$. The relevant power of $a$ can be determined by considering the string of bigons connecting it to other intersection points appearing in the tower generator; the difference in $\sU$ power attached to $a$ and another intersection point in the generator is the signed count of marked points covered by the intervening bigons, with bigons below $\alpha_s$ counting negatively, and the generator is such that the minimum power of $\sU$ attached to an intersection point is 0. It follows that $V_s$ is the maximum, over the intersection points $a, c, e, \ldots$ that appear in a cycle generating the tower, of the signed multiplicity with which the domain from the intersection point to $x$ covers the marked points.

Our final example, shown in Figure \ref{fig:nu-example2} is more complicated. We consider the three complexes from Example \ref{ex:different-bchains}, which correspond to the same multicurve $\Gamma$ with three different choices of bounding chain $\bchain$. Recall that $\bchain$ contains either the degree 1 self-intersection point $p_1$, the degree 2 self-intersection point $p_2$, or neither. In each case for each value of $s$ there is a unique intersection of $\alpha_s$ with the horizontal component of $\Gamma$, labelled $a$, and there is a bigon from that intersection point to $x$ that covers zero marked points if $s \ge 0$ and covers $-s$ marked points if $s < 0$. The challenge is determining if $a$ survives in homology. If $\bchain = \emptyset$ then there are no other bigons to or from $a$ or $x$, so it is clear $V_s = \max\{0, -s\}$. If $\bchain = \{p_2\}$ then there are no additional bigons in or out of $a$ so $V_s$ is still $\max\{0, -s\}$. Finally if $\bchain = \{p_1\}$ we note that when $-2 \le s \le 2$ there are also (generalized) bigons contributing $\sU c$ to $\partial a$ and contributing $\sU^{3 - |s|} c$ to $\partial b$. From this we see that the generator of the tower in $H_* C_s$ is $[b + \sU^{2-|s|} a ]$, and for each $s$ we add the exponent of $\sU$ in these generators to the values of $V_s$ computed when $\bchain = \emptyset$. Note that for this complex $\nu^- = 2$. This example shows that $V_s$ does not depend only on the distinguished curve component $\gamma_0$, since all three complexes have the same $\gamma_0$ component, but really requires the richer invariant $\Gamma_0$. This example also demonstrates that $\nu^-$ may be distinct from $\nu$, which only depends on $\gamma_0$ and is 0 for all three complexes in this case.

\begin{figure}
\includegraphics[scale = 1]{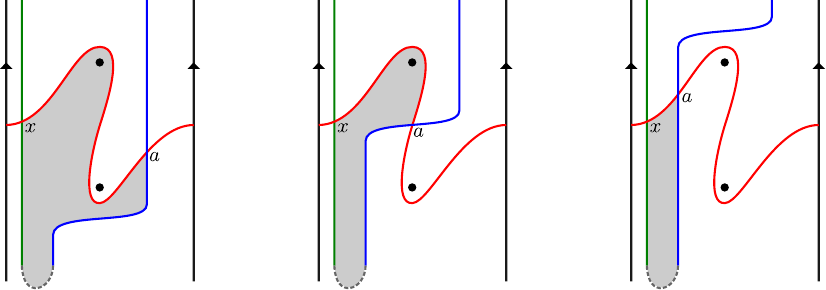}

\vspace{2mm}
$V_{-1} = 1$ \hspace{40 mm} $V_0 = 1$ \hspace{40 mm} $V_1 = 0$
\caption{Computing $V_s(RHT)$. In each diagram the blue curve is $\alpha_s$, the green curve curve is $\mu$, and when combined along with the dotted arc they form the curve $\mu\alpha_s$. The unique intersection $a$ between $\Gamma(RHT)$ and $\alpha_s$ generates the tower in $H_*C_0^-$, and $V_s$ counts the number of marked points enclosed in the shaded bigons.}
\label{fig:nu-RHT}
\end{figure}

\begin{figure}
\includegraphics[scale = .85]{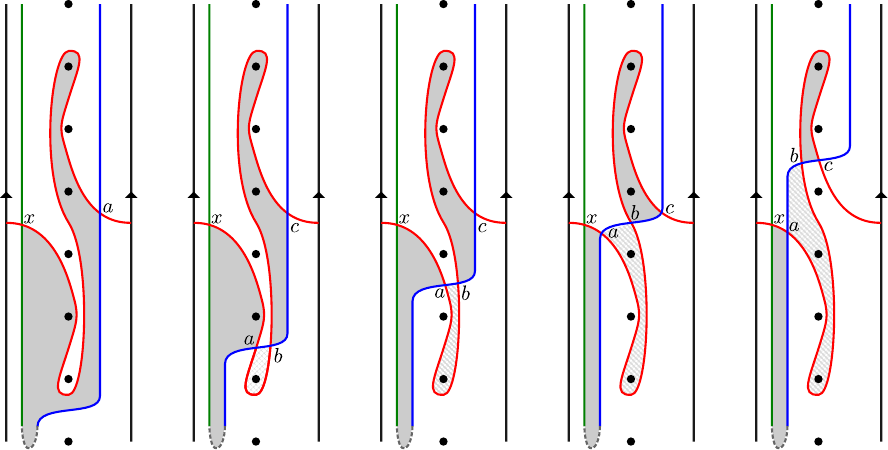}

\vspace{2mm}
$V_{-3} = 3$ \hspace{13 mm} $V_{-2} = 2$ \hspace{13 mm} $V_{-1} = 1$ \hspace{13 mm} $V_0 = 0$ \hspace{14 mm} $V_1 = 0$
\caption{Computing $V_s$ for a complex $C$ whose corresponding curve $\Gamma$ is shown in red. The tower in $H_* C_s$ is generated $a$ when $s = -3$ and in other cases by a sum of $a$ and $c$ with appropriate powers of $\sU$. In those cases $V_s$ is obtained by taking the maximum of the count of marked points in the bigon from $a$ to $x$ or the signed count of the shaded domain from $c$ to $x$, where the hatched regions have negative multiplicity.}
\label{fig:nu-example}
\end{figure}

\begin{figure}
\hspace{9mm}\includegraphics[scale = .75]{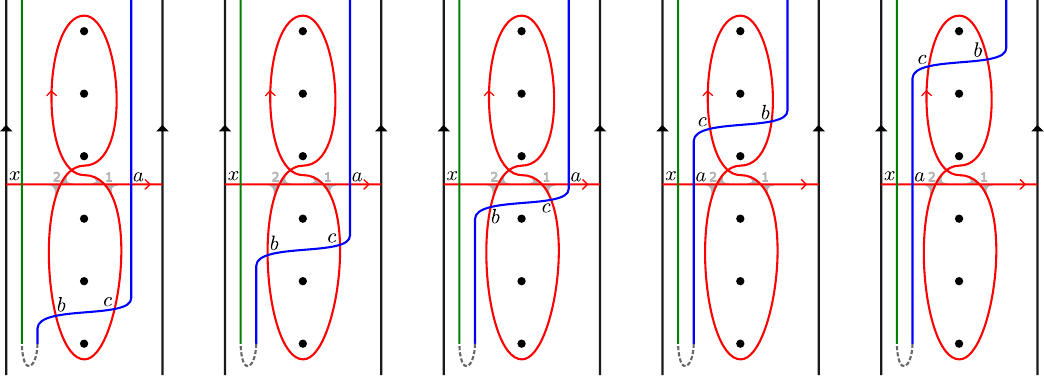}
\begin{tabular}{>{$}c<{$}  >{$}l<{$}   >{$}c<{$}  >{$}c<{$}  >{$}c<{$}   >{$}c<{$} }
\bchain & V_{-2} \hspace{1cm} &\hspace{1cm}  V_{-1} \hspace{1cm}  &\hspace{1cm}  V_{0} \hspace{1cm}  &\hspace{1cm}  V_1 \hspace{1cm}  & \hspace{1cm} V_2 \hspace{1cm} \\ [.5 em]
 \hline \\[-.7em]
\emptyset & \,\, 2 & 1 & 0 & 0 & 0  \\
\{p_2\} & \,\,2 & 1 & 0 & 0 & 0  \\ 
\hspace{4mm} \{p_1\} \hspace{4mm}& \,\,2 & 2 & 2 & 1 & 0  \\ [.3 em]
\hline
\end{tabular}
\caption{Computing $V_s$ for complexes $C$ represented the given multicurve with different bounding chain decorations for $-2 \le s \le 2$. In each case there is a bigon from $a$ to $x$ covering $\max\{0, -s\}$ marked points, but whether $a$ only determines a cycle in $H_* C$ when $\bchain$ is $\emptyset$ or $\{p_2\}$. When $\bchain = \{p_1\}$, the tower in $H_* C$ is instead generated by $b + \sU^{2-|s|} a$. }
\label{fig:nu-example2}
\end{figure}

\subsection{$d$-invariants of surgeries}

Another family of concordance invariants related to the $v_s$ maps are the $d$-invariants of surgeries on $K$. Given a rational homology sphere $Y$ and a spin$^c$ structure $\spin$, the Heegaard Floer homology $\HFminus(Y;\spin)$ is a tower $\F[\sU]$ plus torsion summands and is equipped with an absolute $\Q$-grading. We define the $d$-invariant $d(Y;\spin)$ to be the grading of an element that generates the tower. For any nonzero rational number $\frac p q$ and any $0 \le i < p$ the rational number $d( S^3_{p/q}(K); \spin_i )$ is a knot invariant, since $\HFminus(S^3_{p/q};\spin_i)$ can be obtained from the knot invariant $\CFKr(K)$ using the surgery formula. In fact, since the tower in the resulting homology is not affected by summands of $\CFKr(K)$ that do not support the homology, it is clear the result only depends on the local equivalence class of $\CFKr(K)$, or equivalently on the distinguished component $\Gamma_0(K)$. As such, the invariant $d( S^3_{p/q}(K); \spin_i )$ for any choice of $p, q$ and $i$ is a concordance invariant.

Ni and Wu gave a formula for computing $d( S^3_{p/q}(K); \spin_i )$ in terms of the invariants $V_s$ and $H_s$ \cite{NiWu}, which we now recall with the aid of immersed curves. We will assume that $p > 0$ (the case of negative surgeries can be obtained from this case by taking a mirror image). The first step is to identify an element in the mapping cone complex that generates the tower in $\HFminus(S^3_{p/q};\spin_i)$. If we first take homology of each complex $C_{s_j}$ and $C^-_{s_j}$ and ignore the torsion in $C_{s_j}$ we see that the tower comes from the homology of the following complex.
\begin{center}
\begin{tikzpicture}[scale = .8]
    \node at (1.5,1) {$\cdots$};
    \node (a1) at (2,2) {$\F[\sU]$};
    \node (b2) at (4,0) {$\F[\sU]$};
    \node (a2) at (6,2) {$\F[\sU]$};
    \node (b3) at (8,0) {$\F[\sU]$};
    \node (a3) at (10,2) {$\F[\sU]$};
    \node (b4) at (12,0) {$\F[\sU]$};
    \node (a4) at (14,2) {$\F[\sU]$};
    \node (b5) at (16,0) {$\F[\sU]$};
    \node (a5) at (18,2) {$\F[\sU]$};
    \node at (18.5,1) {$\cdots$};
     
    \draw[->] (a1) -- node[above right] {$\sU^{H_{s_{-2}}}$} (b2);
    \draw[->] (a2) -- node[right] {$\sU^{V_{s_{-1}}}$} (b2);
    \draw[->] (a2) -- node[above right] {$\sU^{H_{s_{-1}}}$} (b3);
    \draw[->] (a3) -- node[right] {$\sU^{V_{s_{0}}}$} (b3);
    \draw[->] (a3) -- node[above right] {$\sU^{H_{s_{0}}}$} (b4);
    \draw[->] (a4) -- node[right] {$\sU^{V_{s_{1}}}$} (b4);
    \draw[->] (a4) -- node[above right] {$\sU^{H_{s_{1}}}$} (b5);
    \draw[->] (a5) -- node[right] {$\sU^{V_{s_{2}}}$} (b5);
\end{tikzpicture}
\end{center}

Recall that the $s_j = \lfloor \frac{i + pj}{q} \rfloor$; in particular the $s_j$ are increasing and $s_j \ge 0$ if and only if $j \ge 0$. Using properties of the sequences $\{V_s\}$ and $\{H_s\}$ it can be shown that the generator of the tower in the homology of this complex is the generator of the tower at the beginning of the map $\sU^{V_{s_0}}$ if $V_{s_0} \ge H_{s_{-1}}$ or the generator of the tower at the beginning of the map $\sU^{H_{s_{-1}}}$ if $V_{s_0} \le H_{s_{-1}}$.

The situation is depicted for the 1-surgery on the right-hand trefoil in Figure \ref{fig:d-invts-of-surgery}, recall that the mapping cone complex is the Floer complex of $\Gamma(K)$ with a line of slope 1 if that line is perturbed into a particular position. The green portions of the curve correspond to the copies of $C^-_0$ in the mapping cone complex and the blue portions correspond to the $C_{s_j}$ portions. It is not always easy to identify a generator of the tower of $H_* C_{s_j}$, but in this case it is since each blue portion can be arranged to have a single intersection with $\Gamma(K)$ which must generate the tower. In this case the generators are connected by bigons covering certain numbers of marked points that determine the maps $v_{s_j}$ and $h_{s_j}$ which are multiplication by $\sU^{V_{s_j}}$ and $\sU^{H_{s_j}}$, respectively. By the discussion above, the generator of the overall homology comes from the generator on either the last blue section crossing the vertical line through the punctures at negative height or the first blue section crossing at nonnegative height, and since $V_0 = 1 > 0 = H_{-1}$ it must be the latter; i.e. the intersection point marked in red corresponds to the tower $\HFminus(S^3_1(K))$. This can be checked geometrically in this example since it is the only intersection that survives pulling the blue/green curve tight, though we do not give a geometric proof that the correct generator can always be found in this way in analogy to Ni and Wu's algebraic work.

\begin{figure}
\includegraphics[scale = .9]{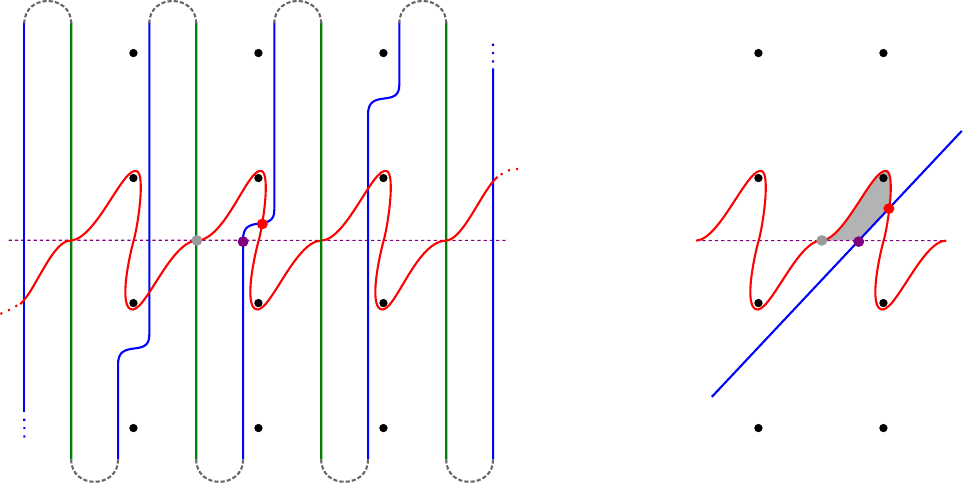}
\caption{Computing the $d$-invariant of $S^3_1(RHT)$, which is the Maslov grading of the intersection point indicated with a red dot. The pair of curves on the left gives the $\CFminus(S^3_1(RHT))$ exactly as it arises from the surgery formula, while the pulled-tight arrangement on the right makes it clear that $\HFminus(S^3_1(RHT))$ consists of a single tower $\F[\sU]$ generated by the red dot. We can find the grading relative to that of the purple dot, which represents the generator of $\HFminus$ of $1$-surgery on the unknot, by counting marked points in the shaded immersed triangle.}
\label{fig:d-invts-of-surgery}
\end{figure}

It remains to determine the absolute $\Q$-grading of the generator we have identified. Relative gradings can be computed from a curve diagram, but we have not yet discussed absolute gradings. For the surgery formula, absolute gradings are set for the $C^-_0$ pieces by the convention that the surgery formula should give the correct gradings when applied to the unknot (the $d$-invariants for surgeries on the unknot can be determined by an explicit formula). Thus it is most convenient to compute the desired $d$-invariant relative to the $d$-invariant of the corresponding surgery on the unknot $U$; that is, we compute $d( S^3_{p/q}(K); i) -  d( S^3_{p/q}(U); i)$. If we compare $\Gamma(K)$ to the curve associated with the unknot (the purple dotted line in the figure), by assumption the grading of the two curves agree at the grey dot is such that the intersection of purple and blue near the red dot (marked with a purple dot) has grading $0$. The grading difference between the generators corresponding to the purple dot and the red dot is twice the number of marked points enclosed by the triangular region with corners at the gray, purple, and red dots, or equivalently twice the number of marked points enclosed by the bigon connecting the red dot to the gray dot that determines the map $v_{s_0}$, i.e. the difference is $2V_{s_0}$ (recall that in this example $V_{s_0} > H_{s_{-1}}$---if instead $H_{s_{-1}}$ were bigger we would consider a bigon to the left of the grey dot corresponding to the map $h_{s_{-1}}$. Thus we have sketched pictorially the proof of Ni and Wu's formula, which says that 
$$d( S^3_{p/q}(K); i) -  d( S^3_{p/q}(U); i) = -2\max( V_{s_0}, H_{s_{-1}} ).$$
In this example, this difference in $d$-invariants can also be computed when the curves are pulled tight by counting the marked points covered by the shaded region in Figure \ref{fig:d-invts-of-surgery}. Since $d( S^3_{1}(U); 0) = 0$, we find that $d( S^3_{1}(K); 0) = -2$.

\subsection{Generalized $V_s$ invariants}\label{sec:other-inclusion-maps-invariants}

More generally the construction of the invariants $V_s$ and $\nu^-$ can be copied using any of the inclusion maps $v_{(s_1, \ldots, s_k)}$ discussed in Section \ref{sec:other-inclusion-maps}.
\begin{definition}\label{def:generalized-Vs}
For a knot-like complex $C$ and a finite or infinite increasing sequence of integers $\vec s = (s_1, \ldots, s_k)$ or $\vec s = (s_1, s_2, \ldots)$, we define
$$V_{\vec s} = \dim_\F \coker( v_{\vec s, *}: H_*(C_{\vec s}) \to H_*(C^-_0).$$
\end{definition}
The integers $V_{\vec s}$, which we will refer to as \emph{generalized $V_s$ invariants}, are invariants of the knot-like complex $C$. In fact, they depend only on $\Gamma_0(C)$ for the same reason the $V_s$ invariants do. While generalized $V_s$ invariants have not been defined before, some invariants equivalent to or determined by $V_{\vec s}$ for particular choices of $\vec s$ have been defined. We may view each particular choice of $\vec s$ as an attempt to glean more partial information from $\Gamma_0(K)$.

An example of this that has appeared in the literature are the invariants $Y_n$ defined by Juh{\'a}sz and Zemke in \cite{JuhaszZemke2020}. These are, by definition, $V_0$ of the complex $C \otimes S_{(-n, -n+1, \ldots, n-1, n)}$. Equivalently, these are $V_{(-n, -n+1, \ldots, n-1, n)}$ for the complex $C$. These are computed for some examples in Figure \ref{fig:Yn}. For the left-hand trefoil curve, it is clear from the figure that $Y_1 = 0$ and $Y_2 = 1$. For the curve from Example \ref{ex:nontrivial-bchain}, we find that $Y_1 = Y_2 = 2$. Note that $H_* C_{(-1,0,1)}$, relevant for computing $Y_1$, has five generators corresponding to the intersection points labeled in the figure, with differential
$$\partial(a) = \sU^2 b, \partial(b) = 0, \partial(c) = \sU b, \partial(d) = \sU e, \text{ and } \partial(e) = 0.$$
Thus the class of $a + \sU c$ generates the tower. The shaded bigon shows that $v_{(-1,0,1)}(c)$ is $\sU$ times the generator $x$ of $H_* C^-_0$ and there are no bigons from $a$ to $x$, so $v_{(-1,0,1)}(a + \sU c) = \sU^2 x$ and $Y_1 = V_{(-1,0,1)} = 2$. To compute $Y_2$, note that the homology of the complex $C_{(-2,-1,0,1,2)}$ also has five generators with differential 
$$\partial(a) = \sU b, \partial(b) = 0, \partial(c) = \sU b + \sU d, \partial(d) = 0, \text{ and } \partial(e) = \sU d.$$
The class of $a + c + e$ generates the tower, and there are bigons from each of $a$, $c$, and $e$ to $x$ each covering one marked point (the bigon from $c$ to $x$ is shaded in the figure) which shows that $Y_2 = V_{(-2,-1,0,1,2)} = 2$. Note that when the tower is generated by a sum of multiple generators, it suffices to find a bigon from any one of these intersection points to $x$; if there are multiple bigons as in this example gradings force them all to determine the same value for $Y_n$.

\begin{figure}
\includegraphics[scale = 1]{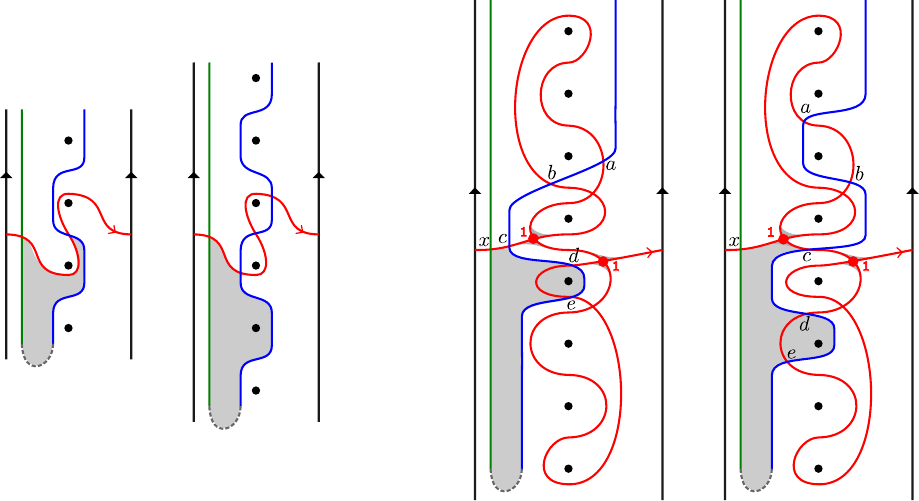}
\vspace{2mm}
$Y_1 = 0$ \hspace{20 mm} $Y_2 = 1$ \hspace{40 mm} $Y_1 = 2$ \hspace{30 mm} $Y_2 = 2$
\caption{The computation of $Y_1 = V_{(-1,0,1)}$ and $Y_2 = V_{(-2,-1,0,1,2)}$ for the left-hand trefoil curve (left) and the curve from Example \ref{ex:nontrivial-bchain} (right).}
\label{fig:Yn}
\end{figure}

\subsection{The $\Upsilon$ invariant} The final invariant we will consider is $\Upsilon_K$, defined by Ozsv{\'a}th, Stipsicz and Szab{\'o} in \cite{OSS:upsilon}. More generally, the definition can be applied to any knot-like complex $C$ to give an invariant $\Upsilon_C$ of the complex, where $\Upsilon_K = \Upsilon_{\CFKr(K)}$. The invariant takes the form of a piecewise linear function $\Upsilon_C: [0,2] \to \R$. It is easiest to define the restriction of $\Upsilon_C$ to $[0,2]\cap \Q$ (which is $\Q$ valued); the function can then be extended to irrational numbers by piecewise linearity. In particular, for any fixed rational $t = \frac m n$ in $[0,2]$ we have a rational invariant $\Upsilon_C(t)$. Though not originally defined in this way, we will show that $\Upsilon_C(t)$ is determined by the generalized $V_s$ invariants defined in Section \ref{sec:other-inclusion-maps-invariants}.

The invariant $\Upsilon_C(t)$ is defined using an algebraic modification of the knot-like complex of $C$, the \emph{$t$-modified complex} $tC$. This is a complex over the ring $F[ v^{\frac 1 n} ]$, equipped with a single grading $gr_t$ for which $v$ has degree $-1$. It has the same generators as $C$, with grading given by
$$\gr_t = \left(1 - \frac t 2\right) \gr_w + \frac t 2 \gr z,$$
and the differential is obtained from the differential on $C$ by setting $Z = v^{\frac m n}$ and $W = v^{2 - \frac m n}$. This modified complex can be extracted from the immersed multicurve $(\Gamma, \bchain)$ in $\overline \sS_u$ representing $C$ by replacing each marked point in $\overline \sS_u$ with $2n$ marked points at the same height with $x$-coordinates $\frac{-n + 1 + 2i}{2n} \epsilon$ for $0 \le i < 2n$ (see Figure \ref{fig:upsilon-example}). Just as $C$ is the Floer complex of $(\Gamma, \bchain)$ with $\mu$ in the doubly punctured strip $\overline S_{w,z}$, the $t$-modified complex $t C$ is the Floer complex of $(\Gamma, \bchain)$ with $\mu_{(\frac m n - 1)\epsilon} = \{(\frac m n - 1)\epsilon\}\times \R$, the vertical line that has exactly $m$ out of each group of $2n$ marked points on its left. Thus for a fixed denominator $n$, finding the $t$-modified knot Floer complex as the numerator increases from 0 to $2n$ corresponds to sliding the vertical line rightward from $\mu_{-\epsilon}$ to $\mu_\epsilon$. The grading function on $\Gamma$ is unchanged outside the strip $[-\epsilon, \epsilon]\times \R$, but inside this strip we now consider vertical cuts from each of the new marked points and the grading jumps by $\frac 1 n$ at each of them, thus spreading a single grading jump by $2$ in the original picture out over $2n$ smaller jumps; since the generators of $tC$ come from intersections in the middle of the interval $[-\epsilon, \epsilon]\times \R$ these fractional jumps are needed to determine the (rational) grading $\gr_t$.

\begin{figure}
\includegraphics[scale = 1]{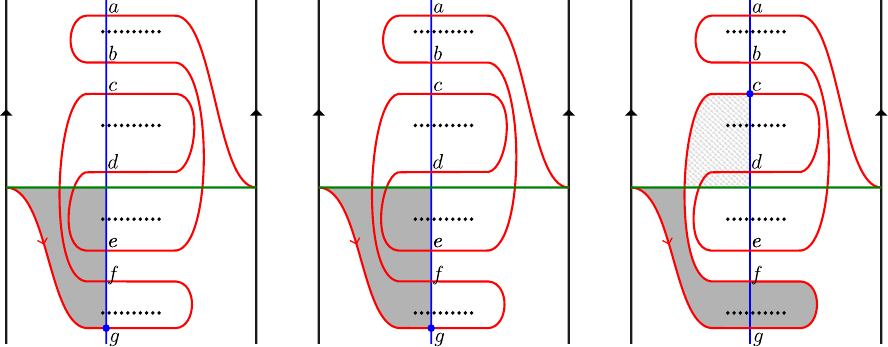}
\vspace{2mm}
$t = \frac 1 5$ \hspace{40 mm} $t = \frac 3 5$ \hspace{40 mm} $t = \frac 4 5$
\caption{Computing $\Upsilon_K(t)$ where $K= C_{2,-1}(LHT) $ for three values of $t = \frac m n$. Each marked point in $\overline \sS_u$ is replaced with $2n$ marked points, each of which contributes $\frac 1 n$ instead of the usual 2 in grading computations. The vertical line $\mu$ is shifted horizontally such that exactly $m$ of each group of $2n$ marked points lies to its left. One can check that the grading of the tower in homology agrees with the grading of the indicated generator, which can be computed as $\frac 1 n$ times the signed count of marked points covered by the shaded regions. }
\label{fig:upsilon-example}
\end{figure}

The $t$-modified complex gives rise to $t$-modified homology $t H_*C$, a finitely generated module over $\F[v^{\frac 1 n}]$, which has the form of a tower $\F[v^{\frac 1 n}]$ plus torsion summands (this uses that $C$ is knot-like). We define $\Upsilon_C(t)$ to be the grading $\gr_t$ of the generator of the tower. To extract $\Upsilon_C(t)$ from the immersed curve picture, we must first identify the intersection point(s) corresponding to the generator of the tower; note that the generator of this tower in homology will always be represented by a cycle, which is a sum of generators of the complex (i.e. intersection points in the curve diagram) times appropriate powers of $v$. As discussed elsewhere in this section, this is where the immersed curve approach is limited; it is sometimes possible in small examples to identify the correct intersection points just from the pictures, but in general some algebraic legwork is needed at this step. Once an appropriate cycle has been identified, the grading of that cycle can be computed from the immersed curve diagram. The desired grading of the generator of the tower is the grading of any term in the cycle, and there must be at least one generator of the complex appearing in the cycle with coefficient $v^0$ so it suffices to compute the grading of that generator. If the relevant intersection point lies on $\gamma_0$ then the grading is given by $\frac 1 n$ times the signed count of marked points covered by the three sided domain bounded by a path in $\gamma_0$ from the left endpoint of $\gamma_0$ to the relevant generator, a vertical segment of $\mu$ and a horizontal segment from $\mu$ to the left endpoint of $\gamma_0$. Note that the relevant generator need not lie on $\gamma_0$; if not we perform the same procedure using a path from the left endpoint of $\gamma_0$ to the generator in the train track associated with $(\Gamma, \bchain)$, and the grading jumps of any crossover arrows used in that path must be taken into account.

Figure \ref{fig:upsilon-example} computes $\Upsilon_K(t)$ where $K$ is the $(2,-1)$ cable of the left-hand trefoil and three values of $t$. The differential of the $t$-modified complex $tCFK(K)$ in each case is given in the following table.
$$\setlength{\arraycolsep}{30pt}\begin{array}{lll}
t = \frac 1 5 & t = \frac 3 5 & t = \frac 4 5 \\ [.5 em]
\hline \\ [-.5 em]
\partial a = v^{\frac 1 5} b & \partial a = v^{\frac 3 5} b & \partial a = v^{\frac 4 5} b \\
\partial b = 0 & \partial b = 0 & \partial b = 0 \\
\partial c = v^2 b + v^{\frac 2 5} f & \partial c = v^2 b + v^{\frac 6 5} f  & \partial c = v^2 b + v^{\frac 8 5} f  \\
\partial d = v^{\frac 9 5} c + v^{\frac 1 5} e & \partial d = v^{\frac 7 5} c + v^{\frac 3 5} e & \partial d = v^{\frac 6 5} c + v^{\frac 4 5} e \\
\partial e = v^{\frac{18}{5}} b + v^2 f & \partial e = v^{\frac{14}{5}} b + v^2 f  & \partial e = v^{\frac{12}{5}} b + v^2 f \\
\partial f = 0 & \partial f = 0 & \partial f = 0 \\
\partial g = v^{\frac 9 5} f & \partial g = v^{\frac 7 5} f & \partial g = v^{\frac 6 5} f
\end{array}$$

When $t = \frac 1 5$, one can check that the tower (the free $\F[v^{\frac 1 5}]$ summand) in homology is generated by the cycle
$x_{\frac 1 5} = g + v^{\frac 7 5} c + v^{\frac{16}{5}} a$. Since the shaded region in the figure covers two marked points, we see that
$$\Upsilon_K\left(\frac 1 5\right) = \gr_t(x_{\frac 1 5}) = \gr_t(g) = \frac 2 5.$$
When $t = \frac 3 5$ the tower in $t HFK$ is generated by $x_{\frac 3 5} = g + v^{\frac 1 5} c + v^{\frac{8}{5}} a$ and we find that
$$\Upsilon_K\left(\frac 3 5\right) = \gr_t(x_{\frac 3 5}) = \gr_t(g) = \frac 6 5.$$
When $t = \frac 4 5$ the tower in $t HFK$ is generated by $x_{\frac 4 5} = c + v^{\frac 2 5} g + v^{\frac 6 5} a$ and we find that
$$\Upsilon_K\left(\frac 4 5\right) = \gr_t(x_{\frac 4 5}) = \gr_t(c) = \frac 6 5.$$

It is worth noting that parts of the computation above can be done independent of $t$, giving a systematic approach to computing $\Upsilon_K$ as a function. Note that as $t$ varies the differential is unchanged except for the powers of $v$; if we ignore powers of $v$ (that is, set $v = 1$ and remove all marked points from the immersed curve diagram) to get a complex over $\F$ we see that the cycles in grading 0 (which supports the homology) are $0$, $c+e$, $a+c+g$, and $a+e+g$, and the image of the differential in this grading is the span of $c+e$, so the generator of homology is represented by either $a+c+g$ or $a+e+g$. For each value of $t$ homology classes in the tower summand are represented by either $a+c+g$ or $a+e+g$ with powers of $v$ added as needed to make each element homogeneous, and the generator of the tower is the cycle of this form with maximal grading. For a fixed cycle representing the nonzero homology class of the $\F$-complex, the maximum grading comes when one or more generators appearing in the cycle are multiplied by $v^0$; these are necessarily the generator(s) with minimal $\gr_t$ among those appearing in the cycle. Thus the minimal grading of all generators appearing in a given cycle is a candidate for the grading of the generator of the tower, and we maximize this over all the cycle representatives we enumerated. Recall that for each generator $x$ of $\CFKr$ the grading $gr_t = M(x) - t A(x)$. Figure \ref{fig:upsilon-example2} shows $gr_t$ as a function of $t$ for the relevant generators of $tCFK$ in the example above, namely $a$, $c$, $e$, and $g$. The middle two panels show the grading of the top element associated with the cycles $a+c+g$ and $a+e+g$, respectively, and the maximum of these two functions, shown in the fourth panel, is the function $\Upsilon_K(t)$.

\begin{figure}
\includegraphics[scale = .8]{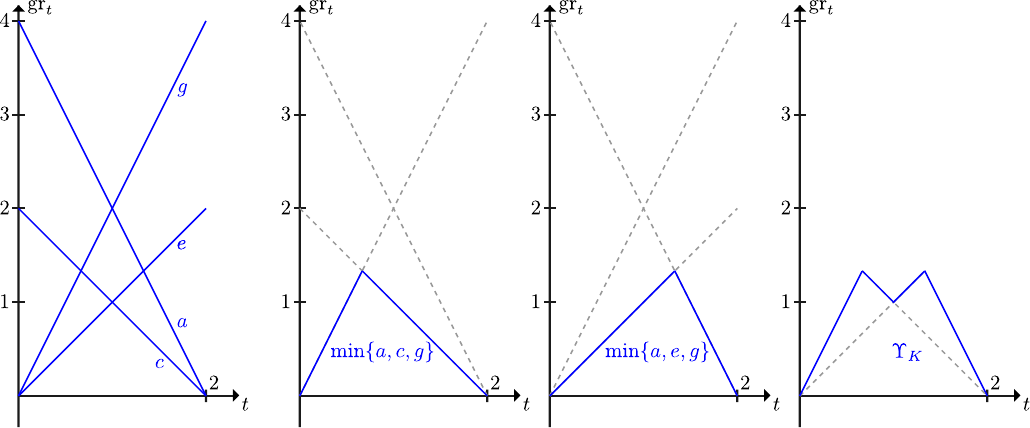}
\caption{Computing the function $\Upsilon_K(t)$ when $K$ is the $(2,-1)$ cable of the left-hand trefoil. We first enumerate all cycles representing the nontrivial homology class in $CFK_\F(K)$; these are $a+c+g$ and $a+e+g$. For each cycle representative we take the minimum grading of the generators involved and then take the maximum of this over all cycle representatives.}
\label{fig:upsilon-example2}
\end{figure}

To relate $\Upsilon_K$ to the generalized $V_s$ invariants we use an alternative description, due to Livingston \cite{Livingston:upsilon}. Recall that a knot-like complex $C$ admits two $\Z$-filtrations given by the negative powers of $W$ and $Z$; that is, the bifiltration level $(i,j)$ contains all elements of the form $W^a Z^b x$ where $x$ is a generator of $C$, $a \ge -i$ and $b \ge -j$. These filtrations extend to filtrations on $C^\infty = C \otimes \F[W, W^{-1}, Z, Z^{-1}]$, and if we restrict to the Alexander grading 0 summand $C^\infty_0$ of $C^\infty$, which when $C = \CFKr(K)$ is isomorphic to $CFK^\infty(K)$ in the traditional notation, the filtrations determined by powers of $W$ and $Z$ are precisely the filtrations denoted $Alg$ and $Alex$ in \cite{Livingston:upsilon}. For a fixed $t$, we define a new rational filtration $\sF_t$ on $C^\infty_0$ such that the filtration level $\sF_t(x)$ of an element is given by $\sF_t(x) = \frac t 2 Alex(x) + (1 - \frac t 2) Alg(x)$. If $t = \frac m n$, note that $\sF_t(x)$ takes values in $\frac{1}{2n} \Z$ if $m$ is odd and in $\frac 1 n \Z$ if $m$ is even. Livingston shows that $\Upsilon_C(t)$ can also be defined as $-2$ times the minimum filtration level $s$ such that inclusion induces a surjection on the $H_0$. Note that $H_*C^\infty_0$ is always a copy of $\F[\sU, \sU^{-1}]$, with an element in each even grading, so $\Upsilon_C(t)$ is $-2$ times the minimum $s$ such that the $s$th filtration level $(C^\infty_0, \sF_t)_s$ contains an element representing the homology class in Maslov grading 0.

To use immersed curves we prefer to work with $C$ rather than $C^\infty$; this can be done by working in sufficiently low filtration levels. We first observe that
\begin{equation}\label{eq:increment-s}
(C^\infty_0, \sF_t)_{s-1} = \sU \cdot (C^\infty_0, \sF_t)_{s},
\end{equation}
 and so for arbitrary $k$ we have
\begin{align*}
\Upsilon_C(t) &= -2 \min\{ s | i_*: H_0 (C^\infty_0, \sF_t)_s \to H_0 C^\infty_0 \text{ is surjective} \} \\
 &= -2 \min\{ s | i_*: H_{-2k} (C^\infty_0 \sF_t)_{s-k} \to H_{-2k} C^\infty_0 \text{ is surjective} \} \\
  &= -2 \min\{ s + k | i_*: H_{-2k} (C^\infty_0, \sF_t)_{s} \to H_{-2k} C^\infty_0 \text{ is surjective} \}.
 \end{align*}
Here $k$ is a fixed integer and it suffices to let $s$ run over all elements of $\frac{1}{2n} \Z$, or just of $\frac 1 n \Z$ if $m$ is even. For sufficiently low filtration levels (i.e. $s \le \bar s$ for some $\bar s$), $C^\infty_0$ and $C_0$ agree since $(C^\infty_0, \sF_t)_s$ contains no elements with $Alg(x)$ or $Alex(x)$ positive. For sufficiently large $k$ there will be some $s \le \bar s$ making the relevant map a surjection, so we may ignore $s > \bar s$ and pass to $C_0$. We then have that for sufficiently large $k$
$$\Upsilon_C(t) = -2 \min\{ s + k | i_*: H_{-2k} (C_0, \sF_t)_{s} \to H_{-2k} C^-_0 \text{ is surjective} \}. $$

The complex $(C_0, \sF_t)_{s}$ is the subcomplex of $C_0$ generated by elements $W^a Z^b x$ with $\sF_t(W^a Z^b x) = (1-\frac t 2)(-a) + \frac t 2 (-b) \le s$. When $C_0$ is viewed in the left half-plane, this is the subcomplex generated by all elements on or below the line with slope $(1 - \frac 2 t)$ and intercept $\frac{2s}{t}$. Applying the floor function to the $y$-coordinate turns this line into an infinite staircase, so the complex $(C_0, \sF_t)_{s}$ may be written as $C_{(i_1, i_2, \ldots)}$ for some sequence $(i_1, i_2, \ldots)$ determined by the slope and intercept of the line. For example, if $t = \frac 2 3$ the slope of the relevant line is $-2$, and the intercept is $3s$. Figure \ref{fig:upsilon-example3} shows the relevant lines for $s = -2$, $s= -\frac 5 3$, and $s = -\frac 4 3$, from which we can see that
\begin{align*}
(C_0, \sF_{\frac 2 3})_{-2} &= C_{(-6, -5, -3, -2, 0, 1, 3, 4, \ldots)} \\
(C_0, \sF_{\frac 2 3})_{-\frac 5 3} &= C_{(-5, -4, -2, -1, 1, 2, 4, 5, \ldots)} \\
(C_0, \sF_{\frac 2 3})_{-\frac 4 3} &= C_{(-4, -3, -1, 0, 2, 3, 5, 6, \ldots)}
\end{align*}
Given $t$ and $s$, let $\text{SC}(t,s) = (i_1, i_2, \ldots)$ denote the relevant sequence of staircase coefficients so that $(C_0, \sF_t)_s = C_{ \text{SC}(t,s) }$. The inclusion map of $(C_0, \sF_t)_s$ into $C_0^-$ is now that map $v_{\text{SC}(t,s)}$ as defined in Section \ref{sec:other-inclusion-maps-invariants}. The homology of the complex $C_0^-$ is a tower $\F[\sU]$ generated in Maslov grading 0, and the map $v_{\text{SC}(t,s), *}$ induced by inclusion is surjective in grading $-2k$ if and only if $V_{\text{SC}(t,s)}(C) = \rank_\F(\coker v_{\text{SC}(t,s), *} ) \le k$. It follows that
\begin{equation}\label{eq:upsilon-as-minimum}
\Upsilon_C(t) = -2 \min\{ s + k | V_{\text{SC}(t,s)}(C) \le k\}
\end{equation}
where $k \gg 0$ is a fixed integer and $s$ ranges over values of $\frac{1}{2n} \Z$ at or below the threshold $\bar s$.

We note that the minimum in \eqref{eq:upsilon-as-minimum} is independent of $k$, provided $k$ is large enough that the set being minimized over is nonempty); this follows from the relation
\begin{equation}\label{eq:V_SC-shift}
V_{\text{SC}(t,s-1)} = V_{\text{SC}(t,s)} + 1 \text{ when } s \le \bar s,
\end{equation}
which in turn follows from \eqref{eq:increment-s} (restricted to $s \le \bar s$ so that we can replace $C^\infty_0$ with $C_0$). In particular, \eqref{eq:V_SC-shift} gives that increasing $k$ by 1 decreases the minimum $s$ for which $V_{\text{SC}(t,s)}(C) \le k$ by 1, preserving the minimum of $s+k$. From this we see that minimum in \eqref{eq:upsilon-as-minimum} we could just as well allow $k$ to range over all integers. Then, by first minimizing over all $k$ values for each fixed $s$ we see that
\begin{equation}\label{eq:upsilon-as-minimum2}
\Upsilon_C(t) = -2 \min\left\{ s + V_{\text{SC}(t,s)}(C) | s \in \frac{1}{2n} \Z, s \le \bar s \right\}.
\end{equation}
Finally, note that by \eqref{eq:V_SC-shift} it suffices to consider one representative of each congruence class mod 1 of $s$. This leads to the following characterization of $\Upsilon_C(t)$ in terms of generalized $V_s$ invariants.
\begin{proposition}
Let $C$ be a knot-like complex over $\sR$ and let $g(C)$ denote the maximum Alexander grading of a generator of $C$. For $t = \frac m n$, let $s_0$ be any integer less than $-\frac t 2 g(C) - 1$ and let $s_i = s_0 + \frac{i}{2n}$ for $0 < i < 2n$. Then
$$\Upsilon_C(t) = -2 \min\{ s_i + V_{\text{SC}(t,s_i)}(C) \}_{i=0}^{2n-1}.$$
Moreover, we may restrict to even values of $i$ if $m$ is even.
\end{proposition}
\begin{proof}
The proof is mostly contained in the discussion preceding the statement. We can quantify the notion of sufficiently low filtration level by noting that the key property we need is \eqref{eq:V_SC-shift}, and this holds as long as the $y$-intercept of the line defining the subcomplex $C_{\text{SC}(t,\bar s)}$ is at or below the lowest generator on the vertical axis, which occurs at height $-g(C)$; the intercept of that line is $\frac 2 t \bar s$, so this occurs for $\bar s = -\frac t 2 g(C)$. We pick $s_0 < \bar s - 1$ to ensure that each $s_i$ is less than $\bar s$. There is one $s_i$ in each equivalence class mod 1, so by \eqref{eq:V_SC-shift} we can restrict to these values of $s_i$ when taking the minimum in \eqref{eq:upsilon-as-minimum2}. The last statement holds because when $m$ is even the filtration function takes values in $\frac 1 n \Z$.
\end{proof}

Sometimes it is not even necessary to compute $V_{\text{SC}(t,s_i)}$ for all values of $i$. Since $V_SC(t,s)$ is a nonincreasing integer valued function satisfying \eqref{eq:V_SC-shift}, the minimum occurs at the unique choice of $0\le i < 2n$ for which $V_{\text{SC}(t,s_i)} \neq V_{\text{SC}(t,s_{i-1})}$.

As an example, when $K$ is the $(2,-1)$-cable of the left-hand trefoil and $t = \frac 2 3$, we can let $s_0 = -2$. Since $m = 2$ is even, we can consider only $s_0 = -2$, $s_2 = s_0 + \frac{2}{2n} = -\frac{5}{3}$, and $s_4 = s_0 + \frac{4}{2n} = -\frac{4}{3}$. The subcomplex $C_{\text{SC}(t,s)}$ lies under lines of slope $1-\frac 2 t = -2$ with intercept $3s$; these lines are shown in Figure \ref{fig:upsilon-example3} for $s_0$, $s_2$ and $s_4$. In each case the complex can be realized as the Floer complex of $(\Gamma,\bchain)$ with a nearly vertical curve weaving through the vertical marked points as shown. In each case one can check that the tower in homology is generated by the sum of the indicated intersection points, and so $V_{\text{SC}(t,s)}$ is the number of marked points covered by the shaded region. We then have that
$$\Upsilon_K\left(\frac 2 3\right ) = -2 \min\left \{ -2 + 2, -\frac 5 3 + 1, -\frac 4 3 + 1 \right\} = -2\left(-\frac{2}{3}\right) = \frac 4 3.$$

\begin{figure}
\includegraphics[scale = .65]{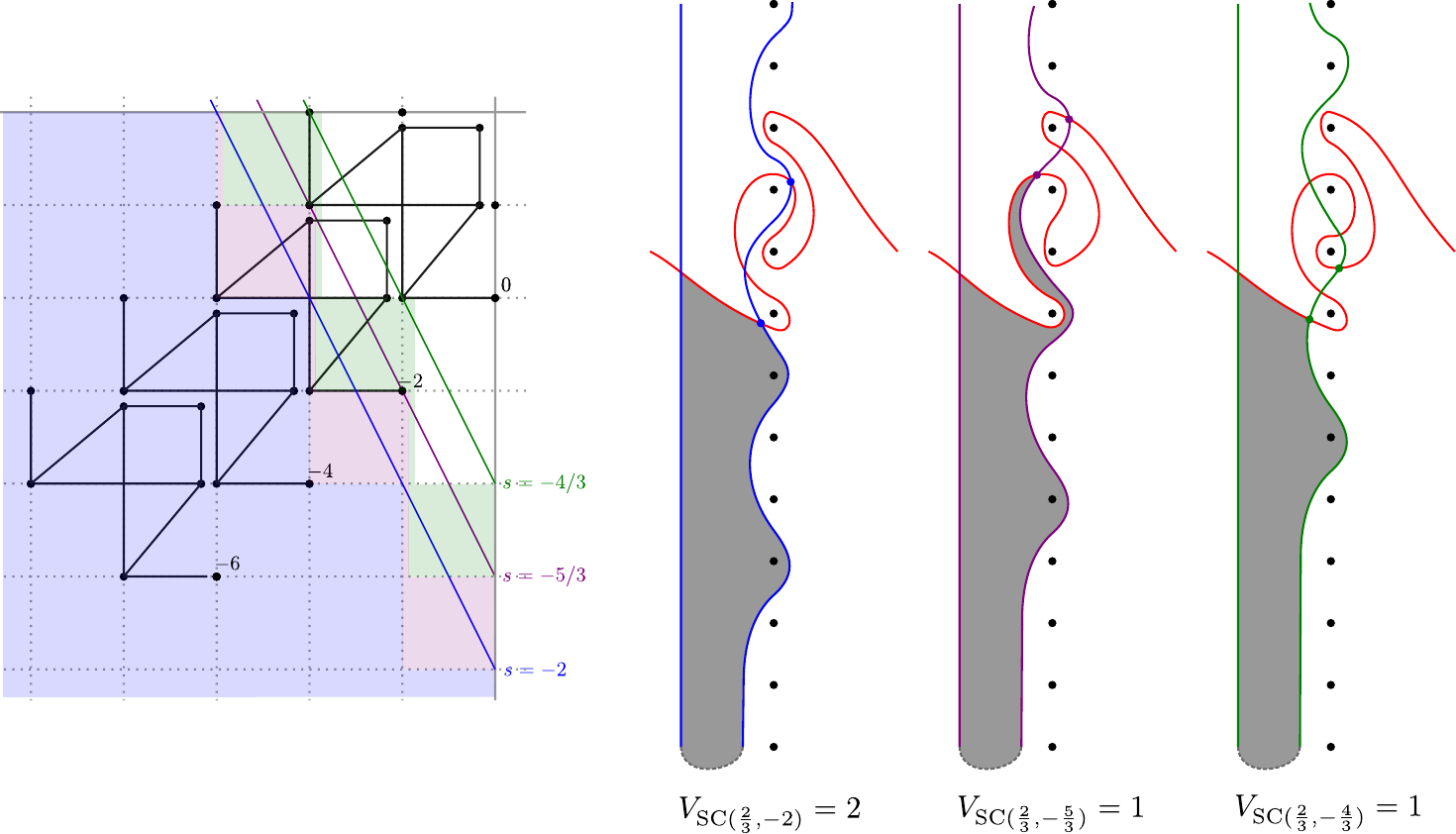}
\caption{The subcomplexes $C_{\text{SC}(t,s)}$ of $C = \CFKr(C_{2,-1}(LHT))$ corresponding to $t=\frac 2 3$ and $s = -2$, $s=-5/3$ or $s=-4/3$ are portions of $C$ below the indicated lines (equivalently in the corresponding shaded regions). In each case the curve diagram on the right can be used to compute $V_{\text{SC}(t,s)}$. Minimizing $s + V_{\text{SC}(t,s)}$ over these three values of $s$, then multiplying by $-2$, gives $\Upsilon_K(\frac 2 3) = \frac 4 3$.}
\label{fig:upsilon-example3}
\end{figure}

It is somewhat unsatisfying that we must compute multiple generalized $V_s$ invariants and take the minimum to compute $\Upsilon_C(t)$. This can be avoided if we replace the complex $C$ with the complex $C'$ over $\Lambda = \F[W^{\frac{1}{2n}}, Z^{\frac{1}{2n}} ]$ obtained by tensoring with $\Lambda$. The gradings $\gr_w$ and $\gr_z$ extend to this complex in an obvious way and take values in $\frac{1}{2n} \Z$ rather than $\Z$. As usual we restrict to the Alexander grading 0 summand, $C'_0$. We can define generalized $V_s$ invariants for this complex analogously to Definition \ref{def:generalized-Vs}, though since tensoring with $\Lambda$ introduces a finer grid to the $WZ$-plane we allow staircases to have corners at values of $\frac{1}{2n} \Z$. That is, a (possibly infinite) increasing sequence $\vec s$ of elements of $\frac{1}{2n} \Z$ defines a subcomplex $C'_{\vec s}$ of $(C')^-_0$ with an inclusion map $v_{\vec s}$, and the invariant $V_{\vec s}(C')$ is the rank of the cokernel of the induced map $v_{\vec s, *}$. The homology of the subcomplex contains a tower $\F[ \sU^{\frac{1}{2n}} ]$, and on this tower $v_{\vec s, *}$ is multiplication by $\sU^{{V_{\vec s}}/{2n}}$.

Given $0 \le t = \frac m n \le 2$ the filtration $\sF_t$ on $C_0$ extends to a filtration on $C'_0$; since every element of $C'_0$ is a power of $\sU^{\frac{1}{2n}}$ times an element of $C_0$ and multiplication by $\sU$ lowers $\sF_t$ by 1, the function defining $\sF_t$ still has values in $\frac{1}{2n} \Z$. For each $s$, let $\text{SC}_n(t,s)$ denote the sequence of elements of $\frac{1}{2n} \Z$ defining the staircase such that $(C'_0, \sF_t)_s = C'_{ \text{SC}_n(t,s) }$; this subcomplex still consists of every element on or below the line of slope $(1-\frac 2 t)$ and intercept $\frac{2s}{t}$ in the $WZ$-plane.  The advantage of tensoring with $\Lambda$ is that all sufficiently low filtration levels are related by multiplication by a power of $\sU^{\frac{1}{2n}}$, and in particular for sufficiently low $s$ we have $V_{\text{SC}_n(t,s-\frac{1}{2n})}(C') = V_{\text{SC}_n(t,s)}(C') + 1$. It follows that $s + \frac{1}{2n} V_{\text{SC}_n(t,s)}(C')$ is constant for all sufficiently low $s$, and similar reasoning to the previous discussion shows that
$$\Upsilon_C(t) = -2 \left( s + \frac{1}{2n} V_{\text{SC}_n(t,s)}(C') \right)$$
for any sufficiently low $s$ (as before, any $s < -\frac t 2 g(C) - 1$ will work).

The complex $C'$ can be represented by the same immersed multicurve $(\Gamma, \bchain)$ as the complex $C$ if we replace each marked point $u$ in $\barT_u$ with $2n$ marked points spaced out vertically along $\mu$ and interpret each marked point as having $\frac{1}{2n}$ times the usual contribution to grading calculations. The complex $C'_{\text{SC}_n(t, s)}$ arises from pairing $(\Gamma, \bchain)$ with a curve $\alpha_{\text{SC}_n(t, s)}$ crossing $\mu$ once at each height in the sequence ${\text{SC}_n(t, s)}$. Assuming we can identify an intersection point that appears with no power of $\sU^{\frac{1}{2n}}$ in a generator of the tower, we can compute $V_{\text{SC}_n(t,s)}(C')$ by counting marked points in an appropriate bigon starting at this intersection point. Figure \ref{fig:upsilon-example4} applies this construction to the example from Figure \ref{fig:upsilon-example3}, where $C = \CFKr(C_{2,-1}(LHT))$, $t = \frac{2}{3}$ and $s = -2$. The intersection point of $\alpha_{\text{SC}_n(t, s)}$ and $(\Gamma, \bchain)$ indicated in the middle diagram corresponds to a generator of the tower, and $V_{\text{SC}_n(t,s)}(C') = 8$ since the shaded bigon covers 8 marked points. It follows that $\Upsilon_C(\frac 2 3) = -2\left( -2 + \frac{8}{2\cdot 3} \right) = \frac 4 3$.

\begin{figure}
\includegraphics[scale = .75]{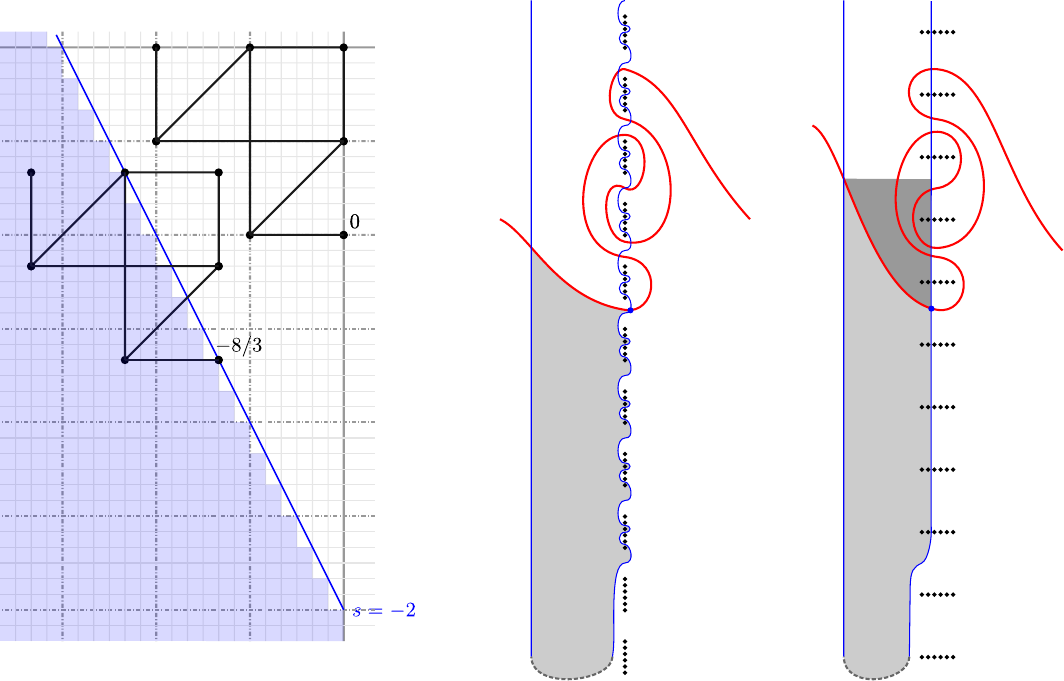}
\caption{The subcomplex $C'_{\text{SC}_n(t,s)}$, when $C = \CFKr(C_{2,-1}(LHT))$, $t=\frac 2 3$ and $s = -2$, is defined by the shaded staircase region. The highest summand containing a generator of the tower in homology has homology supported in grading $-\frac 8 3$ and is obtained from the summand of $C_0^\infty$ whose homology is supported in grading 0 by multiplying by $\sU^{\frac 8 6}$, thus $V_{\text{SC}_n(t,s)} = 8$. This is also the count of marked points contained in the shaded bigon in the middle picture. From this we see that $\Upsilon_C(\frac 2 3) = \frac 4 2$. Rearranging marked points as on the right relates this interpretation of $\Upsilon_K(t)$ to the one given earlier.}
\label{fig:upsilon-example4}
\end{figure}

We conclude by using pictures to explain (informally) why the two definitions of $\Upsilon_C$ appearing in \cite{OSS:upsilon} and \cite{Livingston:upsilon} agree. It is convenient to choose $s$ to be a multiple of $m$, so that the first intersection of $\alpha_{\text{SC}_n(t, s)}$ occurs at integer height $\frac{2s}{t} = \frac{2ns}{m}$. Because the sequence $\text{SC}_n(t, s)$ corresponds to a staircase with slope $(1 - \frac 2 t) = \frac{m-2n}{m}$, one can check that above height $\frac{2ns}{m}$ the curve $\alpha_{\text{SC}_n(t, s)}$ crosses $\mu$ in a pattern that repeats every $2n$ marked points (that is, with vertical period 1) and has exactly $m$ out of every $2n$ marked points on its left. Up to isotopy each group of $2n$ marked points between adjacent integer heights (which came from one instance of $u$ in $\barT_u$) may be arranged horizontally instead of vertically as on the right of Figure \ref{fig:upsilon-example4}. Once we have identified a suitable intersection point corresponding to a generator of the tower in $C'_{\text{SC}_n(t, s)}$, we can shade a bigon starting at that intersection point and also the triangular domain between $\Gamma$, $\mu$ and the line $y=0$, as on the right of Figure \ref{fig:upsilon-example4}. Using the last curve-based description of $\Upsilon_C$ above, which follows the definition in \cite{Livingston:upsilon}, we have that
\begin{align*}
\Upsilon_C(t) &= -2\left(s + \frac{1}{2n} (\# \text{ marked points in the lightly shaded bigon}) \right) \\
&= -\frac 1 n \left( 2ns + (\# \text{ marked points in the lightly shaded bigon}) \right).
\end{align*}
The combined number of marked points in both shaded regions is $-\frac{2ns}{m} \cdot m = -2ns$, which implies that
$$\Upsilon_C(t) = \frac{1}{n} (\# \text{ marked points in the darkly shaded bigon}).$$
This agrees with the first immersed curve description of $\Upsilon_C$ above, which follows the definition in \cite{OSS:upsilon}.

\bibliographystyle{alpha}
\bibliography{bibliography-concordance}

\end{document}